\documentclass{amsart}
\usepackage{amsmath, amsfonts, amssymb, amsthm, color, dsfont}
\usepackage{enumitem}
\usepackage[utf8]{inputenc}
\usepackage[colorlinks,linkcolor=blue,citecolor=blue]{hyperref}
\usepackage[all, knot, color]{xy}
\xyoption{arc}
\xyoption{line}
\usepackage{booktabs}
\UseCrayolaColors

\usepackage{tikz}
\usetikzlibrary{cd}

\usepackage{euscript}
\newcommand\eA           {\EuScript{A}}
\newcommand\eB           {\EuScript{B}}
\newcommand\eC           {\EuScript{C}}
\newcommand\eD           {\EuScript{D}}

\newcommand\eF          {\EuScript{F}}

\newcommand\eM          {\EuScript{M}}

\newcommand\cB           {\mathcal{B}}

\newcommand\cM          {\mathcal{M}}

\newcommand{\tR}{{\tilde{R}}}
\newcommand{\tF}{{\tilde{F}}}
\newcommand{\tI}{{\tilde{I}}}
\newcommand{\tf}{{\tilde{f}}}

\newcommand{\fg}{\mathfrak{g}}

\newcommand{\fs}{\mathfrak{s}}
\newcommand{\ft}{\mathfrak{t}}

\newcommand{\fj}{\mathfrak{j}}

\newcommand{\BB}{{\eB}}

\renewcommand{\AA}{{\eA}}
\newcommand{\CC}{{\eC}}

\newcommand{\DD}{{\eD}}

\newcommand{\MM}{{\eM}}
\newcommand{\FF}{{\eF}}

\newcommand{\Mod}[1]{{\eM_{#1}}}
\newcommand{\Modrev}[1]{{\eM^{\rev}_{#1}}}

\newcommand{\BC}{\mathbb{C}}

\newcommand{\BZ}{\mathbb{Z}}

\newcommand{\BS}{\mathbb{S}}
\newcommand{\BT}{\mathbb{T}}

\newcommand{\bX}{\mathbf{X}}
\newcommand{\bY}{\mathbf{Y}}

\newcommand{\s}{\sigma}
\renewcommand{\b}{\beta}

\newcommand{\KB}{{K(\BB)}}
\newcommand{\KA}{{K(\CA)}}
\newcommand{\Ki}{{K(\BB_1)}}
\newcommand{\Kii}{{K(\BB_2)}}

\newcommand{\ld}{\lambda}

\newcommand{\e}{\varepsilon}
\newcommand{\vp}{\varphi}

\newcommand{\ev}{\operatorname{ev}}
\newcommand{\coev}{\operatorname{coev}}

\newcommand{\SL}{\operatorname{SL}}

\newcommand{\id}{\operatorname{id}}
\newcommand{\Hom}{\operatorname{Hom}}
\newcommand{\uH}{\underline{\operatorname{Hom}}}

\newcommand{\irr}{\operatorname{Irr}}
\newcommand{\FPdim}{\operatorname{FPdim}}
\newcommand{\rev}{\operatorname{rev}}
\newcommand{\tr}{\operatorname{tr}}
\newcommand{\Tr}{\operatorname{Tr}}

\newcommand{\Rep}{\operatorname{Rep}}

\renewcommand{\1}{{\mathds{1}}}

\renewcommand{\Vec}{\operatorname{Vec}}

\newtheorem{thm}{Theorem}[section]
\newtheorem{lem}[thm]{Lemma}
\newtheorem{prop}[thm]{Proposition}
\newtheorem{cor}[thm]{Corollary}
\newtheorem{rmk}[thm]{Remark}
\newtheorem{question}[thm]{Question}

\theoremstyle{definition}
\newtheorem{defn}[thm]{Definition}

\def\O2{{\Omega^n_2}}

\def\a{{\alpha}}
\def\g{{\gamma}}
\def\ta{{\tilde{\alpha}}}
\def\o{{\otimes}}
\def\oa{{\,\otimes\!_A\,}}
\def\ol{\overline}
\newcommand{\ul}[1]{{\underline{#1}}}

\def\SLZ{\SL_2(\BZ)}
\def\w{\omega}
\def\CAo{{\CC^0_{\!A}}}
\def\CA{{\CC_A}}
\def\CB{{\CC_B}}
\def\CBo{{\CC^0_{\!B}}}
\def\LCA{{\mathcal{L}(\CA)}}

\def\LaA{{\mathcal{L}_{\operatorname{alg}}(A)}}
\def\iC{\irr(\CC)}
\def\iCA{{\irr(\CC_A)}}

\def\iCBo{{\irr(\CBo)}}
\def\iA{{\irr(\AA)}}
\def\iB{{\irr(\BB)}}
\def\Aut{{\rm Aut}}
\def\Auta{{\operatorname{Aut}_{\operatorname{alg}}}}
\newcommand{\jacobi}[2]{{\genfrac(){}{0}{#1}{#2}}}

\newtheorem*{CPM}{Theorem I}
\newtheorem*{CPMII}{Theorem II}
\newtheorem*{CPMIII}{Theorem III}

\def\vec{{\operatorname{Vec}}}
\newcommand{\mtx}[1]{{\begin{bmatrix} #1 \end{bmatrix}}}

\usepackage[normalem]{ulem}
\usepackage{relsize, stackengine, graphicx}
\stackMath

\newenvironment{enumeri}
{\begin{enumerate}[label=\rm{(\roman*)}, leftmargin=18pt, labelsep=3pt]}
{\end{enumerate}}
\title{Generalized Symmetries From Fusion Actions}
\author{Chongying Dong}
\address{Department of Mathematics, University of California, Santa Cruz, CA 95064}
\email{dong@ucsc.edu}
\author{Siu-Hung Ng}
\address{Department of Mathematics, Louisiana State University, Baton Rouge, LA 70817}
\email{rng@math.lsu.edu}
\author{Li Ren}
\address{Department of Mathematics, Sichuan University, Chengdu, 610064, China}
\email{RenL@scu.edu.cn}
\author{Feng Xu}
\address{Department of Mathematics, University of California, Riverside, CA, 92521}
\email{xufeng@math.ucr.edu }

\usepackage{natbib}
\usepackage{graphicx}
\begin{document}

\begin{abstract} Let $A$ be a condensable algebra in a modular tensor category $\EuScript{C}$. We define an action of the fusion category $\EuScript{C}_A$  of $A$-modules in $\EuScript{C}$ on the morphism space $\mbox{Hom}_{\EuScript{C}}(x,A)$ for any $x$ in $\EuScript{C}$, whose characters are generalized Frobenius-Schur indicators. This fusion action can be considered on $A$, and we prove a categorical generalization of the Schur-Weyl duality for this action. For any fusion subcategory $\EuScript{B}$ of $\EuScript{C}_A$ containing all the local $A$-modules, we prove the invariant subobject $B=A^\EuScript{B}$ is a condensable subalgebra of $A$. The assignment of $\EuScript{B}$ to $A^\EuScript{B}$ defines a Galois correspondence between this kind of fusion subcategories of $\EuScript{C}_A$ and the condensable subalgebras of $A$. In the context of VOAs, we prove for any nice VOAs $U \subset A$, $U=A^{\EuScript{C}_A}$ where $\EuScript{C}=\EuScript{M}_U$ is the category of $U$-modules. In particular, if $U = A^G$ for some finite automorphism group $G$ of $A,$ the fusion action of $\EuScript{C}_A$ on $A$ is equivalent to the $G$-action on $A.$ 
\end{abstract}

\maketitle

\section{Introduction}
Motivated by orbifold theory, we explore the notion of fusion actions on condensible algebras within a modular tensor category in this paper. 

Orbifold theory studies the action of a group $G$ on a vertex operator algebra $A$. We highlight three key results from orbifold theory that are relevant to this paper. The first is the Schur-Weyl duality \cite{DLMcomp1996, KT1997}.
Assume that $A$ is a simple VOA and $G$ is a compact Lie group acting on $A$ continuously. Then $A$ 
decomposes into 
$$A=\bigoplus_{W\in \irr(G)}W\otimes A_W$$
where $W$ is a finite dimensional irreducible module and $A_W=\Hom_G(W,A)$ is the multiplicity space.
Then $A^G=\Hom_G(\BC,A)$ is a simple vertex operator subalgebra of $A$ and each $A_W$ is a (nonzero) irreducible 
$A^G$-module. Moreover $A_{W_1}$ and $A_{W_2}$ are isomorphic $A^G$-modules iff $W_1$ and $W_2$ are isomorphic $G$-modules.
The second is that if $G$ is a finite group, then $H\mapsto A^H$ gives a one to one correspondence between subgroups of $G$ and 
subalgebras of $A$ that contain $A^G$ \cite{DM1997, HMT1999}. Finally, if we further assume that $A^G$ is rational, $C_2$-cofinite and the weight of any twisted module is positive except $A$ itself, then the categorical dimension $\dim_{A^G}A^H=o(G)/o(H)$ for any subgroup $H$ \cite{DJX2013}. These assumptions on $A^G$ ensure  that $\dim_{A^G}A^H$ is well defined and computable. 

It is well known that not every subalgebra $U$ of $A$ can be realized as $A^G$ for some automorphism group $G.$ A natural question arises: can $U$ be realized as the fixed-point subalgebra under the action of something analogous to a group? To answer this question, we have to assume that $U$ is rational and $C_2$-cofinite. In this case the $U$-module category $\Mod{U}$ is a modular tensor category (MTC) \cite{huang2005vertex} and $A$ is a condensable algebra in $\Mod{U}$ \cite{HKL}. We provide a positive answer to this question by employing the categorical Schur-Weyl duality, which is proved in this paper.  Our approach to Galois correspondence is partly motivated by  \cite{Xu2014} on  conformal nets. 

We now state our results in details. Let $\CC$ be a modular tensor category (MTC) and $A$ a condensable algebra in $\CC.$ Then
$$A=\bigoplus_{x\in\irr(\CC)}W_x\otimes x$$
where $W_x=\Hom_{\CC}(x,A).$ It follows from \cite{KO02} that the category $\CC_A$ of right $A$-modules in $\CC$ is a spherical fusion category, and its subcategory $\CAo$ of local $A$-modules is an MTC.  The fusion algebra $K(\CC_A)$ of $\CA$ over $\BC$ is well-known  to be a semisimple associative algebra. Set 
$$e_\1=\frac{1}{\dim(\CAo)}\sum_{X\in\irr(\CAo)}d_A(X)X.$$
Then $e_\1$ is a primitive central idempotent of $K(\CC_A)$
and $(1-e_\1)K(\CC_A)$ is a semisimple ideal of $K(\CC_A).$ 
We define an action of the fusion category $\CC_A$ on   $W_x$ for any $x \in \CC$ in this paper, and we write $X f$ for the action of $X \in \CA$ on $f \in W_x$.  The $\CC_A$-invariant of $W_x$ is defined as 
$$W_x^{\CC_A}=\{f\in W_x\mid Xf=d_A(X)f \ {\rm for }\ X\in \CC_A\}.$$
The subobject $A^{\CC_A}:=\sum_{x}W_x^{\CC_A}x$ of $A$ in $\CC$ will be shown to inherit the algebra structure of $A$ in Section  \ref{s:Gal_Cor}.

Our first main result is the following Schur-Weyl duality.
\begin{CPM}\label{t:I} Let $A$ be a condensable algebra in any modular tensor category $\CC$. Then
    \begin{enumeri}
\item  The kernel of the action of $\CA$ on $W_A$ is equal to $(1-e_\1)K(\CC_A).$
\item  $W_x$ is an irreducible $K(\CA)$-module for any $x \in \irr(\CC)$ whenever $W_x \ne 0$.
\item  For any $x, y\in \iC$, $W_x \cong W_y$ as $K(\CA)$-modules if and only if $x=y$.
\item  The $\CA$-action defines an isomorphism of $\BC$-algebras $e_\1 K(\CC_A) \to \Hom_\CC(A,A)$.
\end{enumeri}
\end{CPM}

The Schur-Weyl duality for the action of a finite group $G$  on a condensable algebra $A$ in a modular tensor category $\CC$ with the assumption that $A^G=\1$ was also obtained in \cite[Thm. 2.11]{K02}. 

Although the duality result for the fusion category action on $A$ resembles the duality result in orbifold theory, 
 the proof in \cite{DLMcomp1996} for orbifold theory setting  cannot be extended to fusion action on a VOA. In the case of orbifold theory of a finite automorphism group $G$ of the VOA $A$, $A$ is a semisimple $G$-module and the key step involves analyzing the multiplicity 
spaces $A_W$ for irreducible $G$-modules $W$. In contrast, in our categorical setting, $A$ is a condensable algebra in an MTC  $\CC$. The relevant multiplicity spaces are $W_x$ of $x\in \iC.$ The generalized Frobenius-Schur
indicators \cite{NS10} play a crucial role in defining the $\CA$-action on $W_x$ and in the proof of Theorem I.

Note that the algebra  $e_\1 K(\CC_A)=e_\1 K(\CC_A) e_\1$ can also be interpreted in terms of hypergroup \cite{Bi2017}. In fact, $e_\1 K(\CC_A)$ is the quotient of hypergroup $K(\CC_A)$ by $K(\CAo).$

We also show that the action of a finite group $G$ on a VOA $A$ is a special case of a fusion category action, under the assumption that $A^G$ is rational and $C_2$-cofinite. In this case $\CC=\Mod{A^G}$ is the $A^G$-module category and $e_\1 K(\CA)$ is isomorphic to $\BC[G].$ Moreover, the action of $\CC_A$ recovers the action of $G$ on $A.$

As in the orbifold theory, we then use the duality result to give a Galois correspondence from the fusion action. Let $\LCA$ be the set of fusion categories of $\CC_A$ containing $\CAo$ and
$\LaA$ be the set of condensable subalgebras of $A.$ 
\begin{CPMII}\label{t:II} Let $\CC$ be a pseudounitary modular tensor category, and $A$ a condensable subalgebra in $\CC$. Then the assignment $\LCA \to \LaA$, $\BB \mapsto A^\BB$, defines an order-reversing bijection, whose inverse is given by assignment $\LaA \to \LCA$,  $(B, \iota) \mapsto (\CBo)_A$, where $(B, \iota)$ is defined in Section 5. In particular, $\BB = (\CC_{A^\BB}^0)_A$ for any $\BB \in \LCA$ and 
    $$
    \dim(\BB)= \frac{\dim(\CC)}{d(A)d(B)}\,.
    $$
    Moreover,  $A^{\CAo}=A$ and $A^{\CC_A}=\1$ which is  the unit object of $\CC.$
\end{CPMII}

A hypergroup action on a condensable algebra was previously introduced by Riesen in \cite{Rie2025}, and some results of \cite{Xu2013, Xu2014, Bi2017} on conformal nets are extended to vertex operator algebras. In particular, he proved that hypergroup action preserves duality. Moreover, if $A$ is an extension of a rational VOA $V$,  there is an hypergroup $K$ acting on $A$ such that $A^K=V.$ In contrast, we first introduce the fusion action of $\CC_A$ on the multiplicity spaces $W_x$ for any object $x$. We prove in Proposition \ref{p:diff_form}
that if $x=A$, our fusion action on $\id_A$ in $\CC(A,A)$ is exactly the hypergroup action on the algebra $A$ defined in \cite{Rie2025}. The fusion action on the multiplicity spaces 
is crucial to establishing our Schur-Weyl duality (Theorem I).

We would like to point out that such a categorical correspondence was first established in \cite[Thm. 4.10, Rem. 4.12]{DMNO} via the theory of relative center. In \cite[Thm. 4.10]{DMNO}, the subalgebra $B$ of $A$  corresponding a fusion subcategory $\BB$ of $\CA$ is given by $B=I(\1)$, where $I$ is the right adjoint of the forgetful functor $F_\BB:\CC \to Z_\BB(\CA)$, the relative center of $\BB$ in $\CA.$ In contrast, our approach follows the classical notion of Galois correspondence using \emph{invariance subspaces} of fusion actions and relies crucially on the Schur-Weyl duality established in Theorem~I. Connections between these two approaches will be discussed in details in Section \ref{s:relative}. 

Now, we elaborate the preceding categorical results in the context of vertex operator algebras. Let $U$ be a simple, rational, $C_2$-cofinite VOA of CFT type such that the weight of any irreducible module is positive except $U$ itself. Then the $U$-module category $\Mod{U}$ is a pseudounitary modular tensor category. Let $A$ be simple VOA which contains $U$ as a subVOA. Then $A\in \Mod{U}$ is a condensable algebra which has a decomposition
$$A=\bigoplus_{x\in\irr(\Mod{U})}W_x\otimes x$$
as before. Applying our results to $\CC=\Mod{U}$ and 
$A$, we have the following theorem.

\begin{CPMIII}\label{t:III} Using the fusion action of $\CC_A$ on $A$, we have 
 \begin{enumeri}
\item  $W_x$ is an irreducible $K(\CC_A)$-module for any $x \in \irr(\CC)$ whenever $W_x \ne 0$.
\item  For any $x, y\in \iC$, $W_x \cong W_y$ as $K(\CC_A)$-module if and only if $x=y$.
\item  The restriction  $e_\1 K(\CC_A) \to \Hom_U(A,A)$ is an isomorphism of algebras.
\item  $\BB \mapsto A^\BB$  gives
a one-to-one correspondence between the fusion subcategories of $\CA$ containing $\CAo$ and subVOAs of $A$ containing $U.$ In particular, $A^{\CAo}=A$ and 
$A^{\CA}=U.$
\end{enumeri}
 \end{CPMIII} 

{\bf Acknowledgment:} We would like to thank Terry Gannon for some discussions  related to this paper and Victor Ostrik for pointing out Shimizu's work \cite{Shi} relevant to part of Corollary \ref{c1}. S.-H. N. was partially supported by the Simons Foundation MPS-TSM-00008039.

\section{Basic terminologies on fusion categories}
In this section, some basic terminologies, conventions and results on fusion categories are provided for the rest this paper. The readers are referred to \cite{EGNO} for more details.

 \subsection{Conventions and terminologies} Let $\cM$ be a semisimple $\BC$-linear category. We denote by $\irr(\cM)$ a complete set of nonisomorphic simple objects of  $\cM$. Following the conventions in \cite{MacLane}, the morphism space of $\cM$ from $x$ to $y$ is simply abbreviated as $\cM(x,y)$, instead of $\Hom_\cM(x,y)$, and $[x,y]_\cM := \dim_\BC \cM(x,y)$. 
 
 Throughout this paper, a fusion category $\cM$ is a semisimple $\BC$-linear rigid finite tensor category with finite-dimensional hom-spaces and $\irr(\cM)$ is a finite set containing the unit object $\1$. The left dual object of any object $x \in \cM$ is denoted by $x^*$. The duality of a fusion category $\cM$ can be extended to a contravariant tensor functor $(-)^*$ from $\cM$ to its reversed tensor category $\cM^{\rev}$. A fusion subcategory of $\cM$ always means a full subcategory which is closed under the tensor product of $\cM$. 

 The Grothendieck ring $K_0(\cM)$ of a fusion category $\cM$ is a $\BZ_+$-based ring. In particular,  $K_0(\cM)$ an algebra over $\BZ$ with the canonical basis  $\irr(\cM)$. In this paper, we often needs to deal with the $\BC$-algebra $K(\cM):=K_0(\cM)\o_\BZ \BC$. It is well-known that $K(\cM)$ is a semisimple $\BC$-algebra.  

 A fusion category $\cM$ equipped with an isomorphism  $j: \id_\cM \to (-)^{**}$ of tensor functors, called a pivotal structure of $\cM$, is  a pivotal fusion category. One can defines a categorical trace $\tr(f)$ for any endomorphism $f: x \to x$ in a pivotal fusion category. The pivotal structure of a pivotal fusion category is called spherical if $\tr(f)=\tr(f^*)$ for all  endomorphisms $f$ in $\cM$. In this case,  the pivotal fusion category will be called a spherical fusion category. For any object $x$ of a spherical fusion category, we denote by $d^\cM\!(x):=\tr(\id_x)$, and is called the (pivotal) dimension of $x$. We often drop the superscript when there is not any ambiguity. 
 
 A pivotal tensor category $\cM$ is called \emph{strict}  if $\cM$ is a strict tensor category and $(-)^{**}=\id_\CC$. Since every pivotal tensor category is equivalent to a strict one (cf. \cite[Thm. 2.2]{NSind}), we may always assume a pivotal fusion category to be strict, and graphically calculus can be applied in this case.

   For any objects $x,y$ of a spherical fusion category $\cM$, there is a canonical nondegenerate bilinear form $\langle \cdot, \cdot \rangle: \cM(x, y) \o \cM(y, x) \to \BC$ given by $\langle f, g \rangle = \tr(f g)$ for $f \in \cM(x, y)$ and $g \in \cM(y, x)$.
 
 Similarly, a fusion category equipped with a \emph{braiding} is  called a braided fusion category. A \emph{premodular} category $\cM$ is a spherical braided fusion category , and it is often called a \emph{ribbon category}. The underlying ribbon structure $\theta^{\cM}_x: x \to x$ is a natural isomorphism of the identity functor $\id_\cM$ and it can be constructed from the Drinfeld isomorphism and the spherical pivotal structure of $\cM$. 

 Let $\cM$ be a premodular category equipped with the braiding $R$. The $S$-matrix of $\cM$ is the complex square matrix indexed by $\irr(\cM)$ and defined by
 \begin{equation} \label{eq:S-matrix}
     S^\cM_{x,y}  :=\tr(R_{y, x^*} R_{x^*,y}) \quad\text{ for } x, y \in \irr(\cM)\,.
 \end{equation}
The premodular category $\cM$ is said to be a \emph{modular tensor category} (MTC) if its $S$-matrix is invertible.   For $x \in \irr(\cM)$, the component $\theta^\cM_x$ of the ribbon structure $\theta^\cM$ of $\cM$ is a scalar multiple of the identity. In this case, we simply identify $\theta_x^\cM$ as the scalar and the $T$-matrix of $\cM$ is defined as $T^\cM_{x,y}:= \delta_{x,y} \theta_x^\cM$.

A \emph{pseudounitry} fusion category admits a canonical spherical pivotal structure so that the dimension of  any nonzero object is positive (cf. \cite{ENO}). Since there could be choices of pivotal structures on  an MTC (cf. \cite[Lem. 2.4]{RankFinite}),  we will assume the canonical spherical structure for any pseudounitary MTC or fusion category in this paper. 

Let $\CC$ be a modular tensor category with the underlying spherical fusion category being strict pivotal.  An object $A$ of $\CC$ is said to be an algebra if there are morphisms $u_A : \1 \to A$ and $m_A: A \o A \to A$ such that the following diagrams commute:
$$
\xymatrix{
A  \o \1 \ar[r]^-{A \o u_A} & A \o A \ar[d]^-{m_A} & \1\o A \ar[l]_-{u_A \o A}\\
&A\ar@{=}[lu] \ar@{=}[ru]&
}
\quad, \quad
\xymatrix{
A \o A \o A \ar[d]_-{A \o m_A} \ar[r]^-{m_A \o A} & A \o A \ar[d]^-{m_A}\\
A \o A \ar[r]^-{m_A} & A
}\,.
$$
The  algebra $A$ in $\CC$ is said to be \emph{connected} if $[\1, A]_\CC=1$, and \emph{commutative} if $m_A R_{A,A} = m_A$. We borrow the terminology in topological order that an algebra $A$ is said to be \emph{condensable} if $A$ is commutative and connected such that $d(A) \ne 0$, $\theta_A = \id_A$ and the composition $\e m_A: A \o A \to \1$ is nondegenerate (cf. \cite{KO02}), where $\e \in \CC(A, \1)$ denotes the section of $u_A$. These conditions imply $\CA$ is separable by \cite[Thm. 3.3]{KO02}. 

According to \cite{KO02}, if $A$ is a condensable algebra in an MTC $\CC$, the category  $\CA$ of right $A$-modules in $\CC$ is also a spherical fusion category with $A$ being the unit object, and the tensor product $\oa$ over $A$.  The pivotal dimension of any object $X \in \CA$ will be denoted by $d_A(X)$, which is equal to $\frac{d(X)}{d(A)}$, where $d$ is the dimension function of $\CC.$

\subsection{\texorpdfstring{$A$}{A}-modules and their graphical conventions}
Let $\CC$ be an MTC and $A\in \CC$ a condensable algebra. Without loss of generality, we may assume $\CC$ is strict pivotal. The induction function $\a : \CC \to \CC_A$ defined by $\a(x) = x \o A$ for $x \in \CC$ and the forgetful functor $\CC_A \to \CC$ are 2-sided adjoints of each other \cite[Lem. 1.16]{KO02}. In particular, we have the natural isomorphism of $\BC$-linear spaces
$$
   \CC_A(\a(x), Y) \cong \CC(x, Y) 
$$
and hence $[\a(x), Y]_{\CC_A} = [x, Y]_{\CC}$ for any objects $x \in \CC$ and $Y \in \CC_A$. Since $A$ is the unit object of $\CC_A$, we will write $A$ as $\1$ in $\CC_A$ when the context is clear, and denote by $\vp_x$ the natural isomorphism of vector spaces
\begin{equation}\label{eq:adj}
\vp_x: \CC_A(\a(x),\1) \xrightarrow{\sim}  \CC(x, A)\,,\quad f \mapsto f \circ (x\o u_A)  
\end{equation}
for $f \in \CC_A(\a(x),\1)$, where $u_A: \1 \to A$ is the unit map of the algebra $A$.  We use the following diagrams for the unit map and multiplication of $A$:
$$
 u_A:= \,\vcenter{\xy  0;/r1pc/:
{\ar@{.} (-1,1.1)*{\circ};(-1,-1.3)},
\endxy} \qquad\text{and}\qquad
m_A:= \vcenter{\xy  0;/r1pc/:
{\ar@{.} (-2,1.5);(-1,0)},
{\ar@{.} (0,1.5);(-1,0)},
{\ar@{.} (-1,0);(-1,-1.5)},
\endxy} \quad,
$$
where the dotted line will always denote the object $A$. Thus, in terms of diagram,
\begin{equation}\label{eq:vp}
\vp_x(f)\,=\,
 \vcenter{\xy  0;/r1pc/:
 (0,0)*=<20pt,10pt>[F]{\scriptstyle f},
{\ar@{-} (-.5,2)*+{\scriptstyle x};(-.5,.4)},
{\ar@{.} (.5,2)*{\circ};(.5,.4)},
{\ar@{.} (0,-.4);(0,-2)},
\endxy} \quad \text{ for } f \in \CA(\a(x), \1)\,.
\end{equation}
The inverse of $\vp_x$, denoted by $\psi_x$, is given by the diagram:
\begin{equation}\label{eq:psi}
   \psi_x(g)\,=\,
 \vcenter{\xy  0;/r1pc/:
 (0,0)*=<10pt,10pt>[F]{\scriptstyle g},
{\ar@{-} (0,2)*+{\scriptstyle x};(0,.4)},
{\ar@{.}@`{(1,1),(.8, -1.2)} (1,2);(0,-1.5)},
{\ar@{.} (0,-.4);(0,-2)},
\endxy} \quad \text{ for } g \in \CC(x, A)\,. 
\end{equation}

For $Y \in \CC_A$, we will denote by $\mu_Y: Y \o A \to Y$  its underlying $A$-module structure morphism. The right $A$-module $Y$ naturally admits two left $A$-module structures via the braiding $R$ of $\CC$. We abbreviate by $Y^+$ the $A$-bimodule with the left $A$-module structure $\mu_Y^+$ of $Y$ defined by
$$
A \o Y \xrightarrow{R_{Y,A}^{-1} } Y \o A \xrightarrow{\mu_Y} Y\,.
$$
It is immediate to check that $\CC_A(X,Y) = {}_A\CC_A(X^+, Y^+)$ and so $\CC_A$ can be considered as a subcategory of ${}_A\CC_A$ via the functor $Y \mapsto Y^+$. 
Similarly, $Y^-$ is the $A$-bimodule with the left $A$-module structure morphism $\mu_Y^-$ of $Y$ defined as
$$
A \o Y \xrightarrow{R_{A,Y}} Y \o A \xrightarrow{\mu_Y} Y\,.
$$
Throughout this paper, unless stated otherwise, every object $Y$ of $\CC_A$ will be considered as the $A$-bimodule $Y^+$. 

The braiding $R_{x,y}:x \o y \to y\o x$ for any $x,y \in \CC$, and its inverse are depicted as
$$
R_{x,y}=\vcenter{\xy 0;/r1pc/:
\vtwistneg~{(-.5,1)*+{\scriptstyle x}}{(.5,1)*+{\scriptstyle y}}{(-.5, -1)}{(.5,-1)},
\endxy}, \quad 
R_{x,y}^{-1}=\vcenter{\xy 0;/r1pc/:
\vtwist~{(-.5,1)*+{\scriptstyle y}}{(.5,1)*+{\scriptstyle x}}{(-.5, -1)}{(.5,-1)},
\endxy}\,.
$$
We will use the following diagram for the left and right action of $A$ for any object $Y$ of $\CC_A$:
$$
\mu_Y=\vcenter{\xy 0;/r1pc/:
(0,0); (0,1.5)*+{\scriptstyle{Y}} **\dir{-},
(0,0); (0,-1.5)*+{\scriptstyle{Y}} **\dir{-},
(0,0); (1.5,1.5)*+{\scriptstyle{A}} **\dir{..},
\endxy} \quad\text{and }\quad
\mu^+_Y=\vcenter{\xy 0;/r1pc/:
{\ar@{-} (0,1.5)*{\scriptstyle Y};(0,-1.5)*+{\scriptstyle Y}},
{\ar@{.}@`{ (1.8, -.1)}|(.3)\hole(-1.5,1.5)*+{\scriptstyle A};(0,-.8)},
\endxy} 
$$

The tensor product $X \oa Y$  for any $X,Y \in \CC_A$ is defined as the cokernel 
$X \o Y \xrightarrow{p_A} X \oa Y$ of 
$
X \o A \o Y \xrightarrow{\mu_X \o Y - X \o \mu^+_Y} X \o Y\,.
$
In particular, there are unique natural isomorphisms of right $A$-modules  $l_Y: A \oa Y \to Y$ and $r_Y: Y \oa A \to Y$  which satisfy the commutative diagrams:
$$
\xymatrix{
A \o Y \ar[d]_-{p_A} \ar[r]^{\mu_Y^+} & Y \\
A \oa Y \ar[ru]_{l_Y}
}, \qquad
\xymatrix{
 Y \o A \ar[d]_-{p_A} \ar[r]^{\mu_Y} & Y \\
Y \oa A \ar[ru]_{r_Y} 
}\,. 
$$

For any object $y \in \CC$, the coevaluation $\coev_y: \1 \to y \o y^*$ and the evaluation $\ev_y: y^* \o y \to \1$ morphisms associated with the left dual $y^*$ of $y$ are respectively depicted by the diagrams  
$$
\vcenter{\xy 0;/r1pc/:
(0,2.1)\vcap[2],
(2.2,1.8)*{\scriptstyle y^*},
(0,1.8)*{\scriptstyle y},
(5,2)*{,},
(8,2.7)\vcap[-2],
(8,3.3)*{\scriptstyle y^*},
(10,3.3)*{\scriptstyle y},
(12,2)*{.},
\endxy}
$$
Since $\CC$ is assumed to be strict pivotal, the evaluation and coevaluation morphisms for right dual of $y$ is given by $\coev'_y =\coev_{y^*}$ and $\ev'_y= \ev_{y^*}$, and their diagrams are respectively 
$$
\vcenter{\xy 0;/r1pc/:
(0,2.1)\vcap[2],
(2.2,1.8)*{\scriptstyle y},
(0,1.8)*{\scriptstyle y^*},
(5,2)*{,},
(8,2.7)\vcap[-2],
(8,3.3)*{\scriptstyle y},
(10,3.3)*{\scriptstyle y^*},
(12,2)*{.},
\endxy}
$$

Following \cite{KO02}, $\CC_A$ is spherical and the left dual object of $Y \in \CC_A$ is the left dual object $Y^*$ in $\CC$ with the right $A$-action given by
$$
\mu_{Y^*} = \vcenter{\xy 0;/r1pc/:
(0,0); (0,1.5)*+{\scriptstyle{Y^*}} **\dir{-},
(0,0); (0,-1.5)*+{\scriptstyle{Y^*}} **\dir{-},
(0,0); (1.5,1.5)*+{\scriptstyle{A}} **\dir{..},
\endxy} :=
\vcenter{\xy 0;/r1pc/:
(0,-.5)\vcap[-1],
(1,.5)\vcap[1],
(0,-.5); (0,2)*+{\scriptstyle{Y^*}} **\dir{-},
(2,.5); (2,-2)*+{\scriptstyle{Y^*}} **\dir{-},
(1,.5); (1,-.5)**\dir{-},
{\ar@{.}@`{ (2, -.2)}|(.5)\hole  (1,0);(3,1.5)*+{\scriptstyle A}},
\endxy} .
$$
The coevaluation $\ol \coev_Y : A \to Y \oa Y^*$ and the evaluation $\ol \ev_Y: Y^* \oa Y \to A$ for the left dual of $Y$ in $\CA$, and the coevaluation $\ol \coev'_Y : A \to Y^* \oa Y$ and the evaluation $\ol \ev'_Y: Y \oa Y^* \to A$ for the right dual of $Y$ in $\CA$ are respectively depicted as morphisms in $\CC_A$ by
$$
\def\objectstyle{\scriptstyle}
{\color{blue}\xy (0,-1)="ctext",,
{"ctext"+(-4,0)="v1"*+{\color{black} Y}; "ctext"+(4,0)="v2"*+{\color{black} Y^*} **\crv{"v1"+(0,8)&"v2"+(0,8)}}\endxy}, \quad
{\color{blue}\xy (0,-1)="ctext",,
{"ctext"+(-4,4)="v1"*{\color{black} Y^*}; "ctext"+(4,4)="v2"*{\color{black} Y} **\crv{"v1"+(0,-7)&"v2"+(0,-7)}}\endxy},\quad
{\color{blue}\xy (0,-1)="ctext",,
{"ctext"+(-4,0)="v1"*+{\color{black} Y^*}; "ctext"+(4,0)="v2"*+{\color{black} Y} **\crv{"v1"+(0,8)&"v2"+(0,8)}}\endxy},\quad\text{and}\quad
{\color{blue}\xy (0,-1)="ctext",,
{"ctext"+(-4,4)="v1"*+{\color{black} Y}; "ctext"+(4,4)="v2"*{\color{black} Y^*} **\crv{"v1"+(0,-7)&"v2"+(0,-7)}}\endxy}.
$$

For any $x \in \CC$, $Y \in \CC_A$, both $x \o Y$ and $Y \o x$ are naturally right $A$-modules with the structure morphisms  
$$
\mu_{x \o Y} = x\o\mu_Y \quad \text{ and } \quad \mu_{Y \o x} = (\mu_Y \o x)(Y \o R_{x,A}),
$$
and $R_{x,Y}: x \o Y \to Y \o x$ is a right $A$-module isomorphism. In particular, $x \o Y \cong Y \o x$ in $\CC_A$. Moreover,   $\a(x)$ admits a natural half-braiding $\s_{\a(x), -}$ on $\CC_A$ given by
$$
\s_{\a(x), Y} :=\left( \a(x) \oa Y  \xrightarrow{x \o l_Y} x \o Y \xrightarrow{R_{x, Y}} Y \o x \xrightarrow{r_Y^{-1} \o x} Y \oa \a(x)  \right)\,. 
$$
Thus, $(\a(x), \s_{\a(x), -})$ is an object in the center $Z(\CC_A)$, and we will denote  this object in $Z(\CC_A)$ by $\a^+(x)$ as it depends on the ``+" choice of the embedding of $\CC_A$ into ${_A}\CC_A$. Moreover, $\a^+: \CC \to Z(\CC_A)$ defines a fully faithful braided tensor functor (cf. \cite[Cor. 3.28]{DMNO}). Since $d_A(\a(x)) = d(x)$ for $x \in \CC$, $\a^+$ also preserves pivotal structures. Thus, $\a^+(\CC)$ is a modular subcategory of $Z(\CC_A)$.   Hence, $\CC$ can be considered as modular subcategory of $Z(\CC_A)$ via $\a^+$. The half-braiding $\s_{\a(x), Y}$ will be simply depicted as
$$
\s_{\a(x), Y} = \vcenter{\color{blue} \xy 0;/r.4pc/:
{\ar@{-}@`{(0,-2.5), (3,-3)} (0,0)*+{\color{black}\scriptstyle \a(x)}; (3,-5)},
{\ar@{-}@`{(3,-2.5),(0,-3)}|(0.5){\hole} (3,0)*+{\color{black}\scriptstyle Y}; (0,-5)},
\endxy}\,.
$$
Note that  the half-braiding $\s_{\a(x), Y}$ is only natural in $Y$ and the blue lines  or morphisms indicate they are considered in $\CC_A$. Moreover,  $\a^+(\1) = (\1, \s_{\a(\1), -})$ in $Z(\CA)$, where  $\s_{\a(\1), Y}=\id_Y$ for $Y \in \CC_A$.

Recall that the category $\CAo$  of local $A$-modules, which are those $Y \in \CA$  satisfying the condition $\mu^+_Y =  \mu_Y^-$, is a fusion subcategory of $\CA$. Moreover, the braiding of $\CC$ descends to nondegenerate braiding on $\CAo$ which makes $\CAo$ a modular tensor category (\cite[Thm. 4.5]{KO02}). For $X \in \CAo$,  $X$ also admits a natural half-braiding $\s_{X,-}$ on $\CC_A$ defined as follows: One can check directly 
$$
0 = \Big(X \o A \o Y \xrightarrow{\mu_X \o Y - X \o \mu^+_Y} X \o Y \xrightarrow{R_{Y,X}^{-1}} Y \o X \xrightarrow{p_A} Y \oa X\Big)\,.
$$
There exists a unique isomorphism $\s_{X,Y}: X \oa Y \to Y \oa X$ of $A$-modules such that the diagram 
\begin{equation} \label{eq:local_module_briading}
 \xymatrix{
X \o Y \ar[r]^{R_{Y,X}^{-1}} \ar[d]_{p_A} & Y \o X \ar[d]^{p_A} \\
X \oa Y \ar[r]^{\s_{X,Y}} & Y \oa X
}\,   
\end{equation}
commutes,  and $(X, \s_{X,-}) \in Z(\CC_A)$. It is easy to see that $(X, \s_{X, -})$ and $\a^+(x)$ centralize each other for any $x \in \CC$ and $X \in \CAo$, and $\b: (\CAo)^{\rev} \to Z(\CC_A)$, $X \mapsto (X, \s_{X,-})$, can be extended to a fully faithful braided tensor functor that preserves dimensions. By Comparing their Frobenius-Perron dimensions, we find $Z(\CC_A) \simeq \CC \boxtimes (\CAo)^{\rev}$ (cf. \cite[Cor. 3.30]{DMNO}). More precisely, the functor 
\begin{equation} \label{eq:decomp}
    \a^+ \boxtimes \b: \CC \boxtimes (\CAo)^{\rev} \to Z(\CC_A), \,x \boxtimes X \mapsto \a^+(x) \o \b(X)
\end{equation}
defines a braided tensor equivalence preserving pivotal structures.

\section{Fusion actions and Frobenius-Schur indicators} \label{s:indicator}
Let $\CC$ be an MTC and $A$ a condensable algebra in $\CC$. In this section, we define the actions of the objects of $\CA$ on $A$, and discuss their relations to generalized Frobenius-Schur indicators, which are essential to the proof of categorical Schur-Weyl duality.
\subsection{Actions of spherical fusion categories on some hom-spaces}  In this subsection, we define an action of a spherical fusion category $\BB$ on the hom-spaces $\BB(X, \1)$ for any $\bX=(X, \s_{X,-}) \in Z(\BB)$ by using some operators whose traces are generalized Frobenius-Schur indicators. 

The notion of the second Frobenius-Schur indicators was introduced for the complex representations of finite groups. They have been generalized to the settings of semisimple Hopf algebras \cite{LM}, quasi-Hopf algebras \cite{MNg, NS4}, rational conformal field theory \cite{Ban}  and fusion categories \cite{FGSV, FS}. The higher Frobenius-Schur indicator $\nu_n(V)$ for any representation $V$ of a semisimple Hopf algebra was introduced in \cite{LM}, and studied extensively in \cite{KSZ}.  These developments inspired further generalizations in semisimple Hopf algebras \cite{SZ}, and in pivotal tensor categories \cite{NS07, NSind, NS10}. The fusion actions to be defined involve the generalized Frobenius-Schur indicators for spherical fusion categories introduced in \cite{NS10}, which are briefly recalled as follows.

Let $\BB$ be a spherical fusion category, $(n,k)$ a pair of integers, $Y \in \BB$ and $\bX = (X, \s_{X,-}) \in Z(\BB)$. Assuming the conventions $Y^{\o n} =(Y^*)^{\o |n|}$ if $n < 0$ and $Y^{\o 0} = \1$ in $\BB$, one can define a linear operator $E_{\bX, Y}^{(n,k)} : \BB(X, Y^{\o n{}}) \to \BB(X, Y^{\o n})$ via the half-braiding $\s_{X,-}$ of $\bX$.  The generalized Frobenius-Schur indicator is defined as $\nu_{(n,k)}^{\bX}(Y): = \Tr(E^{(n,k)}_{\bX, Y})$. Here $\Tr$ denotes the ordinary trace of a linear operator. These generalized Frobenius-Schur indicators are related by the $\SLZ$-actions on $K(Z(\BB))$ and on $\BZ^2$. These relations will be discussed in more detail when we prove the Schur-Weyl duality in the next section. 

For $n> 0$, the operator $E^{(n,1)}_{\bX, Y}$ is invertible and $E^{(n,k)}_{\bX, Y}=\Big(E^{(n,1)}_{\bX, Y}\Big)^k$. In particular, $E^{(n,0)}_{\bX, Y} = \id$ and hence
\begin{equation}\label{eq:(n,0)}
    \nu_{(n,0)}^{\bX}(Y) = [X, Y^{\o n}]_\BB\,.
\end{equation}
If $\BB$ is strict pivotal, the operator $E^{(n,1)}_{\bX, Y}$ can be defined diagrammatically, namely 
 \begin{equation*}
   E_{\bX,Y}^{(n,1)}: 
    {\xy (0,0)="ctext", "ctext"*=<20pt,11pt>[F]{\scriptstyle f},
    {\ar@{-}(0,2);(0,10)*+{\scriptstyle X}},
{"ctext"+(0,-2);"ctext"+(0,-10)*+{\scriptstyle Y^{\o n}}**\dir{-}},
\endxy} \quad \mapsto\quad 
  {\xy (0,0)="ctext", "ctext"*=<20pt,11pt>[F]{\scriptstyle f},
{"ctext"+(2,-2);"ctext"+(2,-10)*{\scriptstyle Y^{\o n-1}}**\dir{-}},
{\vcross~{"ctext"+(-8,14)*+{\scriptstyle X}}{"ctext"+(0,10)="v0"}{"ctext"+(-8,0)="v2"}{"ctext"+(0,2)}},
{"v0"; "v0"+(8,0)="v1" **\crv{"v0"+(0,4)& "v1"+(0,4)}},
{"v1"; "v1"+(0, -20) **\dir{-}?(1)+(2,0)*{\scriptstyle Y}},
{"v2";"v2"+(0,-2)="v3" **\dir{-}},
{"v3"; "ctext"+(-2,-2)="y" **\crv{"v3"+(0,-4)& "y"+(0,-4)}},
\endxy} \text{ where}\quad
\s_{X,Y^*} = \vcenter{\xy 0;/r.4pc/:
{\ar@{-}@`{(0,-2.5), (3,-3)} (0,0)*+{\color{black}\scriptstyle X}; (3,-5)},
{\ar@{-}@`{(3,-2.5),(0,-3)}|(0.5){\hole} (3,0)*+{\color{black}\scriptstyle Y^*}; (0,-5)},
\endxy}.
\end{equation*}

The operator $E_{\bX, Y}^{(0,-1)}$ defines a left action of $Y$ on $\BB(X, \1)$.  Recall from \cite[Sec. 2]{NS10} that $E_{\bX, Y}^{(0,-1)}: \BB(X,\1) \to \BB(X,\1)$ is a $\BC$-linear operator defined by
\begin{align*}
E_{\bX, Y}^{(0,-1)}(f) & := \Big(X \xrightarrow{r^{-1}} X \o \1 \xrightarrow{X \o \coev} X \o  (Y \o Y^*) \xrightarrow{a^{-1}_{X, Y, Y^{*}}} (X \o Y) \o Y^{*} \xrightarrow{\s_{X, Y}\o Y^*}\\
& (Y  \o  X) \o Y^* \xrightarrow{(Y \o f) \o Y^*}  (Y \o \1) \o Y^* \xrightarrow{r \o Y} Y\o Y^* \xrightarrow{j \o V^*} Y^{**}\o Y^* \xrightarrow{\ev} A\Big), 
\end{align*}
where $a$ and $r$ respectively denote the associativity and the right unit isomorphisms of $\BB$.  If $\BB$ is strict pivotal, we can describe  $E_{\bX, Y}^{(0,-1)}$ as
\begin{equation}\label{eq: indicator_diag}
    E_{\bX, Y}^{(0,-1)}: 
\vcenter{\xy 0;/r1pc/:
(0,0)*=<10pt,10pt>[F]{\scriptstyle f},
{\ar@{-}@`{(0,2)} (0,.45);(0,2)*+{\color{black}\scriptstyle X}},
\endxy} \quad \mapsto \quad 
\vcenter{\xy 0;/r1pc/:
(0,0)*=<10pt,10pt>[F]{\scriptstyle f},
{\ar@{-}@`{(-.3,1), (-2,.7)} (0,.45);(-1.5,2.2)*+{\color{black}\scriptstyle X}},
(0,0)="o";(-1.4,.7)**\dir{}, "o",
{\ellipse(1.3,1.5)^,=:a(340){-}},
(-1.8,-.5)*+{\scriptstyle Y},
\endxy} \qquad\text{ where}\quad
\s_{X,Y} = \vcenter{\xy 0;/r.4pc/:
{\ar@{-}@`{(0,-2.5), (3,-3)} (0,0)*+{\color{black}\scriptstyle X}; (3,-5)},
{\ar@{-}@`{(3,-2.5),(0,-3)}|(0.5){\hole} (3,0)*+{\color{black}\scriptstyle Y}; (0,-5)},
\endxy}.
\end{equation}
This leads us to the following definition of $K(\BB)$-action on $\BB(X,\1)$.
\begin{defn}\label{def:action} Let $\BB$ be a spherical fusion category. 
    For $\bX = (X, \s_{X,-}) \in Z(\BB)$ and $Y \in \BB$, we define the (left) action of $Y$ on $\BB(X, \1)$ as
    $$
    Y \cdot f := E_{\bX, Y}^{(0,-1)}(f) \quad \text{for }f \in \BB(X, \1)\,.
    $$
    In particular, $\Tr(Y, V_\bX) = \nu_{(0,-1)}^{\bX}(Y)$ by definition, where $V_\bX = \BB(X, \1)$. This action of objects in $\BB$ on $V_\bX$ can be extended $\BZ$-linearly to an action of $K_0(\BB)$, and we will simply call this $K_0(\BB)$-action on $V_{\bX}$ as \emph{the fusion action of} $\BB$. We may simply write $Yf$ for $Y\cdot f$ when there is not any ambiquity.
\end{defn}

\begin{rmk}{\rm 
  It follows from the proof of \cite[Prop. 2.8]{NS10} that $E_{\bX, Y}^{(0,k)} = E_{\bX, Y^{\o (-k)}}^{(0,-1)}$ for any integer $k$. In particular, $Y^*\cdot f = E_{\bX, Y}^{(0,1)} (f)$, which defines a right action of $Y$ on $\BB(\bX, \1)$ by the following proposition.}   
\end{rmk}

We first show that $\BB(X,\1)$ is a left $K_0(\BB)$-module with such fusion action of $\BB$.

\begin{prop} \label{p:action} Let $\BB$ be a spherical fusion category.
    The fusion action of $\BB$ on $\BB(X,\1)$ introduced in \ref{def:action} defines an $K_0(\BB)$-module. In particular, its character value at $Y$ is the generalized Frobenius-Schur indicator $\nu_{(0,-1)}^{\bX}(Y)$. 
\end{prop}
\begin{proof}
     Without loss of generality, we may assume $\BB$ is strict pivotal. For any $\bX \in Z(\BB)$ and $f \in \BB(X,\1)$, we obviously have $\1 \cdot f = f$. For any objects $V, W \in \BB$, 
    $$
    (V \o W)\cdot f=    
\vcenter{\xy 0;/r1pc/:
(0,0)*=<10pt,10pt>[F]{\scriptstyle f},
{\ar@{-}@`{(-.3,1), (-2,.7)} (0,.45);(-1.5,3)*+{\color{black}\scriptstyle X}},
(0,0)="o";(-1.4,.7)**\dir{}, "o",
{\ellipse(1.3,1.5)^,=:a(340){-}},
(-2.3,-.5)*+{\color{black}\scriptstyle V},
(0,0)="o";(-1.4,.7)**\dir{}, "o",
{\ellipse(1.9,2.2)^,=:a(340){-}},
(-.4,-.9)*+{\color{black}\scriptstyle W},
\endxy} \qquad= V\cdot (W\cdot f)\,.
    $$
    \smallskip\\
    The last assertion follows directly from the definition of generalized Frobenius-Schur indicators.
\end{proof}

Now, for any fusion subcategory $\eD$ of  $\BB$ and $\bX=(X, \s_{X,-})  \in Z(\BB)$, $\BB(X, \1)$ is then an $K(\BB)$-module by Proposition \ref{p:action}, and we define
\begin{equation}
    \BB(X, \1)^\DD :=\{f \in \BB(X, \1)\mid Y f = d(Y) f \text{ for all }Y \in \DD \}
\end{equation}
which is called the $\DD$-\emph{invariant subspace} of $\BB(X, \1)$.

The hom-space $Z(\BB)(\bX, \1)$  can be characterized as the $\BB$-invariant subspace of $\BB(X, \1)$ in the following lemma.
\begin{lem} Let $\BB$ be a spherical fusion category.
 For $\bX = (X,\s_{X,-}) \in Z(\BB)$,
 $$
 Z(\BB)(\bX, \1) = \BB(X, \1)^{\BB}\,.
 $$
\end{lem}
\begin{proof}

For $f \in Z(\BB)(\bX, \1)$ and $Y\in \BB$, we have
$$
Y\cdot f = {
\vcenter{\xy 0;/r1pc/:
(0,0)*=<10pt,10pt>[F]{\scriptstyle f},
{\ar@{-}@`{(-.3,1), (-2,.7)} (0,.45);(-1.5,2.2)*+{\color{black}\scriptstyle X}},
(0,0)="o";(-1.4,.7)**\dir{}, "o",
{\ellipse(1.3,1.5)^,=:a(340){-}},
(-1.8,-.5)*+{\color{black}\scriptstyle Y},
\endxy}
}\qquad  =\quad
{
\vcenter{\xy 0;/r1pc/:
(-2,.5)*=<10pt,10pt>[F]{\scriptstyle f},
{\ar@{-}@`{(-2,1)} (-2,.95);(-2,2.7)*+{\color{black}\scriptstyle X}},
(0,0)="o";(-1.4,.7)**\dir{}, "o",
{\ellipse(1.3,1.5)^,=:a(360){-}},
(-1.6,-.9)*+{\color{black}\scriptstyle Y},
\endxy}}\qquad = d(Y) f\,.
$$
Therefore, $Z(\BB)(\bX, \1) \subseteq \BB(X, \1)^{\BB}$. 
Conversely, if $f \in\BB(X, \1)^{\BB}$, then  $Yf = d(Y) f$ for all $Y \in \irr(\BB)$. Consider the nondegenerate bilinear form $\langle g, Yf\rangle = d(Y) \langle g, f\rangle$ for $g \in \BB(\1, X)$. This means
$$
{
\vcenter{\xy 0;/r1pc/:
(0,-.5)*=<10pt,10pt>[F]{\scriptstyle f},
(-1.5,2.4)*=<10pt,10pt>[F]{\scriptstyle g},
{\ar@{-}@`{(-.3,1), (-1.5,.7)} (0,-.1);(-1.5,2)},
{\ar@{-}@`{(0,1), (-2,.7)}|(.5){\hole} (0,1.8);(-1.5,-.7)},
(-2,-.5)*+{\color{black}\scriptstyle Y},
\endxy}} \quad =\quad 
{
\vcenter{\xy 0;/r1pc/:
(0,-.5)*=<10pt,10pt>[F]{\scriptstyle f},
(0,2)*=<10pt,10pt>[F]{\scriptstyle g},
{\ar@{-} (0,-.1);(0,1.6)},
{\ar@{-} (1.5,-.8);(1.5,2)},
(2,-.5)*+{\color{black}\scriptstyle Y},
\endxy}} \quad \text{and hence }\quad
{
\vcenter{\xy 0;/r1pc/:
(0,-.5)*=<10pt,10pt>[F]{\scriptstyle f},
{\ar@{-}@`{(-.3,1), (-1.5,.7)} (0,-.1);(-1.5,2)},
{\ar@{-}@`{(0,1), (-2,.7)}|(.5){\hole} (0,1.8);(-1.5,-.7)},
(-2,-.5)*+{\color{black}\scriptstyle Y},
\endxy}} \quad =\quad 
{
\vcenter{\xy 0;/r1pc/:
(0,-.5)*=<10pt,10pt>[F]{\scriptstyle f},
{\ar@{-} (0,-.1);(0,2)},
{\ar@{-} (1.5,-.8);(1.5,2)},
(2,-.5)*+{\color{black}\scriptstyle Y},
\endxy}} \,,
$$
or $f \in Z(\BB)(\bX,\1)$. This proves the  lemma.
\end{proof}
\subsection{Fusion actions on condensable subalgebras} Now, we consider a condensable algebra $A$ in any modular tensor category $\CC$. Since $\CA$ is a spherical fusion category and $\a^+(x) \in Z(\CA)$ for any $x \in \CC$,  we have the fusion action of $\CA$ on $\CA(\a(x), \1)$ by the preceding subsection. We will define the fusion action of $\CA$ on the algebra $A$ in Definition \ref{def:rho}.

The notion of the objects $Y \in \CA$ acting on $A$ is equivalent to an algebra homomorphism $K(\CA) \to \CC(A,A)$. Since $\CC$ is semisimple, compositions of morphisms provide the decomposition
$$
\CC(A,A) = \bigoplus_{x \in \iC} \CC(x, A) \o \CC(A, x) \,.
$$
We proceed to define an action of $K_0(\CA)$ on $\CC(x, A)$ for $x \in \CC$. The adjunction \eqref{eq:adj} insinuates  an action of $K(\CA)$ on $\CC(x, A)$  which is equivalent the fusion action of $\CA$ on $\CA(\a(x), \1)$. 

Recall that the evaluation $\ol{\ev}_Y: Y^* \oa Y \to A$ and coevaluation $\ol{\coev}_Y: A \to Y \oa Y^*$ morphisms for the left dual of $Y \in \CA$ are determined by the morphisms $\widetilde{\ev}_Y : Y^* \o Y \to A$ and $\widetilde{\coev}_Y: A \to Y \o Y^*$ in $\CC$ given by the diagrams 
$$
\widetilde{\ev}_Y\, =\, \frac{1}{d(A)}
\vcenter{\xy  0;/r1pc/:
(0,0)="o";(-1.4,-.2)**\dir{}, "o",
{\ellipse(1.5,2.1)^,=:a(160){-}},
{\ar@{.}@`{(2.5, 1), (2.3, -1.7)} (1.3,-1.3);(2.5, -2.8)},
(-1.4, -.1)*{\scriptstyle Y^*},
(1.4, -.1)*{\scriptstyle Y},
\endxy}\, \hspace{10pt} = \quad  \frac{1}{d(A)}
\vcenter{\xy  0;/r1pc/:
(0,0)="o";(-1.4,-.2)**\dir{}, "o",
{\ellipse(1.5,2.1)^,=:a(160){-}},
{\ar@{.}@`{(2.5, 0), (2, -1.9)}|(.35){\hole} (-1,-1.6);(2.2, -2.6)},
(-1.4, -.1)*{\scriptstyle Y^*},
(1.4, -.1)*{\scriptstyle Y},
\endxy}
$$
$$
\widetilde{\coev}_Y\, = 
\vcenter{\xy  0;/r1pc/:
(0,0)="o";(-1.4,-.2)**\dir{}, "o",
{\ellipse(1.5,2.1)_,=:a(160){-}},
{\ar@{.}@`{(-1.3, 1), (1.7, 1.8)}|(.63){\hole} (-1.3,1);(1.7, 2.9)},
(-1.4, 0)*{\scriptstyle Y},
(1.6, 0)*{\scriptstyle Y^*},
\endxy} \hspace{10pt} = \quad
\vcenter{\xy  0;/r1pc/:
(0,0)="o";(-1.4,-.2)**\dir{}, "o",
{\ellipse(1.5,2.1)_,=:a(160){-}},
{\ar@{.}@`{(-1.3, 1), (-2, 2)} (-1.3,1);(-2, 2.9)},
(-1.4, 0)*{\scriptstyle Y},
(1.6, 0)*{\scriptstyle Y^*},
\endxy}
$$
and the commutative diagrams
$$
\xymatrix{ Y^* \o Y \ar[r]^-{\widetilde{\ev}}\ar[d]_-{p_A} & A \\
Y^* \oa Y \ar[ru]_-{\ol{\ev}} & 
}\,, \qquad \xymatrix{ A \ar[rd]_-{\ol{\coev}} \ar[r]^-{\widetilde{\coev}} & Y \o Y^* \ar[d]^-{p_A}  \\
& Y \oa Y^* 
}\,.
$$
The adjunction \eqref{eq:adj} and Definition \ref{def:action} inspire the following definition of fusion action of $\CA$ on $\CC(x, A)$ for any $x \in \CC$.  
\begin{defn}\label{def:action2}
For $x \in \CC$, $Y \in \CC_A$ and $g \in \CC(x, A)$,  we define:
\begin{align*}
    Y\cdot g = &  \Big(x \xrightarrow{x \o \coev_Y}x \o Y \o Y^* \xrightarrow{R_{x, Y}\o Y^*} Y \o x \o Y^* \xrightarrow{Y\o g\o Y^*}Y \o A \o Y^* \\
    & \xrightarrow{\mu_Y\o Y^*}  Y \o Y^* =  Y^{**}  \o Y^* \xrightarrow{\widetilde{\ev}_{Y^*} } A\Big)\,.
\end{align*}
In terms of diagram with the omission of ``$\cdot$'', we have
\begin{equation}
    Yg:=
\frac{1}{d(A)}\vcenter{\xy  0;/r1pc/:
(0,0)="o";(-1.4,1)**\dir{}, "o",
{\ellipse(1.5,2.1)^,=:a(340){-}},
(0,.3)*=<10pt,10pt>[F]{\scriptstyle g},
{\ar@{-}@`{(-.1,1.5), (-2,2)} (0,.8);(-1.5,3.3)*+{\scriptstyle x}},
{\ar@{.}@`{(-.3, -1)} (0,-.2);(-1.2, -1.1)},
{\ar@{.}@`{(2.5, 1), (2.3, -1.7)} (1.3,-1.3);(2.5, -2.8)},
\endxy}\,.
\end{equation}
\end{defn}

For $x \in \CC$, $\a^+(x) = (\a(x), \s_{\a(x),-}) \in Z(\CA)$ where the half-braiding $\s_{\a(x), Y}$ is given by
$$
\s_{\a(x), Y} := \left(x \o A \oa Y \xrightarrow{x\o l_Y^{-1}} x \o Y \xrightarrow{R_{x,Y}} Y \o x \xrightarrow{\sim} Y \oa x \o A \right)
$$
for $Y \in \CA$. Now, we can prove $\vp_x$ is an isomorphism of $K(\CA)$-modules.

\begin{prop} \label{p:interwiner}
The natural isomorphism $\vp_x : \CA(\a(x), \1) \to \CC(x,A)$, for any  $x \in \CC$, is an intertwining operator of $K_0(\CA)$-actions on $\CA(\a(x), \1)$ and $\CC(x, A)$  given  respectively by Definitions \ref{def:action} and \ref{def:action2}. In particular, $\CC(x,A)$ is an $K(\CA)$-module.
\end{prop}
\begin{proof}
    Without loss of generality, we may assume both $\CC$ and $\CA$ are strict pivotal. For $Y \in \CA$ and $f \in \CA(\a(x), \1)$, 
    \begin{align*}
      \vp_x(Yf) =\Big( x \xrightarrow{x \o u_A} \a(x) = \a(x) \oa A \xrightarrow{x\o \ol{\coev}} \a(x) \oa Y \oa Y^* \xrightarrow{\s_{\a(x),Y}\oa Y^*} \\ Y\oa \a(x) \oa Y^* 
       \xrightarrow{Y \oa f \oa  Y^*} Y \oa A \oa Y^*  = Y^{**} \oa Y^* \xrightarrow{\ol{\ev}} A\Big)\,.
    \end{align*}
    The commutativity of the following diagram
    $$
    \xymatrix{\a(x) \ar@{=}[r] & a(x)\oa A  \ar[r]^-{x\o \ol{\coev}} & \a(x) \oa Y \oa Y^* \ar[r]^-{\s\oa Y^*} & Y\oa \a(x) \oa Y^* \ar[r] & \\
    x\ar[u]^-{x \o u_A} \ar[rr]^-{x \o \coev_Y}&&  x \o Y \o Y^* \ar[u]^-{p_A} \ar[r]^-{R_{x,Y} \o Y^*} & Y \o x \o Y^* \ar[u]_-{p_A} \ar[r]& \ar@{}[u]^-{\text{continue} \atop \text{next line}}
    }
    $$
     $$
    \xymatrix{  \ar[rr]^-{Y \oa f \oa  Y^*} &&  Y \oa A \oa Y^* \ar@{=}[r] &Y \oa Y^* \ar@{=}[r]& Y^{**} \oa Y^* \ar[dr]^-{\ol{\ev}_{Y^*}}\\
    \ar[rr]_-{Y\! \o \vp_x(\!f\!) \o Y\!^*} && Y \o A \o Y^* \ar[u]_-{p_A} 
       \ar[r]_-{\mu_Y \o Y^*} & 
        Y \o Y^*  \ar@{=}[r]\ar[u]^-{p_A} & Y^{**} \o Y^* \ar[u]^-{p_A}\ar[r]_-{\tilde{\ev}} & A
    }
    $$
    implies $\vp_x(Y f)$ equals to the composition
     \begin{align*}
       x &\xrightarrow{x \o \coev} x \o Y \o Y^* \xrightarrow{R_{x,Y}\o Y^*} Y\o x \o Y^* \xrightarrow{Y \o \vp_x(f) \o Y^*} Y \o A \o Y^* \\ 
       &\xrightarrow{\mu_Y \o Y^*} 
        Y \o Y^*  = Y^{**} \o Y^* \xrightarrow{\tilde{\ev}} A,
    \end{align*}
    which is $Y \vp_x(f)$ by Definition \ref{def:action2}.
\end{proof}
Similarly, for any fusion subcategory $\BB$ of $\CA$ and $x \in \CC$, we define
\begin{equation} \label{eq:inv_subspace}
    \CC(x, A)^\BB = \{ g \in \CC(x,A) \mid  Y g = d_A(Y) g  \text{ for all } Y \in \BB\}
\end{equation}
which is the isotypic component of $\CC(x, A)$ corresponding to the dimension character of $K(\BB)$, and is called the $\BB$-\emph{invariant subspace} of $\CC(x,A)$. The fusion subcategory $\CAo$ can be characterized as follows: 
\begin{prop} \label{p:CAo}
    If $Y$ is a simple $A$-module but $Y \not\in \CAo$, then $\tr(Y\id_A) = 0$,  where $\tr$ is the categorical trace of $\CC$. Moreover, a fusion subcategory $\BB$ of $\CA$ is contained in $\CAo$ if and only if
    \begin{equation}\label{eq:inv0}
        \CC(x, A)^\BB = \CC(x, A) \quad \text{ for all } x \in \CC\,.
    \end{equation}
    In particular, $\CAo$ is the fusion subcategory of $\CA$ satisfying  \eqref{eq:inv0} with the greatest Frobenius-Perron dimension.
\end{prop}
\begin{proof}
For $x \in \CC$ and $g \in \CC(x,A)$, we find
$$
d(A)\cdot Yg = \vcenter{\xy  0;/r1pc/:
(0,0)="o";(-1.4,1)**\dir{}, "o",
{\ellipse(1.5,2.2)^,=:a(350){-}},
(0,.3)*=<10pt,10pt>[F]{\scriptstyle g},
{\ar@{-}@`{(-.1,1.5), (-2,2)} (0,.8);(-1.5,3.3)*+{\scriptstyle x}},
{\ar@{.}@`{(-.3, -1)} (0,-.2);(-1.2, -1.1)},
{\ar@{.}@`{(2.5, 1), (2.3, -1.7)} (1.3,-1.3);(2.5, -2.8)},
(-1.1,-.1)*+{\scriptstyle Y},
\endxy}=
 \vcenter{\xy  0;/r.9pc/:
(0,0)="o";(-1.4,1)**\dir{}, "o",
{\ellipse(1.5,2.2)^,=:a(350){-}},
(-1.5,2.8)*=<10pt,10pt>[F]{\scriptstyle g}="g",
{\ar@{.}@`{(-1.4,1.5), (1.3,1.2)} (-1.5,2.3);(-1.2, -1.1)},
{\ar@{.}@`{(2.5, 1), (2.3, -1.7)} (1.3,-1.3);(2.5, -2.8)},
(-1.1,0)*+{\scriptstyle Y},
(-1.5,3.2); (-1.5, 4.1)*+{\scriptstyle x} **\dir{-},
\endxy}=
\vcenter{\xy 0;/r.9pc/: (0,0)="o";(1,1.4)**\dir{}, "o",
{\ellipse(1.5,2.2)^,=:a(350){-}},
{\ar@{.}@`{(2.5, 1), (2.3, -1.7)} (1.3,-1.3);(2.4, -2.8)},
{\ar@{.}@`{(.9,1.5)} (.9,2.5);(-1.2, -1.)},
(.9,3)*=<10pt,10pt>[F]{\scriptstyle g}="g",
(-1.1,0)*+{\scriptstyle Y},
(.9,3.5); (.9, 4.3)*+{\scriptstyle x} **\dir{-},
\endxy}
=
\vcenter{\xy 0;/r.9pc/: (0,0)="o";(1,1.4)**\dir{}, "o",
{\ellipse(1.5,2.2)^,=:a(350){-}},
{\ar@{.}@`{(2.5, 1), (2.3, -1.7)} (1.3,-1.3);(2.5, -2.8)},
{\ar@{.}@`{(.9,2), (-.3,1.1)} (1.2,2.5);(1.4, -.1)},
(1.2,3)*=<10pt,10pt>[F]{\scriptstyle g}="g",
(-1,-.4)*+{\scriptstyle Y},
(1.2,3.5); (1.2, 4.3)*+{\scriptstyle x} **\dir{-},
\endxy}
=
\vcenter{\xy 0;/r.9pc/: (0,0)="o";(1,1.4)**\dir{}, "o",
{\ellipse(1.5,2.2)^,=:a(350){-}},
{\ar@{.}@`{(2.5, 1), (2.3, -1.7)} (1.3,-1.3);(2.4, -2.8)},
(1.2,3)*=<10pt,10pt>[F]{\scriptstyle g}="g",
(-1,-.4)*+{\scriptstyle Y},
(1.2,3.5); (1.2, 4.3)*+{\scriptstyle x} **\dir{-},
{\ar@{.}@`{(.9,2), (-.3,1.1)} (1.2,2.5);(1.1, 1)},
{\ar@{.}@`{ (2.6,.6)} (1.5, 1);(1.5,-.1);},
\endxy}\,.
$$
Note that $Y \in \CAo$ if and only if  $Y^* \in \CAo$. Thus, if  $Y \in \CAo$
$$
\vcenter{\xy 0;/r.9pc/: (0,0)="o";(1,1.4)**\dir{}, "o",
{\ellipse(1.5,2.2)^,=:a(350){-}},
{\ar@{.}@`{(2.5, 1), (2.3, -1.7)} (1.3,-1.3);(2.4, -2.8)},
(1.2,3)*=<10pt,10pt>[F]{\scriptstyle g}="g",
(-1.1,-.4)*+{\scriptstyle Y},
(1.2,3.5); (1.2, 4.3)*+{\scriptstyle x} **\dir{-},
{\ar@{.}@`{(.9,2), (-.3,1.1)} (1.2,2.5);(1.1, 1)},
{\ar@{.}@`{ (2.6,.6)} (1.5, 1);(1.5,-.1);},
\endxy} =
\vcenter{\xy 0;/r.9pc/: 
{\ellipse(1.4,2.1){}},
(2.7,3)*=<10pt,10pt>[F]{\scriptstyle g}="g",
{\ar@{.}@`{(2.7,1.5)}, (2.7,2.5) ;(2.3, 0.5)},
{\ar@{.}@`{(3.5, 1), (3.3, -1.7)} (2.2,-1.3);(3.4, -2.8)},
"g"; "g"+(0, 1.5)*+{\scriptstyle x} **\dir{-},
(0,-.4)*+{\scriptstyle Y},
\endxy}
 =
\vcenter{\xy 0;/r.9pc/: 
{\ellipse(1.4,2.1){}},
(2.7,3)*=<10pt,10pt>[F]{\scriptstyle g}="g",
{\ar@{.}@`{(2.7,1.5)}, (2.7,2.5) ;(2.35, 0.5)}, (.4,2.6)="p",
{\ar@{.}@`{"p"+(3.5, 1), (3.3, -1.7)} "p"+(2.2,-1.3);(3.4, -2.8)},
"g"; "g"+(0, 1.5)*+{\scriptstyle x} **\dir{-},
(0,-.4)*+{\scriptstyle Y},
\endxy} = \, d(Y) \,\,
\vcenter{\xy 0;/r.9pc/: 
(0,0)*=<10pt,10pt>[F]{\scriptstyle g}="g",
"g"; "g"+(0, 2)*+{\scriptstyle x} **\dir{-},
"g"; "g"+(0, -2.5) **\dir{.},
\endxy}\,,
$$
and so $Yg = d_A(Y)g$. Thus, $\CC(x, A)^{\CAo} = \CC(x, A)$ for all $x \in \CC$, and so any fusion subcategory $\BB$ of $\CAo$ satisfies $\eqref{eq:inv0}$.

On the other hand, if $Y \in \iCA$ but $Y \not\in \CAo$, we consider the action of $Y$ on $\id_A \in \CC(A,A)$. Then
$$
\tr(Y\id_A) = \frac{1}{d(A)}\tr\left(\vcenter{\xy 0;/r.9pc/: (0,0)="o";(1,1.4)**\dir{}, "o",
{\ellipse(1.5,2.2)^,=:a(350){-}},
{\ar@{.}@`{(2.5, 1), (2.3, -1.7)} (1.3,-1.3);(2.4, -2.8)},
(-1,-.4)*+{\scriptstyle Y},
{\ar@{.}@`{(1.4, 2.7)}(1.2,2.5); (1.2, 3.5)},
{\ar@{.}@`{(.9,2), (-.3,1.1)} (1.2,2.5);(1.1, 1)},
{\ar@{.}@`{ (2.6,.6)} (1.5, 1);(1.5,-.1);},
\endxy} \right) =
\frac{1}{d(A)}\vcenter{\xy 0;/r.9pc/: (0,0)="o";(1,1.4)**\dir{}, "o",
{\ellipse(1.5,2.2)^,=:a(350){-}},
(-1,-.4)*+{\scriptstyle Y},
{\ar@{.}@`{(.9,2), (-.3,1.1)} (1.2,2.2);(1.1, 1)},
{\ar@{.}@`{ (2.6,.6)} (1.5, 1);(1.5,-.1);},
{\ar@{.}@`{(3, -.5), (3, 2.5)} (1.6,-.7);(1.2,2.2)},
\endxy} = 0\,,
$$
where the last equality follows from \cite[Lem. 4.3]{KO02}, and so $Y\id_A \ne d_A(Y) \id_A$.
In particular, if $\BB$ is not a fusion subcategory of $\CAo$, then $\CC(A,A)^\BB \ne \CC(A,A)$. 
\end{proof}
Now, we define the action of $\CA$ on the algebra $A$. 
\begin{defn} \label{def:rho}Let $\CC$ be any MTC and $A \in \CC$ a condensable algebra. We define the $\BC$-linear map $\rho: K(\CA) \to \CC(A,A)$ by
    $$
    \rho(Y) = Y \id_A = \frac{1}{d(A)}
\vcenter{\xy 0;/r.9pc/: (0,0)="o";(1,1.5)**\dir{}, "o",
{\ellipse(1.5,2.2)^,=:a(345){-}},
{\ar@{.}@`{(2.5, 1), (2.3, -1.7)} (1.3,-1.3);(2.4, -2.8)},
{\ar@{.}@`{(.9,1.5)} (1.3, 3);(-1.4, -.3)},
(-1.4,1.5)*+{\scriptstyle Y},
\endxy}\quad \text{for } Y \in \iCA\,.
$$
\end{defn} 
The next lemma proves that $\rho$ defines an action of $K(\CC_A)$ on $A$, and so $\rho$ will be referred as the fusion action of $\CA$ on $A$. We would like to highlight that an analog of Lemma \ref{l:rho_dual} was established in Theorem 5 and in the proof Theorem 6 of \cite{Rie2025}.
\begin{lem}\label{l:rho_dual}
   The linear map $\rho$ is a $\BC$-algebra homomorphism that preserves duality, i.e. $\rho(Y^*) = \rho(Y)^*$ for all $Y \in \CA$.
\end{lem}
\begin{proof}
It suffices to prove that $\rho$ respects the multiplication of simple objects for the first assertion. For any $X, Y \in \iCA$, we have
 \begin{align*}
    \rho(X) \circ \rho(Y) & = \frac{1}{d(A)^2}
\vcenter{\xy 0;/r.8pc/: (0,.1)="o";(.9,1.5)**\dir{}, "o",
{\ellipse(1.5,2.2)^,=:a(330){-}},
{\ar@{.}@`{(2.6, .5), (2.3, -1.7), (2,-4)} (1.2,-1.3);(-.1, -6.2)},
{\ar@{.}@`{(.9,1.5)} (1.3, 3);(-1.3, -1.)},
(-1.1,.2)*+{\scriptstyle Y}, "o"+(1,-5),
(.8,-4.9);(1,-4.6)**\dir{},
{\ellipse(1.5,2.2)^,=:a(330){-}},
{\ar@{.}@`{(4, -5), (2.8, -8.2)} (1.7,-6.6);(3, -8)},
(0,-5)*+{\scriptstyle X},
\endxy}  =
\frac{1}{d(A)^2}
\vcenter{\xy 0;/r.8pc/: (0,0)="o";(.2,.4)**\dir{},
{\ellipse(2,3)^,=:a(345){-}}, "o" ;(.1, .2)**\dir{},
{\ellipse(0.8,1.6)^,=:a(330){-}},
{\ar@{.}@`{(.9,1.5)} (1.3, 4);(-.55, -.8)},
{\ar@{.}@*{[thicker]}@`{(2, 1.5), (2.5, -2)} (.7,-.7);(-.9, -2.2)},
(-.4,.4)*+{\scriptstyle Y}, (-.6,4.5)="p",
{\ar@{.}@`{"p"+(4, -5), "p"+(2.8, -8.2)} "p"+(2,-6.6);"p"+(3, -8)},
(-2.4,.4)*+{\scriptstyle X},
\endxy} \\
& =\frac{1}{d(A)^2}
\vcenter{\xy 0;/r.8pc/: (0,0)="o"; (1,1.4)**\dir{}, "o",
{\ellipse(2,3)^,=:a(345){-}}, "o";(0,0)**\dir{},
{\ellipse(0.8,1.6)^,=:a(330){-}},
{\ar@{.}@`{(.9,1.5)} (1.3, 4);(-.55, -.8)},
{\ar@{.}@`{(1.6, 1.4)} (.7,-.7);(1.9, -.7)},
(-.4,.4)*+{\scriptstyle Y}, (-.6,4.5)="p",
{\ar@{.}@`{"p"+(4, -5), "p"+(2.8, -8.2)} "p"+(2.1,-6.6);"p"+(3, -9.1)},
(-2.4,.4)*+{\scriptstyle X},
\endxy}
= \frac{1}{d(A)} \vcenter{\xy 0;/r.9pc/: (0,0)="o";(1,1.4)**\dir{}, "o",
{\ellipse(1.5,2.2)^,=:a(340){-}},
{\ar@{.}@`{(2.5, 1), (2.3, -1.7)} (1.3,-1.3);(2.4, -2.8)},
{\ar@{.}@`{(.9,1.5)} (1.3, 3);(-1.2, -1.)},
(-2,2,2)*+{\scriptstyle X \oa Y}
\endxy}  = \rho(X \oa Y)\,.
 \end{align*}
Here the second last equation follows from \cite[Lem. 1.21]{KO02}. 

    We proceed to prove the equality $\rho(Y^*)=\rho(Y)^*$ by using the nondegeneracy of the pairing $\langle\cdot, \cdot \rangle : \CC(A,A) \times \CC(A,A) \to \BC$. For any $f \in \CC(A,A)$,
    \begin{align*}
   & d(A) \langle f, \rho(Y)^* \rangle = 
\tr\left(\vcenter{\xy 0;/r.9pc/: (0,0)="o";(-1,1.4)**\dir{}, "o",
{\ellipse(1.5,2.2)^,=:a(340){-}},
{\ar@{.}@`{(2, -.8), (2.3, 0)} (1.2,-1.3);(2.4, 1.6)},
{\ar@{.}@`{(.7,4),(-4,3.5), (-3,-2)} (-1.3, -1.); (-3, -3) },
(-1.3,1.8)*+{\scriptstyle Y},
(2.4,2)*+=<9pt,9pt>[F]{\scriptstyle f},
{\ar@{.} (2.4, 2.4); (2.4, 3.5) },
\endxy}\right)\, =\,
\vcenter{\xy 0;/r.9pc/: (0,0)="o";(-.3,1.4)**\dir{}, "o",
{\ellipse(1.5,2.2)^,=:a(340){-}},
{\ar@{.}@`{(2, -.8), (2.3, 0)} (1.2,-1.3);(2.4, 1.6)},
{\ar@{.}@`{(0,4), (2.2,4)} (-1.3, -1.); (2.4, 2.4) },
(-1.3,1.8)*+{\scriptstyle Y},
(2.4,2)*+=<9pt,9pt>[F]{\scriptstyle f},
\endxy}\,=\,
\vcenter{\xy 0;/r.9pc/: (0,0)="o";(.4,1.4)**\dir{}, "o",
{\ellipse(1.5,2.2)^,=:a(340){-}},
{\ar@{.}@`{(2, -.8), (2.3, 0)} (1.2,-1.3);(2.4, 1.6)},
{\ar@{.}@`{(-1,2), (2,4)} (1.1, .5); (2.4, 2.4) },
{\ar@{.}@`{(1.9,.3), (1.9,0)} (1.7, .4); (1.5, -.5) },
(-1.3,1.8)*+{\scriptstyle Y},
(2.4,2)*+=<9pt,9pt>[F]{\scriptstyle f},
\endxy} \\
=&\,
\,\vcenter{\xy 0;/r.9pc/: (0,0)="o";(.4,1.4)**\dir{}, "o",
{\ellipse(1.5,2.2)^,=:a(340){-}},
{\ar@{.}@`{(2, -.8), (2.3, 0)} (1.2,-1.3);(2.4, 1.6)},
{\ar@{.}@`{(-1,2), (2,4)} (1.1, .5); (2.4, 2.4) },
{\ar@{.} (1.7, .3); (2.1, .2) },
(-1.3,1.8)*+{\scriptstyle Y},
(2.4,2)*+=<9pt,9pt>[F]{\scriptstyle f},
\endxy}\,=\, \vcenter{\xy 0;/r.9pc/: (1.8,-1)="o";(2.2,1)**\dir{}, "o",
{\ellipse(1.8,1.8)^,=:a(340){-}},
{\ar@{.}@`{(2, -.8), (2.3, 0)} (0,-1);(2.4, 1.6)},
{\ar@{.}@`{(-1,2), (2,4)}|(.1){\hole} (2.1, 0); (2.4, 2.4) },
(3.2,-1)*+{\scriptstyle Y},
(2.4,2)*+=<9pt,9pt>[F]{\scriptstyle f},
\endxy} \hspace{-5pt} =
\knotstyle{.}
\vcenter{\xy 0;/r.9pc/: (2,1)\vcap[1],
\vtwist~{(2,1)}{(3,1)}{(2,-.5)}{(3.4,-.5)},
(2,-.9)*+=<9pt,9pt>[F]{\scriptstyle f},
{\ar@{.}@`{(2.5,-2.5), (3, -2)} (2, -1.3); (3.5, -3) },
{\ar@{.}@`{(3.5,-1.5), (2, -2)}|(.6){\hole} (3.4, -.5); (2, -3) },
{\ar@{.}@`{(3.5,-5), (1, -5)} (3.5, -3); (1,-5) },
(2.7,-4.8)="o";(4.2,.7)**\dir{}, "o",
{\ellipse(1.8,1.8)^,=:a(345){-}},
{\ar@{.}@`{(2,-3.8)} (2, -3.3);(3,-4.3) },
(4,-5)*+{\scriptstyle Y},
\endxy}\hspace{-10pt} =\,
\vcenter{\xy 0;/r.9pc/: (0,0)="o";(-1,1.4)**\dir{}, "o",
{\ellipse(1.5,2.2)^,=:a(340){-}},
{\ar@{.}@`{(.7,2.3), (-1.2,2.3)} (-1.3, -1.); (-1.2, 3) },
(-1.2,3.4)*+=<9pt,9pt>[F]{\scriptstyle f},
(-1.7,1.5)*+{\scriptstyle Y^*},
(-1.2,3.9)\vcap[2],
{\ar@{.}@`{(.9,2)}|(.4){\hole} (0.8, 3.9); (-.8, -1.9) },
\endxy}\,=\hspace{-4pt}
\vcenter{\xy 0;/r.9pc/: (0,0)="o";(-1,1.4)**\dir{}, "o",
{\ellipse(1.3,2)^,=:a(340){-}},
{\ar@{.}@`{(.7,2.3), (-1.2,2.3)} (-1.2, -1.); (-1.2, 3) },
(-1.2,3.4)*+=<9pt,9pt>[F]{\scriptstyle f},
(-1.5,1.5)*+{\scriptstyle Y^*},
(-1.2,3.9)\vcap[3],
{\ar@{.}@`{(2,1),(1.5, 0)} (1.8, 3.9); (1.1,-1.2) },
\endxy} =d(A) \langle f, \rho(Y^*) \rangle\,, 
     \end{align*}\\\\
  where the fifth equality follows from the spherical pivotal structure, and the 7-th equality is a consequence of $\theta_A = \id_A$ and $A$ being commutative.  The nondegeneracy of $\langle\cdot, \cdot\rangle$ and $d(A) \ne 0$ imply $\rho(Y^*) = \rho(Y)^*$.
\end{proof}

\begin{prop} \label{p:diff_form}
    For any $Y \in \CA$, we have
    $$
    \vcenter{\xy 0;/r.9pc/: (0,0)="o";(1,1.5)**\dir{}, "o",
{\ellipse(1.5,2.2)^,=:a(345){-}},
{\ar@{.}@`{(2.5, 1), (2.3, -1.7)} (1.3,-1.3);(2.4, -2.8)},
{\ar@{.}@`{(.9,1.5)} (1.3, 3);(-1.4, -.3)},
(-1.6,1.5)*+{\scriptstyle Y},
\endxy} \quad = \quad
    \vcenter{\xy 0;/r1pc/: (0,0)="o";(-1,1.4)**\dir{}, "o",
{\ellipse(1.3,2)^,=:a(340){-}},
{\ar@{.}@`{(.7,2.3), (-1.2,2.3)} (-.5,-1.6); (-1.2, 3) },
{\ar@{.}@`{(-.7,-2), (-1,-3)} (-.7,-2); (-1.3, -3) },
{\ar@{.}@`{(-.7,1)} (-1.3,-.3); (-.1, -.3) },
(-1.5,1.1)*+{\scriptstyle Y},
\endxy} \qquad\,.
$$       
\end{prop}
\begin{proof}
Recall that $\ev_A = \e m_A$. So
\begin{align*}
&\vcenter{\xy 0;/r1pc/: (0,0)="o";(-1,1.4)**\dir{}, "o",
{\ellipse(1.3,2)^,=:a(340){-}},
{\ar@{.}@`{(.7,2.3), (-1.2,2.3)}|(.05)\hole (1,-2); (-1.2, 3) },
{\ar@{.}@`{(2,-4), (3,3)} (1,-2); (3, 3) },
{\ar@{.}@`{(-.7,1)} (-1.3,-.3); (.7, -.3) },
(-1.5,1.1)*+{\scriptstyle Y},
\endxy} =
\vcenter{\xy 0;/r1pc/: (0,0)="o";(-1,1.4)**\dir{}, "o",
{\ellipse(1.3,2)^,=:a(340){-}},
{\ar@{.}@`{(.7,2.3), (-1.2,2.3)}|(.1)\hole (1.2,-2.5); (-1.2, 3) },
{\ar@{.}@`{(3,2)} (1.2,-2.5); (3, 3) },
{\ar@{.} (1.2,-2.5); (1.2, -3.3) },
{\ar@{.}@`{(-.7,1)} (-1.3,-.3); (.7, -.3) },
(-1.5,1.1)*+{\scriptstyle Y},
(1.2, -3.5)*{\circ},
\endxy}
 =
\vcenter{\xy 0;/r1pc/: (0,0)="o";(-1,1.4)**\dir{}, "o",
{\ellipse(1.3,2)^,=:a(330){-}},
{\ar@{.}@`{(.7,2.3), (-1.2,2.3)}|(.1)\hole (1.7,-1.5); (-1.6, 3) },
{\ar@{.}@`{(3,2)} (1.2,-2.5); (3, 3) },
{\ar@{.} (1.2,-2.5); (1.2, -3.3) },
{\ar@{.}@`{(.5,1.5)}|(.85)\hole (-1.3,-.3); (1.2, -2.5) },
(-1.5,1.1)*+{\scriptstyle Y},
(1.2, -3.5)*{\circ},
\endxy}
 =
\vcenter{\xy 0;/r1pc/: (0,0)="o";(-1,1.4)**\dir{}, "o",
{\ellipse(1.3,2)^,=:a(330){-}},
{\ar@{.}@`{(.7,2.3), (-1.2,2.3)}|(.1)\hole (1.7,-1.5); (-1.6, 3) },
{\ar@{.}@`{(1.5,-3.5)} (1.2,-2.5); (3, 3) },
{\ar@{.}@`{(.5,1.5)}|(.85)\hole (-1.3,-.3); (1.2, -2.5) },
(-1.5,1.1)*+{\scriptstyle Y},
\endxy}\\
&=
\vtop{\xy 0;/r.9pc/: (0,0)="o";(1,1.5)**\dir{}, "o",
{\ellipse(1.5,2.2)^,=:a(345){-}},
{\ar@{.}@`{(2, 0)}|(.55){\hole} (-1.4,-1);(2.4, 3)},
{\ar@{.}@`{(.9,1.5)} (1, 3);(-.3, -.5)},
(-1.6,1.5)*+{\scriptstyle Y},
\endxy} =   
\vtop{\xy 0;/r.9pc/: (0,0)="o";(1,1.5)**\dir{}, "o",
{\ellipse(1.5,2.2)^,=:a(345){-}},
{\ar@{.}@`{(2, 0)}|(.55){\hole} (-1.4,-1);(2.4, 3)},
{\ar@{.}@`{(.9,1.5)} (1.3, 3);(-1.4, -.3)},
(-1.6,1.5)*+{\scriptstyle Y},
\endxy} =\vtop{\xy 0;/r.9pc/: (0,0)="o";(1,1.5)**\dir{}, "o",
{\ellipse(1.5,2.2)^,=:a(345){-}},
{\ar@{.}@`{(2, 0)} (1.3,-1.3);(2.4, 3)},
{\ar@{.}@`{(.9,1.5)} (1.3, 3);(-1.4, -.3)},
(-1.6,1.5)*+{\scriptstyle Y},
\endxy}\,.\\
\end{align*}
Now the statement follows the zig-zag equation of duality.
\end{proof}
It is worth noting that $\rho(Y)$ can be rewritten to a left $A$-module version  
$$
T_Y=\frac{1}{d(A)}
\vcenter{\xy 0;/r.9pc/: (0,0)="o";(-1,1.4)**\dir{}, "o",
{\ellipse(1.3,2)^,=:a(340){-}},
{\ar@{.}@`{(.7,2.3), (-1.2,2.3)} (-.5,-1.6); (-1.2, 3) },
{\ar@{.}@`{(-.7,-2), (-1,-3)} (-.7,-2); (-1.3, -3) },
{\ar@{.}@`{(-.7,1)} (-1.3,-.3); (-.1, -.3) },
(-1.5,1.1)*+{\scriptstyle Y},
\endxy} 
$$ of that defined by Riesen in \cite{Rie2025} and by Kirillov in \cite{K02} up to a scalar. It follows from Proposition \ref{p:diff_form} that Lemma \ref{l:rho_dual} is just another proof for Theorem 5 and part of Theorem 6 in \cite{Rie2025}.

\section{Schur-Weyl duality}
We continue our development of the fusion action of $\CA$ on the condensable algebra $A$ in a modular tensor category $\CC$, and prove the categorical Schur-Weyl duality in Theorem \ref{t1}.

Recall that for any MTC $\MM$, the irreducible characters $\chi_X$ of $K(\MM)$ are 1-dimensional and  parametrized by $X \in \irr(\MM)$. More precisely, $\chi_X(Y) = \frac{S^\MM_{X,Y}}{d^\MM(X)}$ for all $Y \in \irr(\MM)$ defines a character if $K(\MM)$ by the Verlinde formula, where $d^\MM(X)$ denotes the pivotal dimension of $X$ in $\MM$ and $S^\MM$ is the (unnormalized) $S$-matrix of $\MM$. In particular, $\chi_\1(Y) = d^\MM(Y)$. Thus, the indecomposable central idempotent $e^\MM_X$ of $K(\MM)$ are also indexed by $X \in \irr(\MM)$ so that
$u\, e^\MM_X = \chi_X(u) \,e^\MM_X$ for $u \in K(\MM)$. These idempotents can be expressed in terms of the  $S$-matrix of $\MM$, namely
$$
e^\MM_X = \frac{d^\MM(X)}{\dim(\MM)}\sum_{Y \in \irr(\MM)} S^\MM_{X,Y} Y^*\,,
$$
and $\{e_X^\MM\mid X \in \irr(\MM)\}$ is the set of primitive idempotents of $K(\MM)$.

If $\BB$ is a spherical fusion category,  then $\MM = Z(\BB)$ is an MTC \cite{MugerSF2}. It follows from \cite[Lem. 8.49]{ENO} that the forgetful functor $F: \MM \to \BB$ induces an algebra homomorphism $\ul F : K(\MM) \to K(\BB)$, and its image is the center of $K(\BB)$. By \cite[Thm. 2.13]{O15}, $e^\MM_X \in \ker \ul F$ if and only if $[F(X): \1]_\BB =0$. 

Now, we return to our setting with $\BB= \CA$. By \cite{DMNO} or \eqref{eq:decomp},  $Z(\CC_A)$ is equivalent to $\CC \boxtimes \AA$ as spherical braided fusion categories, where $\AA = (\CAo)^{\rev}$. Thus, 
$$
\irr(Z(\CC_A)) = \{\a^+(x)\oa Y \mid x \in \irr(\CC), Y \in \irr(\AA)\} \,.
$$
and  the set of indecomposable central idempotents of $K(\CA)$ is given by 
\begin{equation}\label{e1}
\{e_{x}^{\CC}e_Y^{\AA} \mid x \in \irr(\CC) \text{ and } Y \in \irr(\CAo) \text{ such that } [\a(x) \oa Y,\1]_{\CC_A} \ne 0\}.
\end{equation}

Now, we can explicitly state and  prove our Theorem I as follows.
\begin{thm}[Schur-Weyl duality]\label{t1} Let $\CC$ be an MTC and $A \in \CC$ a condensable algebra.
Let $\AA = (\CAo)^{\rev}$ and $\rho: K(\CA) \to \CC(A,A)$ the $\BC$-algebra homomorphism defined in Definition \ref{def:rho}. Then
\begin{enumeri} 
  \item   $\ker \rho$ is the ideal of $K(\CA)$ generated by $1 -e_\1^\AA.$
  \item   $\CC(x, A)$ is an irreducible $K(\CA)$-module for any $x \in \irr(\CC)$ with $[x,A]_\CC \ne 0.$
  \item  For any $x, y\in \iC$, $\CC(x, A) \cong \CC(y, A)$ as nonzero $K(\CA)$-modules if and only if $x = y.$
  \item  The restriction $\rho : e_\1^\AA K(\CA) \to \CC(A,A)$ is an isomorphism of $\BC$-algebras.
\end{enumeri}
\end{thm}
\begin{proof}
    (i) Note that $e_\1^\AA := \frac{1}{\dim(\AA)}\sum_{X \in \irr(\AA)} d_A(X) X$ is an indecomposable idempotent of $K(\AA)$ corresponding the character $\chi_\1(X) = d_A(X)$ for $X \in \irr(\AA)$.
    By Proposition \ref{p:CAo}, $\rho(X) = d_A(X)\id_A$ for $X \in \irr(\AA)$, and so we find $\rho(e^\AA_\1)  = \id_A$. Thus $1-e^\AA_\1 \in \ker \rho$ and so the ideal of $K(\CA)$ generated by $1-e_\1^\AA$ is contained in $\ker \rho$.
Since $e_\1^\AA$ is a central element of $K(\CA)$, $(1-e^\AA_\1)K(\CA)$ and 
$e^\AA_\1 K(\CA)$ are  two sided ideal and $K(\CA)=(1-e^\AA_\1)K(\CA)\oplus e^\AA_\1 K(\CA).$ For convenience we set $\ol K=e^\AA_\1 K(\CA).$

From  (\ref{e1}) on the indecomposable central idempotents of 
$K(\CA)$, we see that the indecomposable central idempotents $\ol K$ are given by
\begin{equation}\label{e2}
\{e^\AA_\1 e_{x}^{\CC} \mid x \in \irr(\CC) \text{ such that } [\a(x),\1]_{\CC_A} \ne 0\}.
\end{equation}
To complete the proof of this statement, it suffices to show that $\rho(e_\1^\AA e_x^{\CC})\ne 0$ for $[\a(x),\1]_{\CC_A} \ne 0$ or $[x, A]_\CC \ne 0$. 

    Let $x \in \irr(\CC)$ such that $[x, A]_\CC \ne 0$. For any $y \in \irr(\CC)$ and  $f \in \CA(\a(x), \1)$,
    \begin{equation} \label{eq:indicator(a(y))}
       \a(y)f =\hspace{-5pt}{\color{blue}
\vcenter{\xy 0;/r1pc/:
(0,0)*=<10pt,10pt>[F]{\scriptstyle f},
{\ar@{-}@`{(-.3,1), (-2,.7)} (0,.45);(-1.5,2.2)*+{\color{black}\scriptstyle \a(x)}},
(0,0)="o";(-1.4,.7)**\dir{}, "o",
{\ellipse(1.3,1.5)^,=:a(340){-}},
(-2,-.5)*+{\color{black}\scriptstyle \a(y)},
\endxy}
}\hspace{15pt}
={\color{blue}
\vcenter{\xy 0;/r1pc/:
(-1,-2.5)*=<10pt,10pt>[F]{\scriptstyle f},
{\ar@{-}@`{(-1, -1.9),  (.4,.2), (-1.2,.6)}|(0.25){\hole} (-1,-2.1);(-1.5,2.2)*+{\color{black}\scriptstyle \a(x)}},
(0,0)="o";(-1.4,.7)**\dir{}, "o",
{\ellipse(1.3,1.5)^,=:a(340){-}},
(-2.1,-.5)*+{\color{black}\scriptstyle \a(y)},
\endxy}}
\hspace{20pt} =\, \frac{S^\CC_{x^*,y}}{d(x)} f
    \end{equation}
since $\a(y)$ admits a half-braiding (cf. \eqref{eq:S-matrix} for the convention of the $S$-matrix). Therefore, for $z \in \iC$,
$$
e_{z}^{\CC} f = \frac{d(z)}{\dim(\CC)}\sum_{y \in \irr(\CC)} S^\CC_{z,y} \a(y^*) f =
\frac{d(z)}{\dim(\CC)}\sum_{y \in \irr(\CC)} \frac{S^\CC_{z,y} S^\CC_{x^*,y^*}}{d(x)} f = \delta_{z,x^*}f\,.
$$
Thus, $e_{x}^{\CC}$ acts as identity on $\CA(\a(x^*), \1)$ or $\CC(x^*, A)$ by Proposition \ref{p:interwiner}.  

For any $x \in \irr(\CC)$, let $B^x = \{\iota_{x,i}\}_i$ be a basis for $\CC(x, A)$, and  $B_x=\{\pi_{x,i}\}_i$ the dual basis of $B^x$ for $\CC(A, x)$ relative to the nondegenerate bilinear form $ \CC(A, x) \times \CC(x, A) \to  \CC(x,x)$ defined by composition, i.e. $\pi_{x,i} \circ \iota_{x,j} = \delta_{i,j} \id_x$. Then  $P_x=\sum_i \iota_{x,i} \circ \pi_{x,i} \in \CC(A,A)$ is an idempotent, and the image of $P_x$ is the isotypic component of $x$ in $A$. Moreover, $\id_A = \sum_{x \in \irr(\CC)} P_x$. In particular, 
$$
\rho(e_\1^\AA e_{x}^{\CC}) = \rho(e_\1^\AA) \circ \rho(e_{x}^{\CC}) = e_{x}^{\CC} \id_A = P_{x^*} \ne 0
$$
for all $x \in \irr(\CC)$ with $[x,A]_\CC \ne 0$. Thus, the restriction $\rho|_{\ol K}$  is injective or $\ker \rho = (1-e_\1^\AA )K(\CA)$.

(ii) Let $x \in \iC$. It follows from (i) that $\CA(\a(x), \1)$ is a (complex) representation of $K(\CA)$ such that all but one indecomposable central idempotents of $K(\CA)$ acts trivially on $\CA(\a(x), \1)$ and $e_\1^\AA e_{x^*}^{\CC}$ acts as identity. Let $M_x$ be a simple $K(\CA)$-module such that $e_\1^\AA e_{x^*}^{\CC}$ acts as identity of $M_x$. Then $\CA(\a(x), \1) \cong n_x M_x$ for some positive integer $n_x$, and 
$\a(y) f = \frac{S_{x^*, y}}{d(x)} f$ for $y \in \iC, f \in \CA(\a(x), \1)$. 

By \cite{Lu03}, the element 
$$
\phi_{M_x} = \sum_{Y \in  \iCA} \Tr(Y, M_x)Y^* \in K(\CA)
$$
is central and it acts as a scalar $f_{M_x}$ on $M_x$ and zero on any other irreducible $K(\CA)$-modules. The scalar $f_{M_x}$, called  the \emph{formal codegree} of $M_x$, was proved in \cite[Thm. 2.13]{O15} that
$$
f_{M_x} = \frac{\dim(\CC_A)}{d_A(\a(x^*))} = \frac{\dim(\CC)}{d(x)d(A)}\,.
$$

Let  $W_x$ denote the $K(\CA)$-module $\CA(\a(x),\1)$.  Then 
$$
\phi_{W_x} = \sum_{Y \in  \iCA} \Tr(Y, W_x)Y^* = n_x \phi_{M_x}
$$
and it acts on $W_x$ as the scalar $n_x f_{M_x}$, and so 
\begin{equation}\label{eq:trace1}
    \Tr(\phi_{W_x}, W_x) = n_x f_{M_x} [x, A]_\CC = n_x \frac{\dim(\CC) [x, A]_\CC}{d(A)d(x)}\,.
\end{equation}

Recall from Section \ref{s:indicator} that $\Tr(Y, W_x)= \Tr(E_{\a(x), Y}^{(0,-1)}) = \nu_{(0,-1)}^{\a^+(x)}(Y)$ where $E_{\a(x), Y}^{(0,-1)}$ is the endomorphism on  $\CC_A(\a(x),\1)$ defined in terms of diagram \eqref{eq: indicator_diag}.  We proceed to compute $\Tr(\phi_{W_x}, W_x)$ by using properties of generalized Frobenius-Schur indicators $\nu_{(n,k)}^{\a^+(x)}(Y)$, which was synopsized in Section \ref{s:indicator}. The readers are referred to \cite{NS10} for more detail on their arithmetic properties, which are briefly summarized as follows.

The modular group $\SLZ$,  generated by $\fs = \mtx{0&-1\\ 1 & 0}$ and $\ft= \mtx{1 & 1\\ 0 & 1}$, admits a  canonical action on $K(Z(\CC_A))$. Since $Z(\CA)$ is equivalent to $\CC \boxtimes \AA$ as braided spherical fusion categories, the simple objects of $Z(\CA)$  are indexed by the pairs $(x, X) \in \iC \times \iA$. The pair $(x,X) \in \iC \times \iA$ is identified with the simple object $\a^+(x) \boxtimes X$ in $Z(\CC_A)$ (cf. \eqref{eq:decomp}).

Let $\BS$ and $\BT$ be the $S$ and $T$ matrices of $Z(\CC_A)$. They are indexed by the set $\iC \times \iA$, and 
$$
\BS_{(x, X), (y,Y)} = S^\CC_{x,y} S^\AA_{X,Y} \quad\text{and}\quad \BT_{(x, X), (y,Y)} = \delta_{X,Y} \delta_{x,y} \theta_x^{\CC} \theta^\AA_X 
$$
for any $(x, X), (y,Y) \in \iC \times \iA$, where $\theta^\CC_x$ and $\theta_X^\AA$ are the scalars of the components of the ribbon structures $\theta^\CC$ and $\theta^\AA$. 

For simplicity, we will write $\a^+(x) X$ for $\a^+(x) \boxtimes X$ in $Z(\CA)$ and $\a(x) X$ for $\a(x) \oa X$ in $\CA$. Then, the $\SLZ$-action on $K(Z(\CC_A))$ is given by
\begin{align}
   \fs \cdot \a^+(x) X & =\sum_{\substack{y \in \iC \\ Y \in \iA}}   \frac{S^\CC_{x, y} S^\AA_{X, Y}}{\dim(\CC_A)} \, \a^+(y) Y, \label{eq:s-action} \\
    \ft \cdot \a^+(x)X & =  \theta_x^\CC \theta_X^\AA \,\a^+(x)X\,. \label{eq:t-action}
\end{align}
For any $(n,k) \in \BZ^2$, the generalized Frobenius-Schur indicator $\nu_{(n,k)}^{\bX}(Y)$, defined for any objects $\bX$  in $Z(\CA)$ and $Y \in \CC_A$, is additive in $\bX$ and $Y$.  Based on the discussion for \eqref{eq:(n,0)},
\begin{equation}\label{eq:ind(1,0)}
    \nu_{(1,0)}^{\bX}(Y) = [X, Y]_{\CA}  \quad\text{for } \bX=(X, \s_{X,-}) \in Z(\CA)\,.
\end{equation}

One can extend $\nu_{(n,k)}^{\bX}(Y)$ linearly to a functional $\nu_{(n,k)}^{z}(Y)$  for $z \in K(Z(\CC_A))$. By \cite[Prop. 3.3, Lem. 5.3 (i)]{NS10}, we  have
\begin{equation} \label{eq:dual}
    \nu_{(n,k)}^{\fs^2\cdot\bX}(Y) =\nu_{(n,k)}^{\bX^*}(Y) = \nu_{(n,k)}^{\bX}(Y^*) 
\end{equation}
for $\bX \in Z(\CC_A)$. Moreover, it follows by \cite[Thm. 5.4]{NS10} that
\begin{equation} \label{eq:equivariance}
\nu_{(n,k)}^{\fg \cdot z}(Y) = \nu_{(n,k)\tilde \fg}^{z}(Y)
\end{equation}
for any $\fg \in \SLZ$, where $\tilde \fg = \fj \fg \fj$ with  $\fj = \mtx{1 & 0 \\ 0 & -1}$ and $(n,k)\tilde{\fg}$ denotes the matrix multiplication.

By the definition of $\nu_{(0,-1)}^{\a^+(x)}(Y)$, we first rewrite 
$$
\phi_{W_x} = \sum_{Y \in \iCA} \nu_{(0,-1)}^{\a^+(x)}(Y) Y^*
$$
and so
\begin{align*}
    \Tr(\phi_{W_x}, W_x) =& \sum_{Y \in \iCA}\nu_{(0,-1)}^{\a^+(x)}(Y) \Tr(Y^*, W_x) = \sum_{Y \in \iCA}\nu_{(0,-1)}^{\a^+(x)}(Y)\cdot \nu_{(0,-1)}^{\a^+(x)}(Y^*) \\
    \stackrel{\text{by }\eqref{eq:dual}}{=}& \sum_{Y \in \iCA}\nu_{(0,-1)}^{\a^+(x)}(Y)\cdot \nu_{(0,-1)}^{\a^+(x)^*}(Y) = \sum_{Y \in \iCA}\nu_{(0,-1)}^{\a^+(x)}(Y)\cdot \nu_{(0,-1)}^{\fs^2\cdot \a^+(x)}(Y)
    \end{align*}
\begin{align*}
    = &  \sum_{Y \in \iCA}\nu_{(0,-1)}^{\a^+(x)}(Y)\cdot \nu_{(1,0)}^{\fs\cdot \a^+(x)}(Y)  \quad\text{by  } \eqref{eq:equivariance}\\
    = &  \sum_{\substack{Y \in \iCA\\ y \in \iC, X \in \irr(\AA)}}\nu_{(0,-1)}^{\a^+(x)}(Y)\cdot \frac{S^\CC_{x,y} S^\AA_{\1,X}}{\dim(\CC_A)}\nu_{(1,0)}^{\a^+(y)  X}(Y) \quad\text{by  } \eqref{eq:s-action}\\
    = & \sum_{\substack{Y \in \iCA\\ y \in \iC, X \in \irr(\AA)}}\nu_{(0,-1)}^{\a^+(x)}(Y)\cdot \frac{S^\CC_{x,y} S^\AA_{\1,X}}{\dim(\CC_A)}[\a(y)X, Y]_{\CC_A} \quad\text{by  } \eqref{eq:ind(1,0)}\\
    = & \sum_{\substack{X \in \irr(\AA)\\ y \in \iC}}\nu_{(0,-1)}^{\a^+(x)}(\a(y)X)\cdot \frac{S^\CC_{x,y} S^\AA_{\1,X}}{\dim(\CC_A)}\\
    = & \sum_{\substack{X \in \irr(\AA)\\ y \in \iC}}\nu_{(0,-1)}^{\a^+(x)}(\a(y))\cdot \frac{ d_A(X)S^\CC_{x,y} S^\AA_{\1,X}}{\dim(\CC_A)}  \quad\text{by Proposition }\ref{p:CAo} \\
    = &   \sum_{y \in \iC}\frac{S^\CC_{x^*, y}}{d(x)} [x, A]_\CC \frac{ \dim(\AA)S^\CC_{x,y} }{\dim(\CC_A)} \quad\text{by }\eqref{eq:indicator(a(y))} \\
    = &\,\, \frac{ \dim(\CC)\dim(\AA)[x, A]_\CC}{d(x) \dim(\CC_A)} = \frac{ \dim(\CC)[x, A]_\CC}{d(x) d(A)} \,.
\end{align*}
It follows from \eqref{eq:trace1} that $n_x=1$, and hence $W_x$ is an irreducible $K(\CA)$-module.

(iii)  From the proofs of (i) and (ii), we know that,  for any simple object $x \in \CC$ with $[x,A]_\CC \ne 0$, $\CA(\a(x), \1)$ is a simple $K(\CA)$-module such that $e_\1^\AA e_{x^*}^{\CC}$ acts as identity on $\CA(\a(x), \1)$  and all other indecomposable central idempotents of $K(\CA)$ act trivially on it. It follows immediately that $\CA(\a(x), \1)$ and $\CA(\a(y), \1)$,  for any simple objects $x,y \in \CC$ such that $[x,A]_\CC \cdot [y,A]_\CC \ne 0$, are inequivalent $K(\CA)$-modules if $x\not\cong y.$   

(iv) By (i)-(iii), for $x \in \iC$ with $[x,A]_\CC \ne 0$, $P_x$ defined in the proof of (i) is an indecomposable central idempotent of $\CC(A,A)$, and the simple ideal of $\CC(A,A)$ generated by $P_x$ has dimension $[x, A]_\CC^2$.  Since $\rho(e_\1^\AA e_x^\CC) = P_{x^*}$, the assertion (iv)
follows from  \eqref{e2} on the indecomposable central idempotents 
of $\ol K$. \end{proof}
The following corollary  is a refinement of \cite[Cor. 2.16]{O15} and it also addresses the question thereafter.
\begin{cor}\label{c1}
Let $\BB$ be any spherical fusion category, $\CC=Z(\BB)$ the center of $\BB$, $I: \BB \to \CC$ the right adjoint of the forgetful functor from $\CC$ to $\BB$, and $A = I(\1)$.  Then 
$\rho: K(\BB) \to \CC(A,A)$ defined in \ref{def:rho} is an isomorphism of $\BC$-algebras. Moreover, for any $\bX = (X, \s_{X,-}) \in \CC$, 
the representation of $K(\BB)$, $Y \mapsto E_{\bX, Y}^{(0,-1)}$, on the vector space $\BB(X, \1)$ is irreducible if and only if $\bX$ is a simple object in $\CC$ and $[X, \1]_\BB >0$. Furthermore, every irreducible representation of $K(\BB)$ is of this form, and $\BB(X, \1)$ and $\BB(Y, \1)$ are equivalent irreducible $K(\BB)$-modules if and only if $\bX \cong \bY$ in $\CC$.
\end{cor}
\begin{proof}
    Since $A$ is a Lagrangian algebra in $\CC$, $\CAo = \Vec$ and $\CC_A \simeq \BB$ as spherical fusion categories. Now the result follows immediately from Theorem \ref{t1}. 
\end{proof}

By considering a certain central Hopf comonad, an isomorphism $\BB(A,\1) \cong \CC(A,A)$ of algebras is obtained in \cite[Thm. 3.8]{Shi}, which implies a canonical isomorphism $K(\BB) \cong \CC(A,A)$ \cite[Rmk. 4.5]{Shi}, without any assumption on the characteristic of the based field. 

\begin{rmk}{\rm
    For any condensable algebra $A$ in a modular tensor category $\CC$,  $A = \bigoplus_{x \in \iC} [x, A]_\CC\,  x$ as objects in $\CC$. One can identify $\CC(x,A)$ as the object $[x,A]_\CC \, \1$ in $\CC$ and write
    $$
    A = \bigoplus_{x \in \iC} \CC(x,A)\o x\,.
    $$
    Since $\CC(x,A)$ is an irreducible representation of $\ol K=e_\1^\AA{K(\CA)}$, the simple subobjects $x$ of $A$ are parametrized by the irreducible characters $\lambda$ of $\ol K$ and so 
    $$
    A = \bigoplus_{\lambda \in \irr(\ol K)}W_\lambda \o x_\lambda
    $$
    where $W_\lambda$ is an $\ol K$-module which affords the character $\lambda$. Thus, Theorem \ref{t1} is a categorical generalization of Schur-Weyl duality.  } 
\end{rmk}

  \section{Galois correspondence from fusion actions} \label{s:Gal_Cor}
  In this section, we prove the Galois correspondence between the condensable subalgebras $B$ of any given condensable algebra $A$ in a pseudounitary MTC $\CC$ and the fusion subcategories $\BB$ of $\CA$ containing $\CAo$.  In particular, we prove that $B=A^\BB$ is a condensable subalgebra of $A$, and $\BB \mapsto A^\BB$ defines a order-reversing  bijection between the fusion subcategories of $\CA$ containing $\CAo$ and the condensable subalgebras of $A$. Moreover, $\BB = (\CBo)_A$. Our approach is motivated by \cite{DM1997, Xu2014}, and is a generalization  of Galois correspondence \cite{DM1997, HMT1999, DJX2013, Xu2014} of orbifolds of vertex operator algebras, and that of intermediate conformal nets. 
  
  We would like to point out that such a correspondence was previously established in \cite[Thm. 4.1]{DMNO} for $\CC = Z(\CA)$ or $A$ is the Lagrangian algebra of $\CC$, where the subalgebra $B$ of $A$ corresponding to the fusion subcategory $\BB$ of $\CA$ is given by $B=I_\BB(\1)$ and $I_\BB$ is the right adjoint of the forgetful functor $F_\BB:\CC \to Z_\BB(\CA)$. This approach will be discussed in detail in Section \ref{s:relative}. In this section, we essentially provides an alternate approach for \cite[Rmk. 4.12]{DMNO} in the spirit of classical Galois correspondence. The remark \cite[Rmk. 4.12]{DMNO} was also proved in \cite[Thm. 3.5]{LKW}.

  Recall that every morphism $f: x \to y$ in an abelian category has a  factorization $f = \iota \pi$, in which $\pi:x \to z$ is an epimorphism and  $\iota:z \to y$ is a monomorphism for some object $z$ (cf. \cite[VIII.3]{MacLane}).  This \emph{epi-monic} factorization is unique up to isomorphism, which means that if $f= \iota' \pi'$ for some epimorphism $\pi': x \to z'$ and monomorphism $\iota': z' \to y$, there exists an isomorphism $h: z \to z'$ such that the following diagram commutes:
$$
\xymatrix{
x\ar@{=}[d]\ar[r]^-{\pi} & z \ar[d]^-{h} \ar[r]^-{\iota}& y \ar@{=}[d] \\
x \ar[r]_-{\pi'} & z'\ar[r]_-{\iota'}& y \,.
}
$$
The subobject $(z, \iota)$ of $y$ is called the \emph{image} of $f$, and the pair $(\pi, z)$ is called its \emph{coimage}.

  Let $A$ be a condensable algebra in an MTC $\CC$ with the algebra structure $(u_A, m_A)$, and $\BB$  a fusion subcategory of $\CA$ containing $\CAo$.  Recall that 
$$
e_\1^\BB = \frac{1}{\dim(\BB)}\sum_{Y \in \iB} d_A(Y) Y
$$
is a primitive central idempotent of $\KB$, but it may not be central in $K=\KA$. Since $\AA=(\CAo)^{\rev}$ is a fusion subcategory of $\BB$, $e_\1^{\AA} e_\1^{\BB} = e_\1^\BB$, and $\rho(e_\1^\BB)$  is an idempotent in $\CC(A,A)$,  which has a factorization $\rho(e_\1^\BB) = \iota \pi$,  where $\iota: B \to A$ the image of $\rho(e_\1^\BB)$ and $\pi: A \to B$ its coimage. The subobject $\iota: B \to A$ of $A$ will be denoted by $(B, \iota)$. It is immediate to see that  $B\cong \sum_{x \in \iC} \CC(x,A)^{\BB} \o x$ as objects in $\CC$, where $\CC(x,A)^{\BB}$ is identified with $\dim(\CC(x,A)^{\BB})\1$ in $\CC$.

Let $B, B' \in \CC$ be algebras with structures $(u_B, m_B)$ and $(u_{B'}, m_{B'})$ respectively. A morphism $h : B \to B'$ is called an algebra homomorphism if the following commutative diagrams hold:
\begin{equation} \label{eq:subalgebra}
     \xymatrix{
     B \ar[r]^-{h}  & B'\\   
     \1 \ar[u]^-{u_B} \ar[ru]_-{u_{B'}}&
    }, \qquad
    \xymatrix{
     B \o B \ar[d]_-{m_B}\ar[r]^-{h\o h}  & B'\o B' \ar[d]^-{m_{B'}}\\ 
     B \ar[r]^-{h} & B'
    }\,.
\end{equation}

A subobject $(B, \iota)$ of $A$ is called a subalgebra, if $B$ admits an algebra structure $(u_B, m_B)$ such that $\iota : B \to A$ is an algebra monomorphism. A condensable subalgebra $(B, \iota)$ of $A$ means  $B$ is also condensable algebra in $\CC$.

 \begin{defn} \label{d:lattices}
Let $\LCA$ be the lattice of fusion (full) subcategory of $\CA$ containing $\CAo$. For $\BB \in \LCA$, we define $A^\BB$ as the image $\iota: B \to A$ of $\rho(e_\1^\BB)$. We will prove in Proposition \ref{p:subalgebra} that $(B, \iota)$ is a subalgebra of $A$ with the unit map $u_B:=\pi u_A$ and multiplication $m_B:=\pi m_A(\iota \o \iota)$. We denote by $\BB_1  \le \BB_2$ if  $\BB_1$ is a fusion subcategory of $\BB_2$. Let $(B_1, \iota_1)$, $(B_2, \iota_2)$ be subalgebras of $A$. We will write  $(B_1, \iota_1) \le  (B_2, \iota_2)$ if  $\iota_1 = \iota_2 h $ for some algebra monomorphism $h: B_1 \to B_2$, and $(B_1, \iota_1) \equiv (B_2, \iota_2)$ if $h$ is an algebra isomorphism.  If $(B_1, \iota_1) \le  (B_2, \iota_2)$ and $(B_1, \iota_1) \not\equiv (B_2, \iota_2)$, we will write $(B_1, \iota_1) <  (B_2, \iota_2)$. The lattice of condensable subalgebras of $A$ is denoted by $\LaA$. Similarly, the partial order $(B_1, \iota_1) \le (B_2, \iota_2)$ can be defined for subobjects of $A$ if $\iota_1 =\iota_2 h$ for some monomorphism of $h:B_1 \to B_2$ of $\CC$. 
\end{defn}

Let $\BB \in \LCA$. We first prove Proposition \ref{p:subalgebra}, which asserts that $A^\BB$ is a subalgebra of $A$. Moreover, under the assumption $d(A^\BB) \ne 0$,  $A^\BB$ is condensable. 

For $y \in \CC$, one can restrict the action of $\CA$ to $\BB$ on $\CC(y,A)$, and we define
$$
\CC(y,A)^\BB = \{ f \in \CC(y,A) \mid Y f = d_A(Y) f \text{ for all }Y \in \BB\},
$$
which is the isotypic component of $\CC(y,A)$ of the dimension character of $\BB$. Note that  $Yf = (Y\id_A)\circ f$ for any $f \in \CC(y,A)$. Thus, for any $g: y \to y'$ and $f \in \CC(y', A)^\BB$,
$$
Y(f \circ g) = ((Y \id_A)\circ f)\circ g = d_A(Y)  f\circ g  
$$
for any $Y \in \BB$. So, $\CC(-, A)^\BB : \CC \to \vec$ is a contarvariant linear functor.

The following lemma provides another definition of $A^\BB$ by the Yoneda lemma.
\begin{lem} \label{l:inv1} Let $\BB \in \LCA$, $\rho(e_\1^\BB) = \iota\pi$ the epi-monic factorization of $\rho(e_\1^\BB)$, $A^\BB = (B,\iota)$ as a subobject of $A$.
    For any $y \in \CC$, $\phi_y: \CC(y, A)^\BB \to \CC(y,A^\BB)$ is a natural isomorphism, where $\phi_y(f) = \pi f$. Moreover, $\iota \in \CC(B,A)^\BB$, $\pi \rho(e_\1^\BB) = \pi$ and $\phi$ is an isomorphism of linear functors.
\end{lem}
\begin{proof}
    For any $Y \in \BB$ and $f \in  \CC(y, A)^\BB$, $(Y\, \id_A)\circ f = Yf=d_A(Y)\ f$. Thus,
    $$
    \rho(e_\1^\BB)\circ  f = \frac{1}{\dim(\BB)}\sum_{Y \in \iB} d_A(Y) (Y\, \id_A)\circ f = \frac{1}{\dim(\BB)}\sum_{Y \in \iB} d_A(Y)^2  f  = f\,. 
    $$
    Since $\rho(e_\1^\BB) = \iota  \pi$, $f = \iota \pi f = \iota \phi_y(f)$, which implies $\phi_y$ is injective. The $\BC$-linearity of $\phi_y$ is obvious.  

    We claim $\phi_y^{-1}(g) = \iota g$ for any $g \in \CC(y,B)$. Since $\rho(e_\1^\BB)^2 = \rho(e_\1^\BB)$, $\iota \pi \iota \pi = \iota \pi$. Note that $\iota$ is a monomorphism and $\pi$ is an epimorphism. Therefore, 
    \begin{equation}\label{eq:projection}
        \iota \pi \iota=\iota, \quad \pi \iota \pi = \pi \quad \text{and}\quad  \pi \iota = \id_B.
    \end{equation}
     The first equality means $\rho(e_\1^\BB)\iota = \iota$ and hence $\iota \in \CC(B,A)^\BB$ and so $\iota g \in \CC(y,A)^\BB$ for any $g \in \CC(y,B)$. It follows from the third equality of \eqref{eq:projection} that  $\phi_y(\iota g) = g$ for any $g \in \CC(y,B)$, which proves the claim. Moreover, the second equality of \eqref{eq:projection} implies $\pi \rho(e_\1^\BB) = \pi$.

    For $h \in \CC(y,y')$ and $f \in \CC(y',A)^\BB$,
    $$
    \phi_{y'}(f) h = (\pi f) h =\pi (fh) = \phi_y(fh)\,,  
    $$
    which means $\phi_y$ is a natural in $y$.
\end{proof}

\begin{lem} \label{l:image_psi}
    For any $x_i  \in \CC$ and $f_i \in \CC(x_i,A)$, $i=1,2$, we have 
    $$ m_A \circ (f_1 \o f_2) \in \CC(x_1 \o x_2, A)\quad \text{and}$$
    $$
    \psi(m_A\circ (f_1 \o f_2)) = \psi(f_1) \oa\psi(f_2) \in \CA(\a(x_1)\oa \a(x_2), \1),
    $$
    where $\a(x_1)\oa \a(x_2)$ is identified with $\a(x_1 \o x_2)$. Moreover, if $f_1, \cdots, f_\ell \in \CC(x_1, A)$ are linearly independent, and $g_1, \dots, g_\ell \in \CC(x_2,A)$ are nonzero, then $$\sum_{i=1}^\ell m_A \circ (f_i \o g_i) \ne 0.$$
\end{lem}
\begin{proof} Let $p_A: X \o Y \to X \oa Y$ denote the cokernel which defines the tensor product $X \oa Y$ for $X,Y\in\CA.$ By the universal property of $p_A:\a(x_1) \o \a(x_2) \to \a(x_1) \oa \a(x_2)$, $\psi(f_1) \oa\psi(f_2)$ is determined by the commutative trapezoid on the right side of the following diagram:
$$
\xymatrix{
\a(x_1\o x_2) &\a(x_1) \o \a(x_2) \ar[l]_{p'}\ar[rr]^-{\psi(f_1) \o \psi(f_2)} \ar[d]_-{p_A}  & &A \o A \ar[d]^-{p_A} \ar[r]^-{m_A}& A\\
&\a(x_1) \oa \a(x_2)\ar[lu]^-{\tilde p'} \ar[rr]^-{\psi(f_1) \oa \psi(f_2)} & &A \oa A \ar@{=}[ru] &
}\,.
$$
 The morphism $p': \a(x_1) \o \a(x_2) \to \a(x_1\o x_2)$ on the left side of the diagram is defined by
$$
p' = 
 \vcenter{\xy  0;/r1pc/:
{\ar@{-} (-1,2)*+{\scriptstyle x_1};(-1,-1.5)},
{\ar@{.}@`{(0,0)}|(.55){\hole} (0,2)*+{\scriptstyle A}; (2,-.5)},
{\ar@{-}@`{(1,-.5), (0,.5)}(1,2)*+{\scriptstyle x_2}; (0,-1.5)},
{\ar@{.} (2,2)*+{\scriptstyle A}; (2,-1.5)},
\endxy},
$$
which satisfies the universal property of $\a(x_1) \oa \a(x_2)$. Thus, there exists an isomorphism $\tilde p':\a(x_1) \oa \a(x_2) \to \a(x_1 \o x_2)$ such that the left triangle of the preceding diagram commutes. To complete the proof, it suffices to show that
$$
\psi(m_A(f_1 \o f_2)) \circ p' = m_A \circ (\psi(f_1) \o \psi(f_2)).
$$
Recall that $\psi(f_i) \in \CA(x_i \o A, A)$ is given by
    $$
    \psi(f_i) = \vcenter{\xy  0;/r1pc/:
(0,0)*=<10pt,10pt>[F]{\scriptstyle f_i},
{\ar@{.} (0,-.4);(0,-2.5)*+{\scriptstyle A}},
{\ar@{-} (0,.4);(0,2)*+{\scriptstyle x_i}},
{\ar@{.}@`{(1.2,-1)} (1,2)*+{\scriptstyle A}; (0,-1.5)},
\endxy}, 
    $$
 and so  
 \begin{align*}
 m_A \circ (\psi(f_1) \o \psi(f_2)) &=
  \vcenter{\xy  0;/r1pc/:
{\ar@{-} (-1,2)*+{\scriptstyle x_1};(-1,.9)},
(-1,.5)*=<10pt,10pt>[F]{\scriptstyle f_1},
{\ar@{.} (-1,0);(-1,-2)},
{\ar@{.}@`{(0,-.3)} (0,2)*+{\scriptstyle A}; (-1,-.5)},
{\ar@{-} (1,2)*+{\scriptstyle x_2};(1,.9)},
(1,.5)*=<10pt,10pt>[F]{\scriptstyle f_2},
{\ar@{.} (1,0);(1,-.5)},
{\ar@{.}@`{(2,-.3)} (2,2)*+{\scriptstyle A}; (1,-.5)},
{\ar@{.}@`{(.8,-1.2), (-.3,-1.4)} (1,-.5); (-1,-1.6)},
\endxy} =
\vcenter{\xy  0;/r1pc/:
{\ar@{-} (-1,2)*+{\scriptstyle x_1};(-1,.9)},
(-1,.5)*=<10pt,10pt>[F]{\scriptstyle f_1},
{\ar@{.} (-1,0);(-1,-2)},
{\ar@{.}@`{(0,-.3)} (0,2)*+{\scriptstyle A}; (0.5,-1.1)},
{\ar@{-} (1,2)*+{\scriptstyle x_2};(1,.9)},
(1,.5)*=<10pt,10pt>[F]{\scriptstyle f_2},
{\ar@{.} (1,0);(1,-.5)},
{\ar@{.}@`{(2,-.3)} (2,2)*+{\scriptstyle A}; (1,-.5)},
{\ar@{.}@`{(.8,-1.2), (-.3,-1.4)} (1,-.5); (-1,-1.6)},
\endxy} =
\vcenter{\xy  0;/r1pc/:
{\ar@{-} (-1,2)*+{\scriptstyle x_1};(-1,.9)},
(-1,.5)*=<10pt,10pt>[F]{\scriptstyle f_1},
{\ar@{.} (-1,0);(-1,-2)},
{\ar@{.}@`{(0,-.4)} (0,2)*+{\scriptstyle A}; (1,-.3)},
{\ar@{-} (1,2)*+{\scriptstyle x_2};(1,.9)},
(1,.5)*=<10pt,10pt>[F]{\scriptstyle f_2},
{\ar@{.} (1,0);(1,-.5)},
{\ar@{.}@`{(2,-.3)} (2,2)*+{\scriptstyle A}; (1,-.6)},
{\ar@{.}@`{(.8,-1.2), (-.3,-1.4)} (1,-.5); (-1,-1.6)},
\endxy}\\
&=\,
\vcenter{\xy  0;/r1pc/:
{\ar@{-} (-1,2)*+{\scriptstyle x_1};(-1,.9)},
(-1,.5)*=<10pt,10pt>[F]{\scriptstyle f_1},
{\ar@{.} (-1,0);(-1,-2)},
{\ar@{.}@`{(-.2,0)} (0,2)*+{\scriptstyle A}; (.9,0)},
{\ar@{-} (1,2)*+{\scriptstyle x_2};(1,1.1)},
(1,.7)*=<10pt,10pt>[F]{\scriptstyle f_2},
{\ar@{.} (1,.3);(1,-.5)},
{\ar@{.}@`{(2,-.3)} (2,2)*+{\scriptstyle A}; (1,-.5)},
{\ar@{.}@`{(.8,-1.2), (-.3,-1.4)} (1,-.5); (-1,-1.6)},
{\ar@{.}@`{(1.4,0), (1.4,-.3)} (1.1,0); (1,-.3)},
\endxy} =
\vcenter{\xy  0;/r1pc/:
{\ar@{-} (-1,2)*+{\scriptstyle x_1};(-1,.2)},
(-1,-.2)*=<9pt,9pt>[F]{\scriptstyle f_1},
{\ar@{.} (-1,-.6);(-1,-2)},
{\ar@{.}@`{(-.2,.7)}|(0.6){\hole} (0,2)*+{\scriptstyle A}; (1.6,.2)},
{\ar@{-}@`{(1,.5), (0, 1)} (1,2)*+{\scriptstyle x_2};(0,.2)},
(0,-.2)*=<9pt,9pt>[F]{\scriptstyle f_2},
{\ar@{.} (0,-.6); (0,-1.2)},
{\ar@{.}@`{(2,-.8)} (2,2)*+{\scriptstyle A}; (-1,-1.6)},
\endxy} =\psi(m_A(f_1 \o f_2)) \circ p'\,. 
 \end{align*}
 
  If $f_1, \cdots, f_\ell \in \CC(x_1, A)$ are linearly independent, so are  $\psi(f_1), \cdots, \psi(f_\ell) \in \CA(\a(x_1), \1)$. There exists $\tf_1, \cdots, \tf_\ell \in \CA(\1, \a(x_1))$ such that $\psi(f_i) \circ \tf_j = \delta_{i,j} \id_\1$. If 
 $$\sum_{i=1}^\ell m_A \circ (f_i \o g_i) = 0$$
for any $g_1, \cdots, g_\ell \in \CC(x_2,A)$, then
$$
\sum_{i=1}^\ell \psi(f_i)\oa \psi(g_i) = 0 
$$
and hence
$$
\psi(g_j) = \sum_{i=1}^\ell (\psi(f_i)\oa \psi(g_i))\circ (\tf_j \oa \id_{\a(x_2)}) = 0 
$$
for $j=1, \cdots, \ell$. This implies $g_j=0$ for all $j$, and completes the proof of the last assertion.
\end{proof}

Since $\CA$ is a spherical fusion category, for any $X,Y \in \CA$, the pairing $\langle\cdot, \cdot\rangle:\CA(Y,X) \o \CA(X,Y) \to \BC$
defined by the composition of morphisms and taking categorical trace, i.e. $f \o g \mapsto \langle f, g \rangle:=\tr(f\circ g)$ is a nondegenerate bilinear form. Now, we apply this to prove the following lemma. 

\begin{lem} \label{l:inv2}
    For any $x_i \in \CC$, $f_i \in \CC(x_i,A)^\BB$, $i=1,2$, we have 
    $$
    m_A\circ (f_1 \o f_2) \in \CC(x_1 \o x_2,A)^\BB.
    $$
\end{lem}
\begin{proof}
   By Lemma \ref{l:image_psi},
    $$
    \psi(m_A \circ(f_1 \o f_2)) =  \psi(f_1)\oa \psi(f_2) \in \CA(\a(x_1) \oa \a(x_2), \1),
    $$
    where $\psi(f_i) \in \CA(\a(x_i), \1)^\BB$, $i=1,2$.  We claim that $Y\cdot(\psi(f_1)\oa\psi(f_2))=d_A(Y)\psi(f_1)\oa\psi(f_2)$ for all $Y \in \BB$. For any $g_i \in \CA(\1, \a(x_i))$,
    $$
    \langle  Y \psi(f_i), g_i \rangle = d_A(Y) \langle \psi(f_i), g_i \rangle.
    $$
    In terms of diagram, we have
    \begin{equation*}
    \langle Y\psi(f_i),  g_i\rangle =
{\color{blue}
\vcenter{\xy  0;/r1pc/:
(0,0)="o";(-1.4,1)**\dir{}, "o",
{\ellipse(1.5,1.8)^,=:a(345){-}},
(0,0)*=<20pt,10pt>[F]{\scriptstyle \psi(f_i)},
{\ar@{-}@`{(-.1,1.5), (-2,1.1)} (0,.4);(-1.5,2.5)*=<10pt,10pt>[F]{\scriptstyle g_i}},
(-1.8,-1)*{\color{black}\scriptstyle Y},
\endxy}} \quad = d_A(Y)\!\!
{\color{blue}
\vcenter{\xy  0;/r1pc/:
(0,0)="o";(-1.4,1)**\dir{}, "o",
(0,0)*=<20pt,10pt>[F]{\scriptstyle \psi(f_i)},
{\ar@{-} (0,.4);(0,2.5)*=<10pt,10pt>[F]{\scriptstyle g_i}},
\endxy}}\quad= d_A(Y) \langle\psi(f_i),  g_i\rangle  \,.
\end{equation*}
Thus, for $Y \in \iB$, we have 
\begin{equation*}
{\color{blue}
\vcenter{\xy  0;/r1pc/:
(0,0)="o";(-1.4,1)**\dir{}, "o",
(0,-.5)*=<20pt,10pt>[F]{\scriptstyle \psi(f_i)},
{\ar@{-}@`{(-.1,1.5), (-2,1.1)} (0,-.1);(-1.5,2.5)*=<10pt,10pt>[F]{\scriptstyle g_i}},
{\ar@{-}@`{(1,1), (-2,1.5)}|(.55){\hole} (1,2.5)*+{\color{black}\scriptstyle Y};(-2,-.9)},
\endxy}} \quad = \langle \psi(f_i),  g_i\rangle 
{\color{blue}
\vcenter{\xy  0;/r1pc/: 
{\ar@{-} (0,2)*+ {\color{black}\scriptstyle Y};(0,-1.5)},
\endxy}} \quad \text{in } \BB,
\end{equation*}
and hence
\begin{align*}
& \langle Y(\psi(f_1) \oa \psi(f_2)),  g_1 \oa g_2\rangle =
{\color{blue}
\vcenter{\xy  0;/r1pc/:
(0,0)="o";(-1.4,1)**\dir{}, "o",
{\ellipse(2.5,2.5)^,=:a(320){-}},
(-1,-.5)*=<20pt,10pt>[F]{\scriptstyle \psi(f_1)},
(1,-.5)*=<20pt,10pt>[F]{\scriptstyle \psi(f_2)},
{\ar@{-}@`{(-1.1,1.5), (-2.5,1.1)} (-1,-.1);(-2.5,3.5)*=<10pt,10pt>[F]{\scriptstyle g_1}},
{\ar@{-}@`{(1,1), (-1.2,1)} (1,-.1);(-1.2,3.5)*=<10pt,10pt>[F]{\scriptstyle g_2}},
(-2.2,-2)*{\color{black}\scriptstyle Y},
(.1,.1)*{\ellipse(2.5,2.5)^,=:a(19){-}},
\endxy}} \quad = \langle \psi(f_1), g_1\rangle 
{\color{blue}
\vcenter{\xy  0;/r1pc/:
(0,0)="o";(-1.4,1)**\dir{}, "o",
{\ellipse(1.5,1.8)^,=:a(345){-}},
(0,0)*=<20pt,10pt>[F]{\scriptstyle \psi(f_2)},
{\ar@{-}@`{(-.1,1.5), (-2,1.1)} (0,.4);(-1.5,2.5)*=<10pt,10pt>[F]{\scriptstyle g_2}},
(-1.8,-1)*{\color{black}\scriptstyle Y},
\endxy}} \\
=&\, d_A(Y) \langle \psi(f_1),  g_1\rangle \langle \psi(f_2),  g_2 \rangle\,.  
\end{align*}
This implies 
$$Y(\psi(f_1) \oa \psi(f_2)) = d_A(Y) \psi(f_1) \oa \psi(f_2),$$
and so $\psi(f_1) \oa \psi(f_2) \in \CA(\a(x_1 \o x_2), \1)^\BB$.
\end{proof}
Since $A$ is a condensable algebra, $d(A) \ne 0$, $\theta_A = \id_A$ and the section $\e_A : A \to \1$ of $u_A$ defines the evaluation $\ev_A: A \o A \to \1$  given by $\ev_A=\e_A m_A$, which is the evaluation morphism for the left dual of $A$. In particular, $A=A^*$, and so $f^* \in \CC(A,A)$ for any $f \in \CC(A,A)$. 

Let us denote the morphism $\e_A$ as
$
 \vcenter{\xy  0;/r.6pc/:
{\ar@{.} (-1,1.1)*+{\scriptstyle A};(-1,-1.3)*{\circ}},
\endxy}
$. Then 
$$
\ev_A = 
\vcenter{\xy  0;/r1pc/:
{\ar@{.}@`{(-1,-.5),(1,-.5)} (-1,2)*+{\scriptstyle A}; (1,2)*+{\scriptstyle A}},
\endxy}\,=\,
\vcenter{\xy  0;/r1pc/:
{\ar@{.} (-1,2)*+{\scriptstyle A}; (0,.5)},
{\ar@{.} (1,2)*+{\scriptstyle A}; (0,.5)},
{\ar@{.} (0,.5); (0,-.5)*{\circ}},
\endxy}
$$
It is unclear whether $d(A^\BB) \ne 0$, but it has to be if $\CC$ has positive dimensions. Under the assumption that $d(A^\BB) \ne 0$,  $A^\BB$ is condensable as stated in the following lemma.
\begin{prop}\label{p:subalgebra}
   The subobject $A^\BB=(B, \iota)$ is a subalgebra of  $A$, with the unit map $u_B=\pi u_A$ and multiplication $m_B=\pi m_A (\iota\o \iota)$, where $\rho(e_\1^\BB) = \iota \pi$ is the epi-monic factorization.  In addition, if $d(B)\ne 0$, then $A^\BB$ is condensable, and hence $A^\BB \in \LaA$.
\end{prop}
\begin{proof}
 The hom-space $\CC(\1,A)$, spanned by the unit map $u_A: \1 \to A$, is an 1-dimensional representation of $K(\CC_A)$ associated with the dimension character, i.e. $Y u_A = d_A(Y) u_A$ for all $Y \in \CA$. In particular, $u_A \in \CC(\1, A)^\BB$ and $u_A = \rho(e_\1^\BB) u_A = \iota \pi u_A= \iota u_B$. 
 
 It follows from Lemmas \ref{l:inv1} and \ref{l:inv2}  that $m_A(\iota \o \iota) \in \CC(B \o B, A)^\BB$. So,
 \begin{equation} \label{eq:inv}
   m_A(\iota \o \iota) = \rho(e_\1^\BB)  m_A(\iota \o \iota) = \iota \pi m_A(\iota \o \iota)=\iota m_B\,. 
 \end{equation}
This proves $(B, \iota)$ is a subalgebra of $A$. 

    Since $A$ is a condensable algebra, and $\iota : B \to A$ is a subalgebra,  $\theta_B = \id_B$.  To prove that last assertion the lemma,  it suffices to show that $B$ is rigid in the sense of \cite{KO02} if  $d(B) \ne 0$. We need to show that $(\e_B m_B \o\id_{B^*})(\id_B \o \coev_B)$ is an isomorphism where $\e_B$ is the section of $u_B$. Since $u_B = \pi u_A$, $\e_B = \e_A \iota$. Note that $\e_B m_B = \e_A m_A (\iota \o\iota)$.      
    In terms of diagram, we have
    $$
    (\e_B m_B \o\id_{B^*})(\id_B \o \coev_B) =  
    \vcenter{\xy  0;/r1pc/:
    {\ar@{-} (-.7,3.5)*+{\scriptstyle B}; (-.7,1.9)},
    {\ar@{-}@`{(.9, 3.5), (2,3.5)} (.7,1.9); (2,-1)},
    (.5,2.6)*+{\scriptstyle B},
    (2.5,-.8)*{\scriptstyle B^*},
    (-.5,.3)*{\scriptstyle A},
    (-.7,1.5)*+=<9pt,9pt>[F]{\scriptstyle\iota},
    (.7,1.5)*+=<9pt,9pt>[F]{\scriptstyle\iota},
{\ar@{.} (-.7,1.1); (0,.3)},
{\ar@{.} (.7,1.1); (0,.3)},
{\ar@{.} (0,.3); (0,-.8)*{\circ}},
\endxy} = 
 \vcenter{\xy  0;/r1pc/:
    {\ar@{-} (-.7,3.5)*+{\scriptstyle B}; (-.7,1.9)},
    {\ar@{.}@`{(.9, 1.3), (1,2), (2,3)} (0,.3); (2,1.1)},
    (2,.7)*+=<9pt,9pt>[F]{\scriptstyle\iota^*},
    (.5,1.7)*+{\scriptstyle A},
    (2.5,-.8)*{\scriptstyle B^*},
    (-.5,.3)*{\scriptstyle A},
    (-.7,1.5)*+=<9pt,9pt>[F]{\scriptstyle\iota},
{\ar@{.} (-.7,1.1); (0,.3)},
{\ar@{-} (2,.35); (2,-.8)},
{\ar@{.} (0,.3); (0,-.8)*{\circ}},
\endxy} =
\vcenter{\xy  0;/r1pc/:
    {\ar@{-} (-.7,3.5)*+{\scriptstyle B}; (-.7,2.3)},
    (-.7,1.9)*+=<9pt,9pt>[F]{\scriptstyle\iota},
{\ar@{-} (-.7,1.5); (-.7,1.2)},
(-.7,.8)*+=<9pt,9pt>[F]{\scriptstyle\iota^*},
{\ar@{-} (-.7,.4); (-.7,-.9)*+{\scriptstyle B^*}},
\endxy}\,. 
    $$
    By Lemma \ref{l:rho_dual}, $\rho(e_\1^\BB)^* = \rho(e_\1^\BB)$, and so $\pi^* \iota^* = \iota \pi$. By the uniqueness of epi-monic factorization of $\rho(e_\1^\BB)$, there exists an isomorphism $h: B \to B^*$ such that 
    $h\pi = \iota^*$ and $h\pi^* =\iota$.  Thus
    $$
    \iota^*\iota = h\pi\iota = h \id_B = h\,.  
    $$
    Therefore, $B$ is rigid and hence condensable.
\end{proof}
Now we state and prove our major Theorem II.  
\begin{thm}[Galois correspondence] \label{t:Gal_Cor}
    Let $\CC$ be a pseudounitary modular tensor category, and $A$ a condensable algebra in $\CC$. Then the assignment $\LCA \to \LaA$, $\BB \mapsto A^\BB$, defines an order-reversing bijection, whose inverse is given by the assignment $\LaA \to \LCA$, $(B, \iota) \mapsto (\CBo)_A$. In particular, $\BB = (\CC_{A^\BB}^0)_A$ for any $\BB \in \LCA$ and 
    $$
    \dim(\BB)= \frac{\dim(\CC)}{d(A)d(B)}\,.
    $$
\end{thm}
\begin{proof}
    We divide the proof of Theorem \ref{t:Gal_Cor} into four major steps.\\    
    (i) For any $(B, \iota) \in \LaA$, we claim that $A^{(\CBo)_A }\equiv (B, \iota)$ as subalgebra of $A$.
    
    Let $\BB=(\CC_B^0)_A.$  Since $A$ is a condensable algebra in $\CBo$, we can apply the algebra isomorphism $\rho_B: e_\1^{\AA} \KB \to \CBo(A,A)$ that was established in Theorem \ref{t1}, where the $B$-module structure of $A$ is determined by $(B, \iota)$. In particular, $\rho_B(e_\1^\BB)$ is a primitive central idempotent of $\CBo(A,A)$. Let $\rho_B(e_\1^\BB) =\iota' \pi'$ be the epi-monic factorization of $\rho_B(e_\1^\BB)$ in $\CBo$. Since $\CBo(B,A)$ is 1-dimensional and both $\iota', \iota$ are monomorphism in $\CBo(B,A)$, $\iota'=\lambda \iota$ for some $\lambda \in \BC^\times$.  The uniqueness of the epi-monic factorization of $\rho(e_\1^\BB)$ implies that there exists an isomorphism $h:B \to B$ such that the diagram commutes:
   $$
   \xymatrix{
A\ar@{=}[d] \ar[r]^-{\pi} & B  \ar[r]^-{\iota}\ar[d]^-{h} & A \ar@{=}[d]\smallskip\\
 A \ar[r]^-{\pi'} & B  \ar[r]^-{\iota'} & A 
 }\,.
 $$
 Therefore, $h = \lambda^{-1} \id_B$. Let $(u_B', m_B')$ be the algebra structure of $(B, \iota')$ (cf. Prop. \ref{p:subalgebra}). It follows immediately from Prop. \ref{p:subalgebra} that $u_B' = \lambda^{-1} u_B$ and $m_B' = \lambda m_B$, and direct verification shows that $h:(B,\iota) \to (B,\iota')$ is an isomorphism of algebras. Therefore,  $(B,\iota) \equiv (B, \iota')$ as subalgebras of $A$. This completes the proof of the claim.\smallskip\\
 (ii) The assignment $(B, \iota) \mapsto (\CBo)_A$ is strictly order-reversing.
 
 If $(B_1, \iota_1) < (B_2, \iota_2)$ in $\LaA$, then $(\CC_{B_2}^0)_A \subseteq (\CC_{B_1}^0)_A$. However, 
 $$
 \FPdim((\CC_{B_2}^0)_A) = \frac{\FPdim(\CA)}{\FPdim(B_2)} < \frac{\FPdim(\CA)}{\FPdim(B_1)} = \FPdim((\CC_{B_1}^0)_A)\,,  
 $$
 and so $(\CC_{B_2}^0)_A \subsetneq (\CC_{B_1}^0)_A$. \smallskip\\
 (iii) The assignment $\BB \mapsto A^\BB$ is also strictly order-reversing.
 
Suppose $\BB_1 <  \BB_2$ in $\LCA$. Then $\Ki$ is a proper subalgebra of $\Kii$, and so $e_\1^\AA \Ki \subsetneq e_\1^\AA \Kii$.  In particular, $e_\1^{\BB_1}=e_1^\AA e_\1^{\BB_1} \ne e_\1^\AA e_\1^{\BB_2} = e_\1^{\BB_2}$ and $e_\1^{\BB_1}e_\1^{\BB_2}=e_\1^{\BB_2}$ in $\KA$. Since $\rho: e_\1^\AA \KA \to \CC(A,A)$ is an isomorphism of semisimple algebras, $\rho(e_\1^{\BB_1}) \ne \rho(e_\1^{\BB_2})$ and $\rho(e_\1^{\BB_1}) \rho(e_\1^{\BB_2}) = \rho(e_\1^{\BB_2})$. Let $A^{\BB_i} = (B_i, \iota_i)$ and $\pi_i$ the coimage of $\rho(e_\1^{\BB_i})$ for $i=1,2$. Then $\iota_1 \pi_1 \iota_2 \pi_2 = \iota_2 \pi_2$. Since $\pi_2$ is an epimorphism, we have $\iota_1 \pi_1 \iota_2  = \iota_2$. Let $h : =\pi_1 \iota_2$. The preceding equality implies $h: B_2 \to B_1$ is a monomorphism and $\iota_1 h = \iota_2$.

We claim that $h$ is an algebra monomorphism, with the algebra structures $(u_{B_i}, m_{B_i})$ defined in Definition \ref{d:lattices}. Since $\iota_2 u_{B_2}= u_A$, we have 
$$
h u_{B_2}  =  \pi_1 \iota_2  u_{B_2} =  \pi_1 u_A = u_{B_1}\,.
$$
Since $m_{B_i}= \pi_i m_A (\iota_i \o \iota_i)$, we find
$$
m_{B_1} (h \o h) =\pi_1 m_A (\iota_1 h \o \iota_1 h) =\pi_1 m_A (\iota_2  \o \iota_2) 
=\pi_1 \rho(e_\1^{\BB_2})m_A (\iota_2  \o \iota_2)
$$
where the last equality follows from \eqref{eq:inv} proved in Proposition \ref{p:subalgebra}. Thus,
$$
m_{B_1} (h \o h)  =\pi_1 \iota_2 \pi_2 m_A (\iota_2  \o \iota_2)= h m_{B_2}\,.
$$
This proves our claim. 

If $(B_1, \iota_1)\equiv (B_2, \iota_2)$, then $\FPdim(B_1)=\FPdim(B_2)$ or $h : B_2 \to B_1$ is an algebra isomorphism. Let $h'=\pi_2\iota_1$. Since $\iota_2\pi_2\iota_1\pi_1=\iota_2\pi_2$ and $\iota_2$ is a monomorphism,
we have $h' \pi_1 = \pi_2$ and
$$
h' h=\pi_2 \iota_1 \pi_1 \iota_2= \pi_2 \rho(e_\1^{\BB_1}) \iota_2 
=\pi_2 \iota_2=\id_{B_2}, 
$$
Therefore, $h' = h^{-1}$ and
$\iota_2 \pi_2 = \iota_2 h^{-1} \pi_1 = \iota_1 \pi_1$. This implies 
$$
\rho(e_\1^{\BB_1}) =  \rho(e_\1^{\BB_2})\,,
$$
which is a contradiction. Therefore, $(B_1, \iota_1)\not\equiv (B_2, \iota_2)$ or $A^{\BB_2} < A^{\BB_1}$ in $\LaA$.\smallskip\\
(iv) For $\BB \in \LCA$, we claim that $\BB = (\CBo)_A$ where $(B, \iota)=A^\BB$, and hence
    $$
    \dim(\BB) = \frac{\dim(\CC)}{d(B)d(A)} \,.
    $$
    
    Since the morphism $\iota\in \CC(B,A)^\BB$, for any $X \in \BB$, we have 
    \begin{align*}
        d_A(X) \id_B = \pi (d_A(X) \iota) = \pi\circ X\iota =
 \frac{1}{d(A)}
\vcenter{\xy 0;/r.9pc/: (0,0)="o";(1,1.4)**\dir{}, "o",
{\ellipse(1.5,2.2)^,=:a(345){-}},
{\ar@{.}@`{(2.5, 1), (2.3, -1.7)} (1.2,-1.3);(2.4, -2)},
(2.4,-2.4)*+=<9pt,9pt>[F]{\scriptstyle\pi},
{\ar@{.}@`{(.9,1.5)} (1.2, 2.5);(-1.3, -1.)},
 (1.2,2.9)*+=<9pt,9pt>[F]{\scriptstyle\iota},
 {\ar@{-}(1.2, 3.3);(1.2, 4.5)*+{\scriptstyle B}},
 {\ar@{-} (2.4,-2.8);(2.4, -4)*+{\scriptstyle B}},
(-1.5,1.6)*+{\scriptstyle X},
\endxy} 
 =
 \frac{1}{d(A)}
\vcenter{\xy 0;/r.9pc/: (0,0)="o";(1,1.4)**\dir{}, "o",
{\ellipse(1.5,2.2)^,=:a(345){-}},
{\ar@{.}@`{(.9,1.5)} (1.2, 2.5);(-1.3, -1.)},
 (1.2,2.9)*+=<9pt,9pt>[F]{\scriptstyle\iota},
 {\ar@{-}(1.2, 3.3);(1.2, 4.5)*+{\scriptstyle B}},
(-1.5,1.6)*+{\scriptstyle X},
{\ar@{.}@`{(3.8, 3.8), (2.3, -1.7)}|(.25){\hole} (-1,-1.7);(2.6, -2)},
(2.6,-2.4)*+=<9pt,9pt>[F]{\scriptstyle\pi},
{\ar@{-} (2.6,-2.8);(2.6, -4)*+{\scriptstyle B}},
\endxy}\,.
\end{align*}
Now we take trace of both sides in $\CC$, and find
 \begin{align*}
        d_A(X) d(B)  
 =
 \frac{1}{d(A)}
\vcenter{\xy 0;/r.9pc/: (0,0)="o";(1,1.4)**\dir{}, "o",
{\ellipse(1.5,2.2)^,=:a(345){-}},
{\ar@{.}@`{(.9,1.5)} (1.2, 2.5);(-1.3, -1.)},
 (1.2,2.9)*+=<9pt,9pt>[F]{\scriptstyle\iota},
 (1.2, 4.5)*+{\scriptstyle B},
 (1.2,3.3)\vcap[2.5],
 {\ar@{-} (3.7,3.3);(3.7,-2.8)},
 (2.6,-2.8)\vcap[-1.1],
(-1.5,1.6)*+{\scriptstyle X},
{\ar@{.}@`{(3.8, 3.8), (2.3, -1.7)}|(.25){\hole} (-1,-1.7);(2.6, -2)},
(2.6,-2.4)*+=<9pt,9pt>[F]{\scriptstyle\pi},
\endxy}
 =
 \frac{1}{d(A)}
\vcenter{\xy 0;/r.9pc/: (0,0)="o";(1,1.4)**\dir{}, "o",
{\ellipse(1.5,2.2)^,=:a(345){-}},
{\ar@{.}@`{(.9,1.5)} (1.2, 2.5);(-1.3, -1.)},
 (1.2,2.9)*+=<9pt,9pt>[F]{\scriptstyle\iota},
 (1.2, 4.5)*+{\scriptstyle B},
 (1.2,3.3)\vcap[2],
(-1.5,1.6)*+{\scriptstyle X},
{\ar@{.}@`{(0, -1), (3, 1)}|(.5){\hole} (-1,-1.7);(3.2, 2.5)},
(3.2,2.9)*+=<10pt,10pt>[F]{\scriptstyle\pi^*},
\endxy}
=\frac{1}{d(A)}\tr\left(\hspace{-7pt}\vcenter{\xy 0;/r1pc/: 
{\ar@{-}(1.3,3.8)*+{\scriptstyle X}; (1.3,-2)|(0.5)\hole},
(1.1, 2.2)\hcap[-1.3],
(1.5,2.2)\hcap[3],
(1.5, 0.9)\hcap[1]
{\ar@{-}(1.1, 0.9); (1.5,.9)},
{\ar@{-}(1.3, -.1); (1.5,-.1)},
{\ar@{-}(1.3, -.8); (1.5,-.8)},
(2,2.5)*{\scriptstyle B}
\endxy} \hspace{-16pt}\right)\,.
\end{align*}
Let $X'$ be the sum of local $B$-submodules of $X$. According to \cite[Lem. 4.3]{KO02},
$$
P_{X'}:= \frac{1}{d(B)}\vcenter{\xy 0;/r1pc/: 
{\ar@{-}(1.3,3.5)*+{\scriptstyle X}; (1.3,-1.5)|(0.5)\hole},
(1.1, 2.2)\hcap[-1.3],
(1.5,2.2)\hcap[3],
(1.5, 0.9)\hcap[1]
{\ar@{-}(1.1, 0.9); (1.5,.9)},
{\ar@{-}(1.3, -.1); (1.5,-.1)},
{\ar@{-}(1.3, -.8); (1.5,-.8)},
(2,2.5)*{\scriptstyle B}
\endxy} 
$$
is the projection onto $X'$. Thus $\tr(P_{X'}) = d(X')$ and $d(X) = d(X')$. Since $X'$ is a subobject of $X$ and the dimension of any simple object of $\CC$ is positive, $X=X'$. Therefore, $X$ is a local $B$-module or $\BB \subseteq (\CBo)_A$.

If $\BB < (\CBo)_A$, by (i) and (iii), we have
$$
(B, \iota) = A^\BB >A^{(\CBo)_A} \equiv (B, \iota)\,,
$$
which is a contradiction. Therefore, $\BB = (\CBo)_A$.
\end{proof}
As a consequence of the preceding theorem of fusion action Galois correspondence, we establish a more general version of Schur-Weyl duality. 
\begin{thm}
     Let $\CC$ be a pseudounitary MTC, $A$ a condensable algebra in $\CC$,  $\BB$ a fusion subcategory of $\CA$ containing $\AA=\CAo$, and $A^\BB=(B, \iota)$. Then
     \begin{enumeri}
     \item For any $X \in \irr(\CBo)$, if $\CBo(X,A) \ne 0$, then $\CBo(X,A)$ is an irreducible $K(\BB)$-module. 
          \item For any simple objects $X,Y \in \CBo$ such that $\CBo(X,A) \ne 0$ and $\CBo(Y,A) \ne 0$, the $K(\BB)$-modules $\CBo(X, A)$ and $\CBo(Y, A)$ are isomorphic if and only if $X \cong Y$ as $B$-modules.
          \item $e_\1^\AA K(\BB) \cong \CBo(A, A)$ as $\BC$-algebras.
          \item For any simple object $X \in \CBo$, $\rho(e^\BB_{X^*})\CBo(A,A)  \cong \CBo(A_X, A_X)$ as $\BC$-algebra, where $A_X$ is the isotypic component of $X$ in $A$, and $e_{X^*}^\BB$ is the central idempotent of $e_\1^\AA K(\BB)$ corresponding to the simple $K(\BB)$-module $\CBo(X,A)$. 
     \end{enumeri}   
 \end{thm}
 \begin{proof} 
     Since $(\CBo)_A= \BB$ by Theorem \ref{t:Gal_Cor} and $(\CBo)_A^0 = \AA$ as spherical fusions categories, one can apply Theorem \ref{t1} to  the condensable algebra $A$ in $\CBo$. Now the statements (i), (ii) and (iii) follow directly. For any $X \in \CBo$, $\ta(X) = X \o_B A$ and so 
    $$
    [\ta(X),A]_{\BB} = [X, A]_{\CBo}\,.
    $$ 
     It follows from the proof of Theorem \ref{t1} that the set indecomposable central idempotents of $e_\1^\AA K(\BB)$ is
     $$
     \{e_{X}^\BB \mid X\in \irr(\CBo),\  [X,A]_{\CBo} \ne 0\},
     $$   
     where $e_{X}^\BB := e_\1^{\AA} e_X^{\CBo}$. The kernel of ring epimorphism $\rho: K(\BB) \to \CBo(A,A)$ is $(1-e_\1^\AA)K(\BB)$ and so $\rho(e_{X}^\BB)$ is an indecomposable central idempotent of $\CBo(A,A)$ if $[X, A]_{\CBo} \ne 0$. Moreover, $\rho(e_{X}^\BB)$ is a projection of $A$ onto $A_{X^*}$ by the proof of Theorem \ref{t1}. Thus, $\rho(e_{X^*}^\BB) \CBo(A,A) \cong \CBo(A_X,A_X)$ as $\BC$-algebras. 
 \end{proof}
Since the unit object in $\CBo$ is $B$, the multiplicity space $\CBo(X,A)$ can be identified with $\CBo(X,A)B$, and we have
$$
A = \sum_{X \in \iCBo} \CBo(X, A)\o_B X
$$
which is a categorical generalization of Schur-Weyl duality of the algebra $A$ in $\CBo$.

\begin{proof}[Proof of Theorem {\rm III}] Now the theorem  follows from Theorems \ref{t1} and \ref{t:Gal_Cor} immediately 
with $\CC=\Mod{U}$.
\end{proof}

Next we show that if $\CA$ is a $G$-extension of $\CAo$ for any finite group $G$,  then the fusion action of 
$\CA$ induces a group $G$ action on $A.$ Recall from \cite{EGNO} that 
$\CA$ is a $G$-extension of $\CAo$ means $\CA$ has a faithful $G$-gradation
$$\CA=\bigoplus_{g\in G}\CA(g)$$
such that  $\CA(1)=\CAo$ and $\CA(g)\otimes \CA (h)$ is contained in $\CA(gh).$ In this case,
$$\iCA=\bigcup_{g\in G}\irr(\CC_A(g))$$.

 \begin{prop}\label{p5.8}
Let $A$ be a condensable algebra in a pseudounitary modular tensor category $\CC$. Suppose that the $A$-module category $\CA$ in $\CC$ is a $G$-extension of $\CAo$ for some finite group $G$. Then the assignment $$g \mapsto \ol g := \frac{1}{\dim(\CAo)} \sum_{X \in \irr(\CA(g))}d_A(X) X$$ defines a group monomorphism from $G$ to $\KA$. Moreover, the action of $\CA$ reduces an action of $G$ on $A,$ each $W_x$ with $[x,A]_\CC\ne 0$ is a simple 
$G$-module, and there is one to one correspondence between the subgroups of $G$ and the condensable subalgebras of $A.$
\end{prop}

\begin{proof}
It follows from the proof of \cite[Thm. 3.5.2]{EGNO} that $\ol{g}\ol{h} = \ol{gh}$ for any $g, h \in G$. Obviously, $\ol g = \ol 1$ if and only if $g=1$. Thus, $\ol G$ forms a group under the multiplication of $\KA$ is isomorphic to $G$.  

Next we show that $\ol{K}=e_\1^\AA K(\CA)$ is isomorphic to the group algebra $\BC[G]$ where $\AA=\CAo.$  
It follows from \cite[Prop. 3.3.6]{EGNO} that $X \sum_{g \in G} \ol g = d_A(X) \sum_{g \in G} \ol g$ for $X\in \iCA$. Thus, $X \ol g = d_A(X) \ol{hg}$ for $X \in \CA(h)$. In particular, $X \ol 1 = d_A(X) \ol h$, and this implies $\frac{1}{d_A(X)} X = \ol h$ in $\ol K$ for all $X \in \CA(h)$.  So $\CA(h)$ is an indecomposable $\CAo$-module and $\ol K=\bigoplus_{g\in G}\BC \ol g$ is isomorphic to the group algebra $\BC[G].$

Thanks to Theorem \ref{t:Gal_Cor}, the last assertion is equivalent to that 
there is one to one correspondence between the subgroups of $G$ and the fusion subcategories of 
$\CA$ containing $\CAo.$ It is clear that if $H$ is a subgroup of $G$ then $\oplus_{h\in H}\CA(h)$ 
is a fusion subcategory of $\CA.$ Conversely, let $\BB$ be a fusion subcategory of $\CA$ containing $\CAo.$
Then for any $X\in\irr(\BB)\cap \CA(g),$   $\CA(g)$ is a subcategory of $\BB$ as $\CA(g)$ is an indecomposable $\CAo$-module. Thus there exists $H<G$ such that $\BB=\oplus_{h\in H}\CA(h).$ 
\end{proof}

\section{Galois Correspondence via relative centers }\label{s:relative}
Let $\CC$ be a modular tensor category with a Lagrangian algebra $A$. Then $\CC = Z(\CA)$. By \cite[Thm. 4.10]{DMNO},  there exists an order-reversing bijective correspondence between $\LCA$ and $\LaA$. In this section, we recall the result of \cite[Thm. 4.10]{DMNO} with additional details which is relevant to our discussion, and apply to the case that $A$ is not a Lagrangian algebra of $\CC$. We further show that the condensable subalgebra $(B, \iota)$ of $A$ corresponding to the fusion subcategory $\BB \in \LCA$ is indeed isomorphic to $A^\BB$ as an object in $\CC$. In this regard, Galois correspondence between $\LCA$ and $\LaA$ can also be established from  \cite[Thm. 4.10, Rmk. 4.12]{DMNO}. Following the approach of \cite[Thm. 4.10]{DMNO}, the correspondence for any condensable algebra $A$ in a modular tensor category $\CC$ was proved in \cite[Thm. 3.5]{LKW}.

Let $\CC$ be a pseudounitary MTC. Suppose $A \in \CC$ is a Lagrangian algebra. Then $\CA$ is a spherical fusion category and $Z(\CA)=\CC$. Now, we consider any fusion (full) subcategory $\BB$ of $\CA$, and let $Z_\BB(\CA)$ be the relative center of $\BB$ in $\CA$.  Recall that the objects of $Z_\BB(\CA)$ are pairs $(X, \s_{X, -})$ where $X \in \CA$ and $\s_{X, V}: X \o V \to V\o X$ for $V \in \BB$ is an isomorphism  natural at $V$ and it satisfies the half-braiding conditions. In particular, $Z(\BB)^{\rev}$ is a tensor  subcategory of $Z_\BB(\CA)$. 

The forgetful functor $F: Z(\CA) \to \CA$ can be factorized as
$$
Z(\CA) \xrightarrow{F_\BB} Z_\BB(\CA) \xrightarrow{\tF_\BB} \CA\,,
$$
which is a composition of forgetful functors. Let $I_\BB: Z_\BB(\CA) \to Z(\CA)$ and $\tI_\BB: \CA \to Z_\BB(\CA)$ be the right adjoints of the forgetful functors $F_\BB, \tF_\BB$ respectively. It is clear that $I_\BB \tI_\BB = I$, the right adjoint of $F$.

Since $\1$ is an algebra in $\CA$,  $\tI_\BB(\1)$ is an algebra in $Z_\BB(\CA)$. It was shown in \cite[Thm. 4.10]{DMNO} that $I_\BB(\1) = B$ is a condensable algebra of $\CC$ and is considered as a subalgebra of $A=I(\1)$ via the embedding $I_\BB(u'): B \to A$, where $u' : \1 \to \tI_\BB(\1)$ is unit map of the algebra $\tI_\BB(\1)$. The following theorem has been proved in \cite[Thm. 4.10]{DMNO}. 

\begin{thm}\label{t:DMNO} Let $\CC$ be a pseudounitary MTC and $A$ a Lagrangian algebra in $\CC$.
    The assignment for any fusion subcategories $\BB \subseteq \CA$ to the condensable subalgebra $B:=I_\BB(\1) \subseteq I(\1)$ defines an order reversing one-to-one correspondence between $\LCA$ and $\LaA$, and $\FPdim(B)\FPdim(\BB) = \FPdim(\CA)$. Moreover, $Z_\BB(\CA)$ is tensor equivalent to $Z(\CA)_B$, and the subcategory  $Z(\BB)^{\rev}$ of $Z_\BB(\CA)$ is mapped to $Z(\CA)_B^0$ under this equivalence.
\end{thm}

To archive our goal, we follow a similar approach as the proof of  \cite[Thm. 4.10, Rem. 4.12]{DMNO}  but  $A$ is an arbitrary condensable algebra in $\CC$. Since  $Z(\CA) = (\CAo)^{\rev} \boxtimes \a^+(\CC)$ as modular tensor categories, one can consider the composition
$$
F^+_\BB:=\CC \xrightarrow{\a_A^+} Z(\CA) \xrightarrow{F_\BB} Z_\BB(\CA)  
$$
which is obviously a central functor. Let $I^+_\BB: Z_\BB(\CA) \to \CC$ be the right adjoint of $F^+_\BB$. By \cite[Lem. 3.5]{DMNO}, $B=I^+_\BB(\1)$ is a condensable algebra in $\CC$ and $\CC_B$ is equivalent  as tensor categories to  $\ol{F^+_\BB(\CC)}$, the smallest tensor subcategory of $Z_\BB(\CA)$ containing the image of $F^+_\BB$. One can verify directly that the composition $\tF_\BB F^+_\BB$ is the induction functor $\a_A = (-)\o A$, where $\tF_\BB: Z_\BB(\CA) \to \CA$ denotes the forgetful functor. Since $\a_A: \CC \to \CA$ is surjective, so is $\tF_\BB : \ol{F^+_\BB(\CC)} \to \CA$.  

Let $\tI_\BB: \CA \to Z_\BB(\CA)$ be the right adjoint of $\tF_\BB$. Then the composition $I^+_\BB \tI_\BB$ is the right adjoint of $\a_A: \CC \to \CC_A$, and so $I^+_\BB \tI_\BB$ is isomorphic to the forgetful functor $\CC_A \to \CC$.

If $\BB = \CAo$, $\tF_\BB: \ol{F^+_\BB(\CC)} \to \CA$ is an equivalence. To see this, one should first notice that the braiding $R_{X,Y}$ in $\CC$ for $X \in \CA$ and $Y \in \CAo$ can be descended to an isomorphism $\tR_{X,Y}: X \oa Y \to Y \oa X$ of $A$-modules, and $(X, \tR_{X,-}) \in Z_\BB(\CA)$. These pairs $(X, \tR_{X,-})$ form a tensor subcategory  $\widetilde{\CA}$ of $Z_\BB(\CA)$, and  $\widetilde{\CA} = \ol{F^+_\BB(\CC)}$ since both have the same FP-dimension $\FPdim(\CA)$. Therefore,   $\tF_\BB : \ol{F^+_\BB(\CC)} \to \CA$ is an equivalence of tensor categories and $I^+_{\BB}(\1) = A.$

The preceding paragraph and the following lemma are elaborations of \cite[Rem. 4.12]{DMNO}.

\begin{lem}\label{l:AandB}
 For any fusion subcategory $\BB$ of $\CA$ containing $\CAo$,  $B= I^+_\BB(\1)$ is a condensable subalgebra of $A$, and  $A$ is a condensable algebra in $\CBo$. Moreover, $(\CBo)_A = \BB$ and $Z(\BB)$ is equivalent to $\CBo \boxtimes (\CAo)^{\rev}$ as braided tensor categories.
\end{lem}
\begin{proof}
Let $\BB \in \LCA$. By \cite[Thm. 4.10]{DMNO}, $B:=I^+_\BB(\1)$ is considered as a condensable subalgebra of $A$ via the algebra embedding $\iota: B \to A$ given by $\iota:=I^+_\BB(u)$, where $u: \1 \to \tI_\BB(\1)$ is the unit map of the algebra $\tI_\BB(\1)$ in $Z_\BB(\CA)$. In particular, $A$ is a condensable algebra in $\CBo$.

By \cite[Lem. 3.5]{DMNO}, $\ol{F^+_\BB(\CC)}$  is equivalent to $\CB$ as fusion categories via the internal hom-functor $\uH(\1, Y)$ for $Y \in \ol{F^+_\BB(\CC)}$ (cf.  \cite[Thm. 3.1]{Ost03Mod}). By \cite[Lem 3.3]{Ost03Mod}, for $x \in \CC$, $\uH(\1, F^+_\BB(x)) \cong x \o \uH(\1, \1) \cong x \o B$ as $B$-modules. So, the functor $F^+_\BB : \CC \to \ol{F^+_\BB(\CC)}$ is identified with the induction functor $- \o B : \CC \to \CC_B$, and $\tF_B: \ol{F^+_\BB(\CC)} \to \CC_A$ is identified with $-\o_B A: \CC_B \to \CC_A$. Moreover, $\CBo$ is identified with the subcategory $Z(\BB)^{\rev} \cap \ol{F^+_\BB(\CC)}$ (cf. Theorem \ref{t:DMNO}). Therefore, $(\CBo)_A=\ol{\CBo \o_B A}$ is a subcategory of $\BB$. Since 
$$
\FPdim((\CBo)_A) =\frac{\FPdim(\CC)/\FPdim(B)^2}{\FPdim(A)/\FPdim(B)} =\frac{\FPdim(\CA)}{\FPdim(B)} = \FPdim(\BB), 
$$
we find $(\CBo)_A = \BB$. 

By \cite[Cor. 3.30]{DMNO}, $Z(\BB) = Z((\CBo)_A)$ is equivalent to $\CBo \boxtimes ((\CBo)_A^0)^{\rev}$. It is clear that $\CAo$ is a tensor subcategory of $(\CBo)_A^0$. Since 
$$
\FPdim((\CBo)_A^0) = \frac{\dim(\CBo)}{d_B(A)^2} = \frac{\FPdim(\CC)}{\FPdim(A)^2} = \FPdim(\CAo),
$$
we have $((\CBo)_A^0 = \CAo$, and the last assertion follows.
\end{proof}

Since $A \in \CBo$, one can consider its irreducible decomposition 
$$
A \cong \sum_{X \in \irr(\CBo)} [A, X]_{\CBo}\,X \cong \sum_{X \in \irr(\CBo)} \CBo(X, A)\o_B X
$$
in $\CBo$, where the $\BC$-linear space $\CBo(X, A)$ is identified with $\CBo(X, A) B$ in the second isomorphism. Since $(\CBo)_A=\BB$ by Lemma \ref{l:AandB}, $K(\BB)$ acts on $\CBo(X,A)$ for any $X \in \CBo$ by Definition \ref{def:action2} (with $\CC$ being replaced by $\CBo$), which is irreducible if $X \in  \iCBo$ with $\CBo(X,A) \ne 0$ by Theorem \ref{t1}.


Now, we can obtain the following ``Galois correspondence" between fusion subcategories of $\CA$ and condensable subalgebras of $A$. 
 
 \begin{cor}
     The condensable subalgebra $B=I_\BB^+(\1)$ of $A$ determined by $\BB \in \LCA$ is isomorphic to  $A^\BB$ as objects in $\CC$.
 \end{cor}
 \begin{proof}
     Let $x \in \irr(\CC)$ such that $\CC(x,A) \ne 0$. Then, we have
\begin{align*}
    \CC(x,A) & = \CA(x\o A,A) = \CA(x \o B \o_B A,A) = \CB(x\o B, A) \\
    &=\, \sum_{Y \in \irr(\CB)} \CB(x \o B, Y)\o \CB(Y,A) \\
    &=\, \sum_{Y \in \irr(\CBo)} \CB(x \o B, Y)\o \CBo(Y,A)\\
    &=\, \sum_{Y \in \irr(\CBo)} \CC(x, Y)\o \CBo(Y,A),
\end{align*}
which is the decomposition of $\CC(x,A)$ into a direct some of irreducible $K(\BB)$-modules. Since $\CC(x,A)^\BB$ is the isotypic component of $\CC(x,A)$ corresponding to the dimension character of $\BB$, we find
$$
\CC(x,A)^\BB =  \CC(x, B)\o \CBo(B,A) \cong \CC(x, B)
$$
as $\BC$-spaces since $[B,A]_{\CBo} = 1$. Now, we have
$$
B = \sum_{x \in \irr(\CC)} \CC(x,B) x = \CC(x,A)^\BB x  = A^\BB \,.  \qedhere
$$
 \end{proof}

\section{Examples of fusion actions on VOAs}

We discuss several examples of fusion actions in terms of orbifolds. In particular, we show in these examples that iterated orbifolds of group actions on VOAs are invariants of actions of some weakly integral fusion categories.

\subsection{Orbifold theory}

Let $A$ be a simple VOA. Let $G$ be a finite automorphism group of $A$.  It follows from \cite{DLMcomp1996} that $A^G$ is a simple vertex operator subalgebra of $A$ and $A$ has a Schur-Weyl type duality decomposition 
 $$A=\oplus_{\lambda\in \irr(G)}W_{\lambda}\otimes A_{\lambda}$$  where $W_{\lambda}$ is the irreducible $G$-module with character $\lambda$ and the multiplicity spaces $A_{\lambda}$ are inequivalent irreducible $A^G$-modules. In addition, there is a one-to-one correspondence between the subgroups of $G$ and the sub VOAs of $A$ containing $A^G$ by sending subgroup $H$ to $A^H$ \cite{DM1997, HMT1999}. The tensor product of two modules $V_1$ and $V_2$ of a VOA $A$ will be denoted by $V_1 \o V_2$ instead of the traditional notation $V_1 \boxtimes V_2$. If $A_1, A_2$ are VOAs, then so is their tensor product $A_1 \o A_2$ over $\BC$, and $\Mod{A_1 \o A_2}$ is equivalent to the Deligne product $\Mod{A_1}\!\boxtimes \Mod{A_2}$.  
 
 We now explain that the group action of $G$ on $A$ is the fusion action in this paper. For this purpose, we need to assume that $A^G$ is rational and $C_2$-cofinite. Then the $A^G$-module category $\CC=\Mod{A^G}$ is a modular tensor category, $A$ is a condensable algebra in $\CC$
 and $\CA$ is a $G$-extension of $\CAo$ \cite[Lem. 3.3]{DLXY2024}, \cite[Cor. 2.5]{Mc2021}. 
 
\begin{thm} \label{t2}
   The action of $\ol{K}=e_\1^\AA K(\CA)$ on $A$ is equivalent to the action of $G$ on $A$, where $\AA = \CAo$.
   \end{thm}
\begin{proof} By Proposition \ref{p5.8}, the action of $\ol{K}$ on $A$ is equivalent to the action of $\ol G$ on $A.$ So it suffices to show that the actions of $\ol G$ and $G$ are equivalent. That is,  we need to prove that the actions of the groups $\overline{G}$ and $G$ on $W_\lambda = \CC(A_\lambda,A)$ are equivalent for all $\lambda\in \irr(G)$. Recall from Section 4 that 
$$\phi_{W_\lambda} = \sum_{Y \in  \iCA} \Tr(Y, W_\lambda)Y^*=\sum_{g\in G}\sum_{ Y\in \irr(\CA(g))} \Tr(Y, W_\lambda)Y^*$$
for any $\lambda\in \irr(G).$ Since $Y=d_A(Y) \bar g$ for any  $Y\in \irr(\CA(g))$,  we see that 
$$\sum_{ Y\in \irr(\CA(g))} \Tr(Y, W_\lambda)Y^*=\dim(\CAo)\Tr(\bar g, W_\lambda)\bar g^{-1}.$$
Thus 
$$
\frac{\dim({W_\lambda})}{\dim(\CC_A)}\phi_{W_\lambda}=\frac{\dim(W_\lambda)}{o(G)}\sum_{g\in G} \Tr(\bar g, W_\lambda)\bar g^{-1}.
$$
Noting from  \cite{DJX2013} that $d(A_{\lambda}) =\dim(W_{\lambda}).$
It follows from (\ref{eq:trace1}) that  
$$\frac{\dim(W_{\lambda})}{\dim(\CC_A)}\Tr(\phi_{W_\lambda}, W_\lambda)=\dim(W_{\lambda}).$$
Since $\phi_{W_\lambda}=0$ on $W_{\mu}$ if $\mu\ne \lambda$, we see  immediately that the indecomposable central elements 
$$\frac{\dim({W_\lambda})}{\dim(\CC_A)}\phi_{W_\lambda}=\frac{\dim(W_\lambda)}{o(G)}\sum_{g\in G}\lambda(g)g^{-1}$$
in $\overline{K}=\BC[\overline{G}]=\BC[G]$ and the proof is complete.
\end{proof}
\begin{lem} \label{l:pointed}
    Let $G$ be a finite automorphism group of a simple VOA $A$ such that $A^G$ contains a rational, $C_2$-cofinite subVOA $B$ and $\Mod{B}$ is a pseudounitary MTC. Then $\CA$ contains a $G$-graded fusion subcategory $\BB$ with $\BB(1) = \Mod{A}$, where $\CC= \Mod{B}$. In particular, if $A$ is holomorphic, then $\CA$ contains  $\Vec_G^\w$ as a fusion subcategory  for some $\w \in Z^3(G, \BC^\times)$. 
\end{lem}
\begin{proof}
    By the  Galois correspondence (Theorem \ref{t:Gal_Cor}) that $\BB = (\Mod{A^G})_A = (\CC_{A^G}^0)_A$ is a fusion subcategory of $\CA$. It follows from the remark before Theorem \ref{t2} that $\BB$ is a $G$-extension of $\Mod{A}$. The last assertion is an immediate consequence of the fact  that $\Vec = \Mod{A}$ when $A$ is holomorphic.
\end{proof}

\begin{prop}
    Let $A$ be a holomorphic $VOA$ and  $A=A_0 \supset A_1 \supset \cdots \supset A_n$  be simple subVOAs of $A$ such that $A_n$ is  rational, $C_2$-cofinite of CFT type and
    the weight of any irreducible $A_n$-module is positive except $A_n$ itself. Suppose there exist finite subgroups $G_i$ of $\Aut(A_i)$ such that $A_{i+1} = A_i^{G_i}$ for $i=0, \dots, n-1$.  Then $\BB_{i+1}=(\Mod{A_{i+1}})_{A}$ is a $G_i$-extension of $\BB_i$ and
    $$
    \BB_0=\Vec \subset \BB_1 \subset \BB_2 \subset \cdots \BB_n = \CA
    $$
    where $\CC=\Mod{A_n}.$ In particular, $\CA$ is a nilpotent category, and $\FPdim(\CA) = |G_0|\cdots |G_{n-1}|$.
\end{prop}
\begin{proof}
    From the assumption we know that $\Mod{A_n}$ is  pseudounitary. By the preceding lemma, $\BB_1$ is a $G_0$-extension of $\BB_0$. Assume that $\BB_i$ is a $G_{i-1}$-extension of $\BB_{i-1}$ for some integer $i \ge 1$. By the preceding lemma, $(\Mod{A_{i+1}})_{A_i} = (\CC_{A_{i+1}}^0)_{A_i} = \bigoplus_{g \in G_i} (\Mod{A_{i+1}})_{A_i}(g)$ is a $G_i$-extension of $\Mod{A_i} = \CC_{A_i}^0$.
    Let $\BB_{i+1}(g)$ denote the full subcategory of $\CA$ in which each simple object is a summand of $X \o_{A_i} A$ for some $X \in (\Mod{A_{i+1}})_{A_i}(g)$.   In particular, $X \o_{A_i} A \in \BB_{i+1}(g)$ for any $X \in (\Mod{A_{i+1}})_{A_i}(g).$ Obviously, $\BB_{i+1}(1)= (\Mod{A_i})_A =\BB_i$, and $A$ is a condensable algebra in $\Mod{A_i}$. If $Y \in (\Mod{A_{i+1}})_{A_i}(h)$,  then
    $$
    (X \o_{A_i} A) \oa (Y \o_{A_i} A) \cong (X \o_{A_i} Y) \o_{A_i} A \in \BB_{i+1}(gh)
    $$
    and 
    $$
    (Y\o_{A_i} A)^* \cong  Y^* \o_{A_i} A \in \BB_{i+1}(h^{-1})\,.
    $$
    Therefore, $\BB_{i+1}$ is a $G_i$-extension of $\BB_i$ and $\FPdim(\BB_{i+1}) = |G_i|\FPdim(\BB_i)$. Now the proof  is completed by induction,  and $\FPdim(\BB_{i+1})=|G_0|\cdots |G_i|$. By  definition  (cf. \cite[Sec.1] {ENO2}), $\CA$ is a nilpotent category, and $\FPdim(\CA) = |G_0|\cdots |G_{n-1}|$.
\end{proof}

\subsection{Fusion actions associated to \texorpdfstring{$A_l$}{A}} \label{ss:Al} The simple roots for $A_{l}$ can be expressed in terms of an orthonormal basis $e_1, \ldots, e_{l+1}$ as follows:
$$
\alpha_k=e_k-e_{k+1} \text { for } 1 \leq k \leq l .
$$ Let $L$ be the root lattice of type $A_{l}$ and 
$$
\lambda_k=\sum_{j=1}^k e_j-\frac{k}{l+1} \sum_{j=1}^{l+1} e_j  
$$
 be the fundamental weights for $k=1,....,l$. Also set $\lambda_0=0.$
Then $V_L$ is a rational, $C_2$-cofinite  VOA whose irreducible modules are $V_{L+\lambda_r}$ for $0\leq r\leq l$ \cite{D1993}.  

By \cite[Prop. 4.3]{DNR25} that there exists another positive definite even lattice $K$ such that the orthogonal sum $(L,K)=L+K$ is a sublattice of an even unimodular lattice $E$ and $E=\cup_{r=0}^{l} (L+\lambda_r, K+\mu_r)$
where $K^{\circ}=\cup_{r=0}^{l} (K+\mu_r)$ is the dual lattice of $K$ and $\mu_0=0.$ Moreover,
$\Modrev{V_L}$ is braided equivalent to $\Mod{V_K}.$

The VOA $V_E$ has two automorphisms $\sigma=\sigma_E,\tau$ such that 
$\tau$ acts on $V_{L+\lambda_r}\otimes V_{K+\mu_r}$ as scalar $e^{\frac{2\pi ir}{l+1}}$ and $\sigma$ is the involution
induced from the $-1$-isometry of $E.$ The automorphisms $\tau,\sigma$ generate a dihedral group $G$ of order 
$2(l+1).$ For each subspace $X$ of $V_E$ which is invariant under the action of $\sigma$,  we have $X=X^+\oplus X^-$ where $X^{\pm}$ is the eigenspace of $\sigma$ with eigenvalue $\pm 1.$ 
Then $V_E^G=V_{L+K}^+=(V_L\otimes V_K)^+.$ 

In this example, we consider $\CC=\Mod{V_{L}^+\otimes V_{K}^+ }=\Mod{V_{L}^+}\boxtimes \Mod{V_{K}^+}$ and the holomorphic VOA $A:=V_E$ is a condensable algebra  in $\CC.$ Note that $V_{L+K}^+$ is also a condensable algebra in  $\CC$,
and $V_L^+\otimes V_K^+=(V_{L+K}^+)^{\langle\sigma_L \rangle}.$ Clearly, $V_{L}^+\otimes V_{K}$ is not a fixed point subalgebra of $V_E$ under the action of any finite automorphism group of $V_E$, and so is $V_{L}^+\otimes V_{K}^+\subset V_E$.

\subsubsection{$l=2n$ cases}

Recall that $\sigma_L$ is the order 2 automorphism of $V_L$ induced from the $-1$-isometry of $L.$  We denote the $\pm 1$-eigenspaces of $\sigma_L$ on $V_L$ by
$V_L^{\pm}.$
Then $V_{L+\lambda_r}\circ \sigma_L$ \cite{DLM2000} is isomorphic to $V_{L-\lambda_r}=V_{L+\lambda_{2n+1-r}}$ for $r=1,...,n.$ It follows from \cite{D1994,DLM2000} that $V_L$ has a unique irreducible $\sigma_L$-twisted module $V_L^T$ with the lowest weight $\frac{n}{8}$ \cite{DLrt1996}. Then $\sigma_L$ also acts on $V_L^T$ such that the $\pm1$-eigenspaces $(V_L^T)^{\pm}$ are irreducible $V_L^+$-modules. Furthermore, the fixed points $V_L^+$ is a rational, $C_2$-cofinite VOA whose irreducible modules are 
$$\{V_{L}^{\pm},\, V_{L+{\lambda_r}}\cong V_{L+\lambda_{2n+1-r}},\, (V_L^T)^{\pm}\mid r=1,...,n\}. $$
 The twists and dimensions  of $V_L^+$-modules are as follows:
$$
\renewcommand*{\arraystretch}{1.5}
\begin{array}{c|c|c|c|c|c}
& V_L^+ & V_L^-& V_{L+\lambda_r} & (V_L^T)^+ & (V_L^T)^- \\
\specialrule{1pt}{.2ex}{.2ex}
\theta & 1 & 1 & e^{\frac{r(2 n+1-r)2\pi i}{2(2 n+1)}} & e^{\frac{2\pi i n}{8}} & -e^{\frac{2\pi i n}{8}}\\ 
\hline 
\dim &1 & 1 &2 &\sqrt{2n+1} &\sqrt{2n+1}\\
\end{array}
$$
So $\dim \Mod{V_L^+}=8n+4.$

Here are the fusion products of irreducible modules \cite{ADL2005} for $i,j\ne 0$:
$$ V_L^{-} \o  V_{L+\lambda_i}=V_{L+\lambda_i},\quad V_L^{-}\o  ( V_L^T)^{\pm}= ( V_L^T)^{\mp},$$
$$ V_{L+\lambda_i}\o  V_{L+\lambda_j}=V_{L+\lambda_i+\lambda_j}+V_{L+\lambda_i-\lambda_j}\ \ {\rm if}\  i\ne j,$$ 
$$ V_{L+\lambda_i}\o  V_{L+\lambda_i}=V_L^{+}+V_L^{-}+V_{L+2\lambda_i},$$ 
$$V_{L+\lambda_i}\o  ( V_L^T)^{\pm}=( V_L^T)^{+}+( V_L^T)^{-},$$
$$( V_L^T)^{+}\o  ( V_L^T)^{\pm}=V_L^{\pm}+\sum_{r=1}^nV_{L+\lambda_r},$$
$$( V_L^T)^{-}\o  ( V_L^T)^{\pm}=V_L^{\mp}+\sum_{r=1}^nV_{L+\lambda_r}.$$

Let $d$ be the rank of $K$. Again $V_K$ has a unique irreducible $\sigma_K$-twisted module $V_K^T.$ 
By \cite{AD2004}, the irreducible modules of the rational, $C_2$-cofinite VOA $V_K^+$
are 
$$\{V_{K}^{\pm},\, V_{K+{\mu_r}}\cong V_{K+\mu_{2n+1-r}},\, (V_K^T)^{\pm} \mid r=1,...n\}$$ 
with twists and dimensions:
$$
\renewcommand*{\arraystretch}{1.5}
\begin{array}{c|c|c|c|c|c}
& V_K^+& V_K^-& V_{K+\mu_r}& (V_K^T)^+&(V_K^T)^- \\
\specialrule{1pt}{.2ex}{.2ex}
\theta &1 &1 &e^{\frac{-r(2 n+1-r)2\pi i}{2(2 n+1)}} &e^{\frac{2\pi i d}{16}} &-e^{\frac{2\pi i d}{16}}\\ 
\hline 
\dim &1 &1 &2 &\sqrt{2n+1} &\sqrt{2n+1}\\
\end{array}
$$
So $\dim \Mod{V_K^+}$ is also equal to $8n+4.$ The fusion rules of $V_K^+$-modules are the same as the fusion rules of $V_L^+$-modules with $L$ replaced by $K$ and $\lambda_r$ by $\mu_r.$ 
We believe that $\Modrev{V_{L}^+}$ is braided equivalent to $\Mod{V_{K}^+},$ but we cannot prove it in this paper.

First, note that $\dim\CC=(8n+4)^2$ and $A$ has the decomposition in $\CC$ as follows:
$$A=V_L^+\otimes V_K^+\oplus V_L^+\otimes V_K^-\oplus V_L^-\otimes V_K^+ \oplus V_L^-\otimes V_K^-\oplus 2\bigoplus_{r=1}^n V_{L+\lambda_r}\otimes V_{K+\mu_r}\,.$$
It is now clear that $\dim A=\dim \CC_A=8n+4$ and
$[A,A]_{\CC}=4+4n.$  Since $A$ is holomorphic, we see that $|\iCA|=4+4n$  by Theorem \ref{t1}.

Next we determine the simple objects in $\CC_A.$ For $x\in \CC$ we denote the induced $A$-module by $\alpha(x).$ Then for $r=1,...,n$, 
$$\alpha(V_L^+\otimes V_{K+\mu_r})=2\sum_{s=0}^{2n}V_{L+\lambda_s}\otimes V_{K+\mu_r+\mu_s} \quad\text{in } \CC,$$
and so
$$[\alpha(V_L^+\otimes V_{K+\mu_r}), \alpha(V_L^+\otimes V_{K+\mu_s})]_\CA =2 \delta_{r,s}.$$ 
Together with $A$ itself, we have $2n+1$ inequivalent  simple objects in $\CA.$ We denote this set by $H.$ Clearly, each object in $H$ has dimension $1.$ From the fusion rules of $V_L^+\otimes V_K^+$-modules above we see that $H$ is closed under the fusion product. So $H$ is a group of order $2n+1.$

One can easily see that
$$X=\alpha(V_L^+\otimes (V_K^T)^{+})\cong \alpha(V_L^+\otimes (V_K^T)^{-})=V_L\otimes V_K^T+2\sum_{r=1}^nV_{L+\lambda_r}\otimes V_K^T$$
and $Y=\alpha((V_L^T)^{+}\otimes V_K^+)\cong \alpha((V_L^T)^{-}\otimes V_K^+)$ are two new inequivalent simple objects in $\CA$ of dimensions 
$\sqrt{2n+1}.$ Obviously, both $X$ and $Y$ are central elements in $K(\CA)=\overline{K(\CA)}.$ 

One can also check that for any $x\in \irr(\Mod{V_L^+}),$ $y\in \irr(\Mod{V_K^+})$
such that  either $x\ne (V_L^T)^{\pm}$ or $y\ne (V_K^T)^{\pm}$ then $\alpha(x\otimes y)$ is a sum of simple objects from $H\cup \{X,Y\}.$ By Lemma \ref{l:pointed}, the group $G$ of invertible objects of $\iCA$ contains the dihedral group $D_{2(2n+1)}$ as $A^G= (V_L\otimes V_K)^+.$ 
Since $|\irr(\CA)|=4n+4$ and $|H\cup \{X,Y\}|=2n+3$,  we see that the complement of 
$H\cup \{X,Y\}$ in $\irr(\CA)$ consists of $2n+1$ invertible objects 
of order $2$ and $G = D_{4n+2}$.

Since every simple object in $\CA$ is a subobject of a induced module,
the remaining $2n+1$ invertible objects have to be the subobjects of
$$\alpha((V_L^T)^{\pm}\otimes (V_{K}^T)^{\pm})=(2n+1)(V_L^T\otimes V_K^T).$$
So $\alpha((V_L^T)^{\pm}\otimes (V_{K}^T)^{\pm})$ is a direct sum of $2n+1$ inequivalent invertible simple objects in $\CA$ of order $2.$ 

Let $g\in G\setminus H.$ Then
$$g\otimes_A X\cong Y, \ g\otimes_A Y\cong X$$
by using the relation
$$[g\otimes_A\alpha((V_L^T)^{+}\otimes V_K^+), \alpha(V_L^+\otimes (V_K^T)^{+})]_{\CC_A}
=[g, \alpha((V_L^T)^{+}\otimes (V_{K}^T)^{+})]_{\CC_A}=1.$$
Similarly, if $g\in H$ then 
$$g\otimes_A X\cong X,\  g\otimes_A Y\cong Y.$$ 
So $H$ is the normal subgroup of $G$ isomorphic to cyclic group $\langle \tau \rangle$ generated by $\tau.$

Note that $\iCA=G\cup \{X,Y\}.$  Using the fusion rules for $V_L^+$ and $V_K^+$ we see fusion ring structure of $\CC_A$ is  as follows:
$$hX=X,\, hY=Y \ {\rm for}\ h\in H, \quad gX=Y,\, gY=X \ {\rm for} \ g\in G\setminus H,$$
$$X^2=Y^2=\sum_{h\in H}h,\  XY=\sum_{g\in G\setminus H}g.$$
There are only two nontrivial fusion subcategories $\FF_X=H\cup \{X\}$ and $\FF_Y=H\cup\{Y\}$ which are not pointed. One can verify that 
$V_E^{\FF_X}=V_L\otimes V_K^+$ and $V_E^{\FF_Y}=V_L^+\otimes V_K.$ 

Since $D_{2(2n+1)}$ has many subgroups in general, we only give the Galois correspondence for $n=1.$ In this case, $D_{6}=S_3$ which has 4 proper subgroups $H, H_1, H_2, H_3$ of orders $3,2,2,2$ where $H=\langle \tau\rangle, H_1=\langle\sigma \rangle, H_2= \langle \tau\sigma \rangle$ and $H_3=\langle\tau^2\sigma \rangle.$ One can easily find out that $V_E^H=V_{L+K},$
$$V_E^{H_1}=V_E^{+},\, V_E^{H_r}=V_{L+K}^++(V_{L+K+\lambda_1+\mu_1}+V_{L+K+\lambda_2+\mu_2})^{H_r}$$
for $r=2,3$ where 
$$(V_{L+K+\lambda_1+\mu_1}+V_{L+K+\lambda_2+\mu_2})^{H_r}=\{u+e^{\frac{-2\pi i (r-1)}{3}}\sigma(u)\mid u\in V_{L+K+\lambda_1+\mu_1}\}.$$

\subsubsection{$l=2n+1$ cases} Now only  $V_L$ and $V_{L+\lambda_{n+1}}$ are $\sigma_L$-stable. 
From \cite{D1994,DLM2000}, $V_L$ has two inequivalent irreducible $\sigma_L$-twisted modules $V_L^{T_i}$ for $i=1,2$ with the lowest weight $\frac{2n+1}{16}$ \cite{DLrt1996}. The $\sigma_L$ also acts on $V_L^{T_i}$ such that the $\pm1$-eigenspaces $(V_L^{T_i})^{\pm}$ are inequivalent irreducible $V_L^+$-modules. The irreducible $V_L^+$-modules are 
$$\{V_{L}^{\pm},\, V_{L+{\lambda_{n+1}}}^{\pm},\, V_{L+{\lambda_r}}\cong V_{L+\lambda_{2n+2-r}}, \,(V_L^{T_i})^{\pm}\mid r=1,...,n, i=1,2\}.$$
 The twists and dimensions  of $V_L^+$-modules are as follows:
$$
\renewcommand*{\arraystretch}{1.6}
\begin{array}{c|c|c|c|c|c}
&V_L^{\pm}&V_{L+\lambda_{n+1}}^{\pm}&V_{L+\lambda_r}& (V_L^{T_i})^+&(V_L^{T_i})^- \\
\specialrule{1pt}{.2ex}{.2ex}
\theta&1&e^{\frac{(n+1)\pi i}{2}}&e^{\frac{r(2n+2-r)2\pi i}{2(2 n+2)}}&e^{\frac{2\pi i (2n+1)}{16}}&-e^{\frac{2\pi i (2n+1)}{16}}\\ 
\hline 
\dim&1&1&2&\sqrt{n+1}&\sqrt{n+1}\\
\end{array}
$$
So $\dim \Mod{V_L^+}=8n+8.$

For further discussion, we need to determine the dual of $V_L^{T_i}.$ It follows from \cite{FLM1988, ADL2005} that $V_L^{T_i}$ is self-dual when $n$ is odd, whereas
$V_L^{T_1}$ and $V_L^{T_2}$
are dual to each other when $n$ is even. It follows immediately that $(V_L^{T_i})^{\pm}$ are self dual 
if $n$ is odd for $i=1,2$ and the dual of $(V_L^{T_1})^{\pm}$ is $(V_L^{T_2})^{\pm}$ if $n$ is even as 
$V_L^+$-modules.

 We first discuss the fusion product of $(\Mod{V_L^+})_{V_L}$ which is $\BZ_2$-graded such that
    $(\Mod{V_L^+})_{V_L}(1)=(\Mod{V_L^+})_{V_L}^0=\Mod{V_L}$ and $(\Mod{V_L^+})_{V_L}(\sigma_L)$ is the $\sigma_L$-twisted module category. Observe that $\Mod{V_L}$ is pointed, and  $\irr(\Mod{V_L})$ is a cyclic group of order $2(n+1)$  generated by $V_{L+\lambda_1}$. Since  $(\Mod{V_L^+})_{V_L}(\sigma_L)$ is an indecomposable $\Mod{V_L}$-module, we see that
$\irr(\Mod{V_L})$ acts on $\irr((\Mod{V_L^+})_{V_L}(\sigma_L))=\{V_L^{T_i}\mid i=1,2\}$ transitively. So 
    $V_{L+\lambda_r}\otimes V_{L}^{T_i}=V_L^{T_{i+r}}$ where $i+r$ is understood as 1 or 2 modulo $2$ and 
   $$V_L^{T_i}\otimes V_L^{T_i}=\sum_{r=0}^{n}V_{L+\lambda_{2r}}, \quad V_L^{T_1}\otimes V_L^{T_2}=\sum_{r=0}^{n}V_{L+\lambda_{2r+1}} \ {\rm if} \ n \ {\rm \ is \ odd}$$
   $$V_L^{T_i}\otimes V_L^{T_i}=\sum_{r=0}^{n}V_{L+\lambda_{2r+1}}, \quad V_L^{T_1}\otimes V_L^{T_2}=\sum_{r=0}^{n}V_{L+\lambda_{2r}} \ {\rm if} \ n \ {\rm \ is \ even}.$$
We also  have the following fusion rules  of $V_L^+$-modules for $r, r'=1,...,n,$ $i=1,2$ \cite{ADL2005}:
$$V_{L+\lambda_r}\otimes 
V_{L+\lambda_r}=V_{L+\lambda_{r}-\lambda_{r'}}+V_{L+\lambda_{r}+\ld_{r'}},
$$
where  $V_{L+\ld}=V_{L+\ld}^+ + V_{L+\ld}^-$ if $\ld=0$ or $\ld_{n+1}$,
$$V_L^{-}\otimes V_{L+\lambda_r}=V_{L+\lambda_r},\ V^{-}_L\otimes V_{L+\lambda_{n+1}}^{\pm}=V_{L+\lambda_{n+1}}^{\mp}, \ V_L^{-}\otimes (V_L^{T_i})^{\pm}=(V_L^{T_i})^{\mp},$$
$$V_{L+\lambda_{n+1}}^{\pm}\otimes V_{L+\lambda_r}=V_{L+\lambda_{n+1+r}},\
V_{L+\lambda_r}\otimes (V_L^{T_i})^{\pm}=V_L^{T_{i+r}}.$$
The fusion product of $V_{L+\lambda_{n+1}}^\e \o X$ for $X= V_{L+\lambda_{n+1}}^{\e}$ or $(V_{L}^{T_i})^{\e}$ can be described as follows: Here, $\e=\pm$ are respectively identified with $\pm 1$.  For any $\e, \e' \in \{\pm \}$,
 \begin{align*}
     V_{L+\lambda_{n+1}}^{\e'}&\otimes\, V_{L+\lambda_{n+1}}^{\e} =\, V_L^{\e'\e (-1)^{n+1}},\\
     V_{L+\lambda_{n+1}}^{\e'} &\otimes\, (V_{L}^{T_1})^{\e} \, = \, (V_{L}^{T_{2+n}})^{\e'\e}, \\
     V_{L+\lambda_{n+1}}^{\e'} &\otimes\, (V_{L}^{T_2})^{\e} \, = \, (V_{L}^{T_{1+n}})^{-\e'\e}\,.
 \end{align*}
The results above also hold for $V_K$ with $\lambda_r$ replaced by $\mu_r.$ 

Again, let $\CC=\Mod{V_L^+\otimes V_K^+}$ and $A=V_E.$ Then $\dim\CC=(8n+8)^2.$ Using the decomposition of $A$ in $\CC$, 
$$A=V_L\otimes V_K+ V_{L+\lambda_{n+1}}\otimes V_{K+\mu_{n+1}} \oplus  2\bigoplus_{r=1}^n V_{L+\lambda_r}\otimes V_{K+\mu_r},$$
where  $V_{K+\mu}=V_{K+\mu}^+ + V_{K+\mu}^-$ if $\mu=0$ or $\mu_{n+1}$, we see that 
$\dim A=\dim \CC_A=8n+8,$ $[A,A]_{\CC}=8+4n$ and $\iCA=8+4n.$ 
Again, $V_E^{D_{2(2n+2)}}=(V_L\otimes V_K)^+$ and  Lemma 7.2 applies. So $\CA$ contains a pointed fusion subcategory $\BB$ such that $\irr(\BB)= D_{2(2n+2)}$. We still use $\tau$ and $\s$ to denote the generators,  respectively order $2n+2$ and $2$, of $D_{2(2n+2)} \subset \irr(\CA)$. 

Since  $o(D_{2(2n+2)})=2(2n+2),$ there are four more simple objects of $\CC_A$ to be determined. We first determine $\alpha(x)$ for any $x\in \CC.$
 Then 
$$\alpha(V_L^+\otimes V_{K+\mu_r})=2\sum_{s=0}^{2n+1}V_{L+\lambda_s}\otimes V_{K+\mu_r+\mu_s}$$
for  $r=1,...,n.$ One can further show that $\alpha(V_L^+\otimes V_{K+\mu_r})=\tau^r+\tau^{-r}$ by using the explicit constructions of $\tau^{\pm r}$-twisted $V_E$-modules from \cite{DM1994}.
Similarly,
$$\alpha(V_L^+\otimes V_{K+\mu_{n+1}}^{\pm})=\sum_{s=0}^{2n+1}V_{L+\lambda_s}\otimes V_{K+\mu_{n+1}+\mu_s}=\tau^{n+1}$$
is simple  of dimension $1.$
Plus $A$ itself, we have $2n+2$ inequivalent simple objects in $\CA.$   In fact, this set $H$ of invertible objects is a cyclic group generated by $\tau.$
One can easily see that
$$\alpha(V_L^{\pm}\otimes (V_K^{T_i})^{\pm})=\sum_{r=0}^{n}V_{L+{\lambda_{2r}}}\otimes V_K^{T_i}+\sum_{r=0}^nV_{L+\lambda_{2r+1}}\otimes V_K^{T_{i+1}}$$
are two new inequivalent simple objects in $\CA$ of dimensions 
$\sqrt{n+1}$ for $i=1,2.$ Similarly, 
 $\alpha((V_L^{T_i})^{\pm}\otimes V_K^{\pm})$ give another two new inequivalent simple objects in $\CA$ of dimensions $\sqrt{n+1}.$ Last four objects  are  central elements in $K(\CA)=\overline{K(\CA)}.$

Since   $\CA$ contains a  pointed fusion subcategory whose irreducible objects form the dihedral group $D_{2(2n+2)}$, it remains to find another $2n+2$ invertible objects of order $2.$ One can compute that 
$$\alpha((V_L^{T_1})^{\pm}\otimes (V_{K}^{T_1})^{\pm})=(n+1)(V_L^{T_1}\otimes V_K^{T_1})+(n+1)
(V_L^{T_{2}}\otimes V_K^{T_{2}}),$$
$$\alpha((V_L^{T_1})^{\pm}\otimes (V_{K}^{T_2})^{\pm})=(n+1)(V_L^{T_1}\otimes V_K^{T_2})+(n+1)
(V_L^{T_{2}}\otimes V_K^{T_{1}}),$$
$$[\alpha((V_L^{T_i})^{\pm}\otimes (V_{K}^{T_j})^{\pm}), \alpha((V_L^{T_i})^{\pm}\otimes (V_{K}^{T_j})^{\pm})]_{\CC_A}=n+1,$$
$$[\alpha((V_L^{T_1})^{\pm}\otimes (V_{K}^{T_1})^{\pm}), \alpha((V_L^{T_1})^{\pm}\otimes (V_{K}^{T_2})^{\pm})]_{\CC_A}=0.$$
So both $\alpha((V_L^{T_1})^{+}\otimes (V_{K}^{T_1})^{+})$ and $\alpha((V_L^{T_1})^{+}\otimes (V_{K}^{T_2})^{+})$  are direct sum of $n+1$ inequivalent simple objects in $\CA.$  Together we have another 
$2(n+1)$ invertible elements in $D_{2(2n+2)}.$

For short we set 
$$X_i=\alpha(V_L^{+}\otimes (V_K^{T_i})^{+}),\quad Y_i= \alpha((V_L^{T_i})^{+}\otimes V_K^{+})$$
for $i=1,2.$ Then 
$$\iCA=D_{2(2n+2)}\cup\{X_i,Y_i\mid i=1,2\}$$
subject to the relations 
$$X_i^2=Y_i^2=\left\{\begin{array}{ll} \sum_{r=0}^{n}\tau^{2r} &  n\ {\rm odd}\\
\\
\sum_{r=0}^{n}\tau^{2r+1} &  n\ {\rm \ even},
\end{array}\right.
X_1X_2=Y_1Y_2=\left\{\begin{array}{ll}\sum_{r=0}^{n}\tau^{2r+1}  & n\ {\rm  odd}\\ 
\\
\sum_{r=0}^{n}\tau^{2r}   &  n\ {\rm  even,}
\end{array}\right.$$
 $$\tau X_1=X_2, \ \tau X_2=X_1,\ \tau Y_1=Y_2,\ \tau Y_2=Y_1,\ 
 \sigma X_i = Y_i, \ \sigma Y_i = X_i \text{ for } i=1,2,$$
 $$ X_1Y_1=Y_1X_1=X_2Y_2=Y_2X_2=\sum_{r=0}^{n}\tau^{2r}\sigma,$$
 $$X_1Y_2=Y_2X_1=X_2Y_1=Y_1X_2=\sum_{r=0}^{n}\tau^{2r+1}\sigma.$$

 If $n$ is odd, $\CC_A$ has ten nontrivial fusion subcategories which are not groups:
$$\FF_{X_i}=\{\tau^{2r}, X_i\mid r=0,...,n\}, \ \FF_{Y_i}=\{\tau^{2r}, Y_i\mid r=0,...,n\},$$
$$\FF_{X_1,X_2}=H\cup \{X_1,X_2\},\  \FF_{Y_1,Y_2}=H\cup \{Y_1,Y_2\},$$
 $$\FF_{X_1,Y_1}=D^1\cup \{X_1,Y_1\},\  \FF_{X_2,Y_2}=D^1\cup \{X_2,Y_2\},$$
  $$\FF_{X_1,Y_2}=D^2\cup \{X_1,Y_2\},\  \FF_{X_2,Y_1}=D^2\cup \{X_2,Y_1\},$$
where $H= \langle \tau \rangle$ and
 $$D^1= \langle \tau^2, \ \sigma \rangle, \ D^2= \langle\tau^2, \tau\sigma \rangle$$
 are subgroups of $D_{2(2n+2)}$ isomorphic to the dihedral group $D_{2n+2}.$
 One can verify that 
the corresponding subalgebras of $V_E$ are 
$$V_L\otimes V_K^++V_{L+\lambda_{n+1}}\otimes V_{K+\mu_{n+1}}^{\pm}, \ V_L^+\otimes V_K+V_{L+\lambda_{n+1}}^{\pm}\otimes V_{K+\mu_{n+1}},$$
$$V_L\otimes V_K^+,\ V_L^+\otimes V_K,\ V_L^+\otimes V_K^++V_{L+\lambda_{n+1}}^{\e}\otimes V_{K+\mu_{n+1}}^{\e'}, \ \e, \e' \in \{\pm\}.$$

If $n$ is even, $\CC_A$ has only two nontrivial fusion subcategories which are not pointed:
$$\FF_{X_1,X_2}=H\cup \{X_1,X_2\},\  \FF_{Y_1,Y_2}=H\cup \{Y_1,Y_2\}.$$
 One can verify that 
the corresponding subalgebras of $V_E$ are 
$$V_L\otimes V_K^+,\ V_L^+\otimes V_K.$$
We remark that in this case, $(V_{L+\lambda_{n+1}}^{\pm})'=V_{L+\lambda_{n+1}}^{\mp},$
and $(V_{K+\mu_{n+1}}^{\pm})'=V_{K+\mu_{n+1}}^{\mp}$ \cite[Prop. 3.7]{ADL2005}, where
$X'$ is the dual of $X$ in VOA setting. So the remaining eight subalgebras in the case when $n$ is odd are not subalgebras in the current situation.

The fusion categories that we obtained in this section are $\BZ_2$-extensions of the pointed category of $\vec^\w_{D_{2(l+1)}}$ where $D_{2(l+1)}$ is a dihedral group. This is a nonabelian generalization of Tambara-Yamagami categories in comparison with \cite{GR25}.

\subsection{Fusion actions associated to coset constructions} Let $A$ be a holomorphic vertex operator algebra and $U,V$ be rational $C_2$-cofinite subalgebras
such that $U^c=V$ and $V^c=U.$ Then 
$$A=\oplus_{i=0}^pU^i\otimes V^i$$
as $U\otimes V$-modules, where $\{U^i\mid i=0,...,p\}$ are the inequivalent irreducible $U$-modules and
$\{V^i\mid i=0,...,p\}$ are the inequivalent irreducible $V$-modules with $U^0=U$ and $V^0=V.$
Moreover, $\Mod{U}$ and $\Modrev{V}$ are equivalent (cf. Theorem 4.2 of \cite{DNR25}). Let $\CC=\Mod{U\otimes V}=\Mod{U}\boxtimes \Mod{V}.$ Then
$$\iCA=\{\alpha(U^i\otimes V)\mid i=0,...p\}$$
by Lemma 4.1 of \cite{DRX}. So $K(\CC_A)$ is a commutative algebra whose irreducible characters
are given by $\chi_i$ such that $\chi_i(\alpha(U^j\otimes V))=\frac{S^{U}_{i,j}}{\dim U^i}$ for all $j$ where $S^U=(S^U_{i,j})$ is the $S$-matrix of $\Mod{U}.$ It is clear that
$A^{\CC_A}=U\otimes V.$ 

Let $I=\{i\mid U^i \ {\rm 
 is \ invertible}\}.$ Then  $B=\oplus_{i\in I}U^i\otimes V^i$ is a subalgebra of $A$ and $(\CC_B^0)_A$ is the full subcategory of $\CC_A$ generated by $\alpha(U^j\otimes V)$ such that $S^U_{i,j}=\dim U^j$ for all $i\in I.$


\begin{rmk} {\rm 
    In Section 3.3 of \cite{Xu2014}, a list of all intermediate conformal nets for some examples of conformal inclusions were given by using Theorem 3.8 therein. These intermediate conformal nets can be recovered in the setting of VOAs by a similar approach with the currently developed Galois correspondence (Theorem \ref{t:Gal_Cor}). 
    The same comments also apply to  Theorem 3.14 of \cite{Xu2013}, where a list of all intermediate VOAs for a class of conformal inclusions were presented. } 
\end{rmk}
\section*{Appendix} The fusion rules and twists of the $\BZ_2$-orbifold VOA $V_L^+$ are obtained in the Subsection \ref{ss:Al}, where $L$ is the lattice $A_\ell$, and  the $S$-matrix of  the category $\Mod{V_L^+}$ of modules over $V_L^+$, simply denoted by $\CC_\ell$, can be computed via the formula \cite[(3.1.2)]{BakalovKirillov}
$$
S_{i,j}  = \theta_i^{-1}\theta_j^{-1} \sum_{k}N_{i^*,j}^k d_k \theta_k\,.
$$
The global dimension of $\CC_\ell$ is $4(\ell+1)$. Since modular data of these $\BZ_2$-orbifolds  are of significant interest,  we include them here for future reference.

\subsection*{Case $\ell = 2n$} We fix the order of simple objects of $\CC_\ell$ as follows:
$$
V_L, V_L^-, V_{L+\ld_1}, \dots, V_{L+\ld_n}, (V_L^T)^+,  (V_L^T)^-.
$$
The rank of $\CC_\ell$ is $n+4$, the twists $\theta_X = e^{2\pi i t_X}$ are given by
$$
\renewcommand{\arraystretch}{1.5}
\begin{array}{c|c|c|c|c|c}
\hline
  X   & V_L &V_L^-& V_{L+\ld_r}, r=1, \dots, n & (V_L^T)^+ &  (V_L^T)^- \\
  \hline
   t_X  &0 & 0 & \frac{r(\ell+1-r)}{2(\ell+1)}&\frac{\ell}{16}&\frac{\ell+8}{16}\\ 
   \hline
\end{array}\,,
$$
and the (unnormalized) $S$-matrix is
$$
\renewcommand{\arraystretch}{1.5}
\setlength{\arraycolsep}{4pt} 
\left[\begin{array}{cc|ccc|cc}
    1 & 1 & 2 & \cdots & 2 & \sqrt{\ell+1} & \sqrt{\ell+1} \\
    1 & 1 & 2 & \cdots & 2 & -\sqrt{\ell+1}  & -\sqrt{\ell+1} \\
    \hline
    2 & 2 & && &0  & 0 \\
    \vdots & \vdots & &{\displaystyle 4\cos\left(\frac{ij}{\ell+1}\right)},\,\, i, j=1, \dots n& & \vdots& \vdots \\
    2 & 2 & && &0  & 0 \\
    \hline
    \sqrt{\ell+1} & -\sqrt{\ell+1} &0 & \cdots & 0 &\sqrt{\ell+1} & -\sqrt{\ell+1}\\
    \sqrt{\ell+1} & -\sqrt{\ell+1} & 0&\cdots &0 &-\sqrt{\ell+1} & \sqrt{\ell+1}
\end{array}
\right]\,.
$$
\subsection*{Case $\ell = 2n+1$} The ordered list of simple objects of $\CC_\ell$ is
$$
V_L, V_L^-, V_{L+\ld_{n+1}}^+, V_{L+\ld_{n+1}}^- , V_{L+\ld_1}, \dots, V_{L+\ld_n}, (V_L^{T_1})^+,  (V_L^{T_2})^+, (V_L^{T_1})^-, (V_L^{T_2})^-.
$$
The rank of $\CC_\ell$ is $n+8$  and the twists $\theta_X = e^{2\pi i t_X}$ are given by
$$
\renewcommand{\arraystretch}{1.5}
\setlength{\arraycolsep}{4pt} 
\begin{array}{c|c|c|c|c|c|c|c|c|c}
\hline
  X   & V_L &V_L^-& V_{L+\ld_{n+1}}^+& V_{L+\ld_{n+1}}^- & V_{L+\ld_r} & (V_L^{T_1})^+ & (V_L^{T_2})^+ 
  & (V_L^{T_1})^- &  (V_L^{T_2})^- \\
  \hline
   t_X  &0 & 0 & \frac{\ell+1}{8} & \frac{\ell+1}{8} & \frac{r(\ell+1-r)}{\ell+1}&\frac{\ell}{16}& \frac{\ell}{16}& \frac{\ell+8}{16} & \frac{\ell+8}{16}\\ 
   \hline
\end{array}
$$
where $r=1, \dots, n.$  The $S$-matrix can be expressed as the $3\times 3$ block matrix 
$$
\renewcommand{\arraystretch}{1.5}
\left[\begin{array}{c|c|c}
    B_{1,1} & B_{1,2} & B_{1,3} \\
    \hline
    B_{1,2}^t  & B_{2,2} & 0 \\
    \hline
    B_{1,3}^t & 0 & B_{3,3}
\end{array}
\right],
$$
where $B_{i,j}^t$ denotes the transpose of $B_{i,j}$, and $B_{2,3}$ and $B_{3,2}$ are zero matrices. To describe $B_{i,j}$, we need the functions $\e, \eta, \g$ on $\BZ_8^\times$ which are defined as follows:  $\e(\ell)$ denotes the Jacobi symbol $\jacobi{-1}{\ell}$ of -1 modulo $\ell$,  $\eta(\ell) = 1$ if $\ell \equiv \pm 1 \mod 8$,  and $-1$ otherwise, and $\g(\ell)=\e(\ell)\eta(\ell)\sqrt{-\e(\ell)}$. The remaining $B_{i,j}$'s are then given by
$$
B_{1,1}=\left[\begin{array}{cccc}
1& 1& 1& 1 \\
1& 1& 1& 1 \\
1& 1& -\e(\ell)& -\e(\ell)\\ 
1& 1& -\e(\ell)& -\e(\ell)
\end{array}
\right], \, 
B_{1,2}=\left[\begin{array}{ccc}
2& \cdots & 2 \\
2& \cdots & 2 \\
-2& \cdots & (-1)^{n} 2 \\
-2& \cdots & (-1)^{n} 2 
\end{array}
\right] \text{ is an  $n \times 4$-matrix}, 
$$
$$
B_{1,3}=\sqrt{n+1}\left[\begin{array}{cccc}
1& 1& 1& 1 \\
-1& -1& -1& -1 \\
-\g(\ell)& \g(\ell) & - \g(\ell)&  \g(\ell)\\ 
 \g(\ell) & - \g(\ell) &  \g(\ell) & - \g(\ell)
\end{array}
\right], \, B_{2,2} = \left[4 \cos\left(\frac{ij}{\ell+1}\right)\right]_{i,j}
$$
where $i, j=1, \dots, n$. Finally, 

$$
B_{3,3} = \left\{\begin{array}{ll}
  \sqrt{n+1}\, M\left(\mtx{\zeta^{\eta(\ell)} & \zeta^{-\eta(\ell)} \\  \zeta^{-\eta(\ell)} & \zeta^{\eta(\ell)}} \right)   &  \text{ if } \ell \equiv 1 \text{ or } 5\!\mod 8;\\\\
   \sqrt{\ell+1}\, M\left(\mtx{0 & 1\\ 1& 0}\right) &  \text{ if } \ell \equiv 3\!\mod 8;\\\\
  \sqrt{\ell+1}\, M\left(\mtx{1& 0 \\ 0 & 1}\right) &  \text{ if } \ell \equiv 7\!\mod 8,
\end{array} \right.
$$
where $\zeta=e^{\frac{2\pi i}{8}}$ and  $M(X)$ denotes the $2 \times 2$-block matrix $\mtx{X& -X\\-X& X}$. 
\begin{rmk}{\rm
    The modular data of $\CC_\ell$, for $\ell=2,4$, are Galois conjugates of the second example of Section 5D (a) and (b) in \cite{GNN} respectively, and their fusion rings are shown to be realizable in \cite{LPR}. For $\ell=1, \dots, 10, 12, 14$, and $16$, $\CC_\ell$ realizes some potential modular data of rank $\le 12$  computationally discovered in \cite{NRWW} and \cite{NRW}. These realizations are listed as follows with the convention $r_{c,d}^{a,b}$  introduced in \cite{NRW} of a modular data (MD), where $r$ is the rank, $c$ is the additive central charge modulo 8, $d$ is the global dimension, $a$ is the order of the $T$-matrix and $b$ is the finger print of the MD:}
    $$
\renewcommand{\arraystretch}{1.5}
\begin{array}{c|c|c|c|c|c|c|c}
\hline
  \ell   & 1 & 2 & 3& 4& 5& 6&7\\
  \hline
   {\rm MD} & 8_{1,8.}^{16,123} & 5_{2,12.}^{24,940} & 9_{3,16.}^{16,696} & 6_{4,20.}^{20,180} & 10_{5,24.}^{48,125}& 7_{6,28.}^{56,193}&11_{7,32.}^{16,304}\\
   \hline
\end{array}
$$
$$
\renewcommand{\arraystretch}{1.5}
\begin{array}{c|c|c|c|c|c|c}
\hline
  \ell   &8&9&10&12&14&16\\
  \hline
   {\rm MD} &8_{0,36.}^{18,162}&12_{1,40.}^{80,190}&9_{2,44.}^{88,112}&10_{4,52.}^{52,489}&11_{6,60.}^{120,369}&12_{0,68.}^{34,116}\\
   \hline
\end{array}\,.
$$
\end{rmk}

\bibliographystyle{plain}
\bibliography{references}
\end{document}